\documentclass[12pt,a4paper,reqno]{amsart}
\usepackage[english]{babel}
\usepackage[applemac]{inputenc}
\usepackage[T1]{fontenc}
\usepackage{enumitem} 
\usepackage{palatino}
\usepackage{amsmath}
\usepackage{amssymb}
\usepackage{amsthm}
\usepackage{amsfonts}
\usepackage{float}
\usepackage{graphicx}

\usepackage{amsrefs}
\usepackage[colorlinks = true, citecolor = blue]{hyperref}
\title{The Dimension Spectrum of conformal Graph Directed Markov Systems}
\address{University of Connecticut, Department of Mathematics}
\address{University of North Texas, Department of Mathematics}
\author{Vasileios Chousionis}
\author{Dmitriy Leykekhman}
\author{Mariusz Urba\'nski}
\thanks{V.C. is supported by  the Simons Foundation via the project `Analysis and dynamics in Carnot groups', Collaboration grant no.\  521845. D.L.  is supported by the NSF grant no.\ DMS-1522555.   M.U. is supported by NSF grant no.\ DMS-1361677.}
\email{vasileios.chousionis@uconn.edu}
\email{dmitriy.leykekhman@uconn.edu}
\email{urbanski@unt.edu}

\newcommand{\ve}{\varepsilon}

\newcommand{\f}{\phi}
\newcommand{\R}{\mathbb{R}}

\newcommand{\N}{\mathbb{N}}

\newcommand{\C}{\mathbb{C}}
\newcommand{\Z}{\mathbb{Z}}

\renewcommand{\a}{{\alpha}}
\newcommand{\om}{{\omega}}

\newcommand{\cT}{\mathcal{T}}

\newcommand{\cS}{\mathcal{S}}

\newcommand{\cf}{\mathcal{CF}_{\C}}
\newcommand{\cfn}{\mathcal{CF}_{\N}}
\newcommand{\cfi}{\mathcal{CF}_{I}}

\newcommand{\spt}{\operatorname{spt}}

\newcommand{\cH}{\mathcal{H}}

\newcommand{\G}{\mathbb{G}}
\newcommand{\sg}{\sigma}

\newcommand{\1}{\mathbf{1}}
\newcommand{\up}{\upsilon}

\newcommand{\diam}{\operatorname{diam}}

\newcommand{\Rea}{\operatorname{Re}}
\newcommand{\Imm}{\operatorname{Im}}

\def\Int{\text{{\rm Int}}}

\newcommand{\fpli}{F_{\infty}^+}
\def\om{\omega}
\def\du{\bigoplus}
\newcommand{\Heis}{{{\mathbf{Heis}}}}

\def\fg{{\mathfrak g}}

\def\fv{{\mathfrak v}}

 \DeclareMathOperator{\Imag}{Im}
\newcommand{\stm}{\setminus}
\newcommand{\ra}{\rightarrow}

\numberwithin{equation}{section}

\newtheorem{thm}{Theorem}[section]

\newtheorem{lm}[thm]{Lemma}
\newtheorem{corollary}[thm]{Corollary}
\theoremstyle{definition}

\newtheorem{propo}[thm]{Proposition}
\theoremstyle{definition}

\newtheorem{defn}[thm]{Definition}
\theoremstyle{definition}

\theoremstyle{definition}
\newtheorem{rem}[thm]{Remark}
\newtheorem{remark}[thm]{Remark}

\newtheorem{conjecture}[thm]{Conjecture}

\begin{document}
\begin{abstract} In this paper we study the dimension spectrum of general conformal graph directed Markov systems modeled by countable state symbolic subshifts of finite type. We perform a comprehensive study of the dimension spectrum addressing questions regarding its size and topological structure. As a corollary we obtain that the dimension spectrum of infinite conformal iterated function systems is compact and perfect. On the way we revisit the role of the parameter $\theta$ in graph directed Markov systems and we show that new phenomena arise. 

We also establish topological pressure estimates for subsystems in the abstract setting of symbolic dynamics with countable alphabets. These estimates play a crucial role in our proofs regarding the dimension spectrum, and they allow us to study Hausdorff dimension asymptotics for subsystems. 

Finally we narrow our focus to the dimension spectrum of conformal iterated function systems and we prove, among other things, that the iterated function system resulting from the complex continued fractions algorithm has full dimension spectrum.  We thus give a positive answer to the Texan conjecture for complex continued fractions.
\end{abstract}

\maketitle
\tableofcontents
\section{Introduction}
Let $X$ be a compact metric space and let $\cS=\{\f_e:X \ra X\}_{e \in E}$ be a countable collection of uniformly contracting maps. Let $J_\cS$, also frequently denoted by $J_E$, be its {\em limit set}. This set is defined as the the image of a natural projection from the symbol space to $X$. If the alphabet $E$ is finite $J_E$ is well known to be a unique compact set, invariant with respect to $\cS$. The {\em dimension spectrum} of $\cS$:
$$
DS(\cS):=\{\dim_{\cH}(J_{A}): A \subset E\}
$$
is the set of all possible values for the Hausdorff dimension of the subsystems of $\cS$. If $\cS$ consists of finitely many maps then $DS(\cS)$ is a finite set. However when $\cS$ is infinite the structure of $DS(\cS)$ becomes much more complex and intriguing. In particular many interesting questions arise related to the  {\em size} and {\em topological properties} of the dimension spectrum.

As the title suggests we are going to investigate the dimension spectrum of limit sets in the general framework of conformal {\em graph directed Markov systems} (GDMS). We postpone the formal definition of a GDMS to Section \ref{sec:gdms}, and we now only provide a short heuristic description. A GDMS consists of a directed multigraph $(E,V)$ with a countable set of edges $E$ and a finite set of vertices $V$. Each vertex $v \in V$ corresponds to a compact set $X_v$ and  each edge $e \in E$ corresponds to a contracting map between two compact sets $X_v$. An incidence matrix $A:E \times E \ra \{0,1\}$ then essentially determines if a pair of these maps is allowed to be composed. 

Limit sets of infinite graph directed Markov systems form a very broad family of geometric objects, which include limit sets of Kleinian and complex hyperbolic Schottky groups, Apollonian circle packings, self-conformal and self-similar sets. The diversity of these examples justifies our decision to study the dimension spectrum in the unified framework of GDMS. Graph directed systems with a finite alphabet consisting of similarities were introduced by Mauldin and Williams in \cite{MW} and further studied by Edgar and Mauldin in \cite{edm}. Mauldin and the last named author developed a fully fledged theory of Euclidean conformal GDMS with a countable alphabet in \cite{MUbook} stemming from \cite{MU1}. In the recent monograph \cite{CTU}, Tyson together with the first and last named authors extended the theory of conformal GDMS  in the setting of nilpotent stratified Lie groups (Carnot groups) equipped with a sub-Riemannian metric. See also \cite{atnip, kes1, kes2, roy, polur} for recent advances on various aspects of GDMSs.

In \cite{MU1,MUbook,CTU}, and in other relevant works, thermodynamic formalism is heavily used in order to study the limit sets of conformal graph directed Markov systems.  In particular one needs to study the {\em topological pressure function} of the system, which will be defined in Section \ref{sec:gdms}. Under some natural assumptions, the zero of the pressure function corresponds to the Hausdorff dimension of the limit set. It is usually denoted by $h$ and it is called Bowen's parameter as it traces back to the the fundamental work of Rufus Bowen \cite{bowen}. Another parameter of crucial importance for a conformal GDMS is the parameter $\theta$, which is the finiteness threshold of the pressure function, see Definition \ref{thetavariants}. In particular, as we will discuss later, the parameter $\theta$ is related to the dimension spectrum of conformal iterated function systems (IFS). 

In Section \ref{sec:gdms}  we introduce several natural parameters for GDMS, which can be thought as variants of the parameter $\theta$. We provide concrete examples showing that these parameters are distinct and we investigate their relations. Moreover we clarify and correct several misconceptions from \cite{MUbook} related to the role of the parameter $\theta$. In connection to the dimension spectrum, it was proved in \cite{MU1} that if $\cS$ is an infinite conformal IFS satisfying the open set condition then $[0, \theta) \subset DS(\cS).$ Somehow surprisingly we prove that if $\cS$ is a GDMS the situation might be quite different, as there are two new parameters $\theta_3:=\theta_3(\cS)$ and $h_0=h_0(\cS)$, introduced respectively in Definition \ref{thetavariants} and Definition \ref{hzero}, which determine the interval contained in the dimension spectrum.
\begin{thm} 
\label{mainspectintro}Let $\cS=\{\f\}_{e \in E}$ be an infinite finitely irreducible conformal GDMS. For every $t \in (h_0, \theta_3)$ there exists some $F \subset E$ such that $\dim_{\cH} (J_F)=t.$
In other words $(h_0, \theta_3) \subset DS(\cS)$.
\end{thm}
Moreover  we prove, see Theorem \ref{sharpexample}, that the lower bound $h_0$ is sharp; that is there exists a GDMS $\cS=\{\f\}_{e \in E}$ such that  every subset $I \subset E$ with $\dim_{\cH}(J_I)>0$ satisfies 
$\dim_{\cH}(J_I) \geq h_0.$ 

We also investigate the topology of the dimension spectrum. In Definition \ref{dspec} we introduce new natural notions of spectra suited to GDMS. These new spectra can be thought as restricted versions of $DS(\cS)$ as they only take into account certain families of subsystems. In Theorems \ref{compactspectrum} and \ref{perfectspectrum} we prove that if $\cS=\{\f_e\}_{e \in E}$ is an infinite and finitely irreducible conformal GDMS, then $DS_{\tilde{\Lambda}}(\cS)$ (see Definition \ref{dspec}) is compact and perfect for any set $\Lambda$ witnessing finite irreducibility for $E$. As a corollary in the case of iterated function systems we obtain the following theorem.
\begin{thm}
 \label{topospectrumintro}
Let $\cS=\{\f_e\}_{e \in E}$ be an infinite conformal iterated function system satisfying the open set condition. Then $DS(\cS)$ is compact and perfect.
 \end{thm}
In the case of conformal IFS it turns out that $\theta=\theta_3$ and $h_0=0$, therefore as a corollary of Theorems \ref{mainspectintro} and \ref{topospectrumintro} we obtain that $[0,\theta] \in DS(\cS)$. We thus improve  the corresponding result from \cite{MU1}, by including $\theta$ in the dimension spectrum.
 
The proofs of Theorems \ref{mainspectintro}, \ref{compactspectrum} and \ref{perfectspectrum}  depend crucially on Propositions \ref{ghenciu1} and \ref{ghenciu2}, which are special cases of Propositions \ref{ghenciu1sym} and \ref{ghenciu2sym} respectively. Propositions \ref{ghenciu1sym} and \ref{ghenciu2sym} are of independent interest as they provide effective estimates for the topological pressure of subsystems in the abstract setting of symbolic dynamics with countable alphabets. These estimates turned out to be extremely useful for our purposes and we anticipate that can be further exploited. For example we employ Propositions \ref{ghenciu1sym} and \ref{ghenciu2sym} in Section \ref{sec:hein}, where we generalize earlier results from \cite{HU} to the setting of GDMSs.  In particular Theorem \ref{hein} provides estimates for the Hausdorff dimension of the limit set of any finitely irreducible and strongly regular conformal GDMS up to any desired accuracy. Moreover our proof substantially simplifies the proof from \cite{HU}. 
 
As mentioned earlier the theory of graph directed Markov systems has been recently extended to the sub-Riemannian setting of Carnot groups in \cite{CTU}. We record that the content of Sections \ref{sec:gdms}-\ref{sec:dimspect} is valid  for Euclidean as well as Carnot graph directed Markov systems. Nevertheless for simplicity of notation we avoid mentioning anything about the ambient space except in Remark \ref{carnot} where we comment further on the issue and  we give a very brief introduction to Carnot groups and sub-Riemannian conformal maps. 
 
In the last two sections we focus our attention to conformal iterated functions systems. The spectrum $DS(\cS)$ has an interesting and intriguing structure already for IFSs, even the ones composed of similarities. Indeed, apart from being compact, perfect, and containing the interval $[0,\theta)$ (keep in mind that we are now in the realm of IFSs), it may happen to be an interval (then necessarily $DS(\cS)=[0,\dim_{\cH}(J)]$), or it may have many non-degenerate connected components (so intervals) and connected components being singletons. Such examples, even for similarities, can be found in \cite{KZ}. All of this leads us to the following conjecture.
\begin{conjecture}
For every compact and perfect subset $K$ of $[0,+\infty)$ there exists a conformal (we even conjecture one composed of similarities) IFS $\cS$ such that $DS(\cS)=K$. 
\end{conjecture}
This conjecture has room for many partial results. For example, does there exist an IFS whose dimension spectrum  has a given prescribed finite number of connected components, or does there exist an IFS whose dimension spectrum $DS$ is not uniformly perfect?

Of special significance is the question of when an IFS $\cS$ has full dimension spectrum, that is $DS(\cS)=[0,\dim_{\cH}(J)]$. Kesseb\"ohmer and Zhu  proved in \cite{KZ} that the spectrum is full for the IFS resulting from the real continued fractions algorithm, resolving the so-called Texan conjecture. Recall that any irrational number in $[0,1]$ can be represented as a continued fraction
$$\cfrac{1}{e_1+\cfrac{1}{e_2+\cfrac{1}{e_3+\dots}}},$$
where $e_i \in \N$ for all $i \in \N$. We refer to the book \cite{henbook} of Hensley for an excellent exposition of continued fractions and their connections to number theory, complex analysis, ergodic theory, dynamic processes, analysis of algorithms, and even theoretical physics. We remark that continued fractions from the perspective of dynamical systems have been studied extensively, see e.g. \cite{good, cu, hen1, daj, gmr, nuss, nusscf, MU2, KZ} or the review \cite{henrev}. It is remarkable that the representation by continued fractions can be described by the infinite conformal  IFS
$$\cfn:= \{\f_{n}:[0,1] \ra [0,1]: \f_n(x)=\frac{1}{n+x} \mbox{ for }i \in \N \}.$$
Therefore according the Texan conjecture, resolved positively in \cite{KZ}, it holds that $DS(\cS_{\cfn})=[0,1]$. That is for any $t \in [0,1]$ there exists some $I \subset \N$ such that $\dim_{\cH}(J_{\cfi})=t$ and $J_{\cfi}$  corresponds to the set of irrational numbers whose continued fraction expansion only contains natural numbers from $I$. 

In Section \ref{sec:fullcf} we investigate the dimension spectrum of the IFS resulting from the {\em complex continued fractions algorithm}. A  complex continued fraction algorithm is an algorithm that provides approximations by ratios of Gaussian integers to a given complex number. The origins of complex continued fractions can be traced back to the works of the brothers Adolf Hurwitz \cite{ahur} and Julius Hurwitz \cite{jhur}. Since their pioneering contributions, complex continued fractions  have been studied widely from different viewpoints, indicatively we mention the breakthrough work of A. L. Schmidt in the 70s \cite{schm}. More information can be found in \cite[Chapter 5]{henbook}.

As in the case of real continued fractions, complex continued fractions can be represented via the infinite conformal IFS $$\cf=\{\f_e: \bar{B}(1/2,1/2) \ra \bar{B}(1/2,1/2)\}$$ where $$E=\{m+ni:(m,n) \in \N \times \Z \} \mbox{ and }\f_e(z)=\frac{1}{e+z}.$$
This system was considered in detail in \cite{MU1}, nevertheless several questions regarding the intriguing geometric structure of $J_{\cf}$ remain open. Currently a very active research topic at the interface of dynamical systems and numerical analysis, is to obtain estimates of high accuracy for the Hausdorff dimension of $J_{\cf}$. In \cite{MU1} it was proved that $\dim_{\cH}(J_{\cf}) \leq 1.885$. Priyadashi \cite{pri} established the lower bound $\dim_{\cH}(J_{\cf}) \geq 1.825$. Recently Falk and Nussbaum \cite{nuss, nusscf}  developed a new method in order to obtain very effective estimates for $\dim_{\cH}(J_{\cf})$. In particular their method indicates that $1.85574 \leq \dim_{\cH}(J_{\cf}) \leq 1.85589$, although as the authors mention, see \cite[Remark 3.2]{nuss}, some interval arithmetic is required in order to make the last estimate rigorous. 

In Section \ref{sec:fullcf} we study the dimension spectrum $DS (\cf)$ and we settle the Texan conjecture for complex continued fractions. Our main result there reads as follows.
\begin{figure}
\centering
\includegraphics[scale = 0.45]{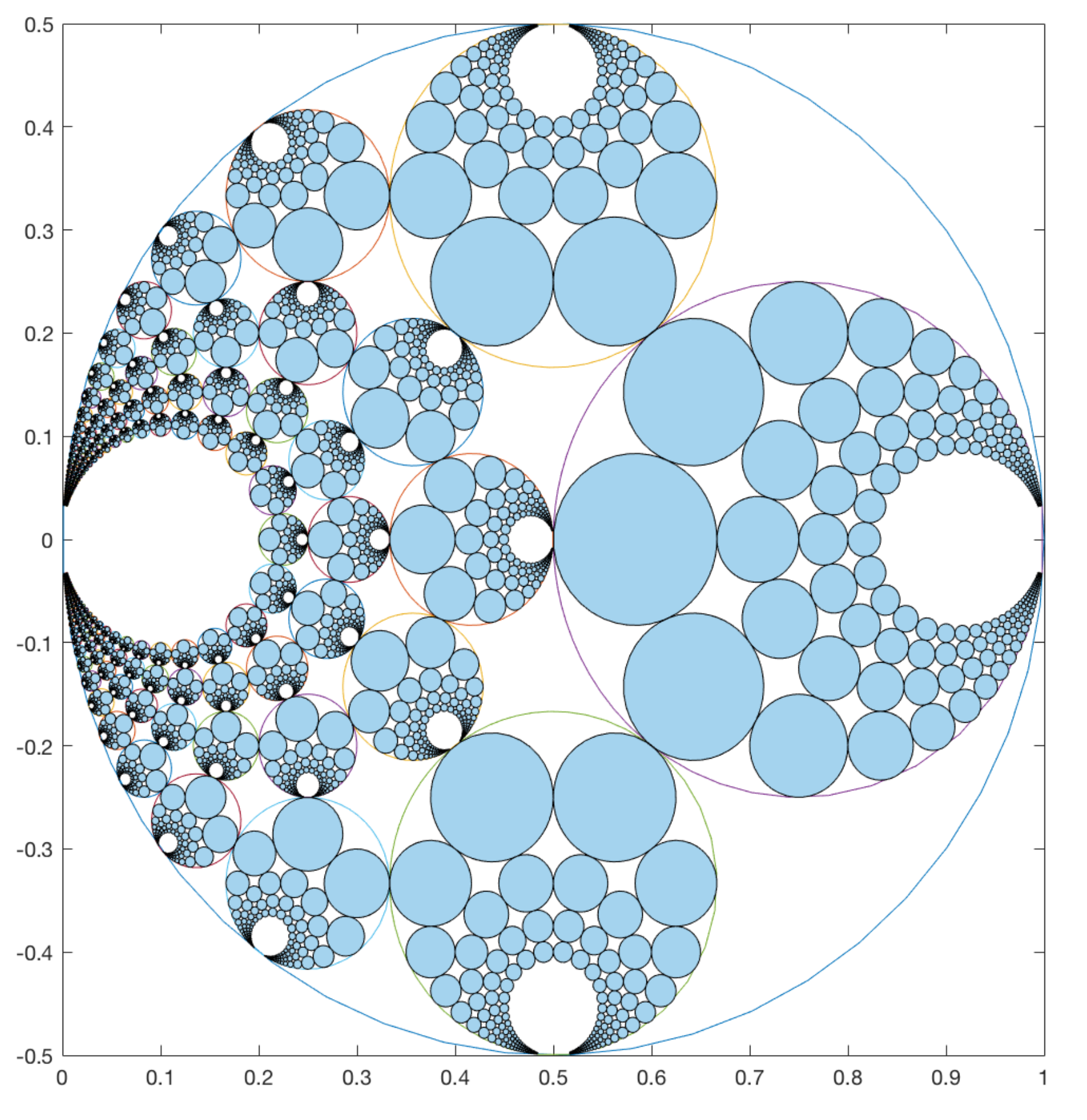}
\caption{An approximation of the limit set of the complex continued fractions IFS after two iterations.}
\label{figccf}
\end{figure}
\begin{thm}
\label{fullspec} The conformal iterated function system associated to the complex continued fractions has full spectrum; that is $$DS (\cf)=[0, \dim_{\cH}(J_{\cf})].$$
\end{thm}
Our proof strategy is inspired by the one of Kesseb\"ohmer and Zhu from \cite{KZ}. In particular we develop further on \cite[Theorem 2.2]{KZ}, which we restate in Theorem \ref{t3fs3}, and derive several crucial consequences, 
see e.g. Corollary \ref{keypropoint}. Among several key new ideas in our proof is the introduction of a natural order on the grid $E$, see Definition \ref{fullspec}, as well as a bootstrapping argument involving the dimension spectrum of certain subsystems of $\cf$. We also record that our proof is technically more involved than the one for real continued fractions, demands subtler estimates, and, as another new feature, it is also heavily computer assisted. For example we use numerics in order to obtain rigorous estimates for the Hausdorff dimension of certain subsystems of $\cf$ which play important role in the proof of Theorem \ref{fullspec}.  This is a rather interesting novelty because it shows that estimates of Hausdorff dimension of limit sets using numerical analysis, as in \cite{pri,nuss, nusscf,jp1,jp2}, can be employed in order to obtain theoretical results such as Theorem~\ref{fullspec}.


The paper is organized as follows. In Section \ref{sec:preses} we lay down the necessary background from symbolic dynamics and we prove various estimates for the topological pressure of subsystems. In Section \ref{sec:gdms} we introduce all the relevant concepts related to graph directed Markov systems and we introduce and study new natural parameters which can be realized as variants of the parameter $\theta$. In Section \ref{sec:gdmsspec} we introduce new dimension spectra for GDMS and study their size and topological properties. In Section \ref{sec:hein}  we provide an effective tool for calculating the Hausdorff dimension of the limit set of any finitely irreducible and strongly regular conformal GDMS with arbitrarily high accuracy. We thus generalize the main result of \cite{HU} to the setting of GDMSs and we simultaneously provide a substantially simpler proof.  In Section \ref{sec:dimspect} we narrow our focus to the dimension spectrum of general conformal iterated function systems. The machinery developed in Section \ref{sec:dimspect} is used, among other tools, in Section \ref{sec:fullcf}  to prove that the dimension spectrum of complex continued fractions is full.

\section{Pressure estimates for countable alphabet subsets of finite type}\label{sec:preses}
In this section we introduce all the relevant concepts from symbolic dynamics with countable alphabets and we establish qualitative bounds for the topological pressure of H\"older functions.
 
Let $\mathbb{N}=\{1, 2, \ldots \}$  be  the set of all positive
integers  and let $E$ be  a  countable  set, either finite or infinite,
which we will call {\it alphabet}.  Let
$$
\sg: E^\mathbb{N} \ra E^\mathbb{N}
$$
be the  shift map,  which  is given by the
formula
$$
\sg\left( (\om_n)^\infty_{n=1}  \right) =  \left( (\om_{n+1})^\infty_{n=1}  \right).
$$
Note that the shift map simply discards the first coordinate. We now proceed with some standard notation from symbolic dynamics.  For every finite word $\om \in E^*:=\cup_{n=0}^\infty E^n$,  $|\om|$ will denote its {\em length}, that is the unique integer $n \geq 0$ such that $\om \in E^n$. As a standard convention we let $E^0=\{\emptyset\}$. If $\om, \upsilon \in
E^\mathbb{N}$, $\tau \in E^*$ and $n \geq 1$, we set
$$
\aligned
\om |_n&=\om_1\ldots \om_n\in E^n, \\
\tau\om&=(\tau_1,\dots,\tau_{|\tau|},\om_1,\dots),\\
\omega\wedge\upsilon&=\mbox{longest
initial block common to both $\om$ and $\tau$}.
\endaligned
$$
Note that $\omega\wedge\upsilon  \in
E^{\mathbb N}\cup E^*$.

Let the matrix $A:E \times E \to \{0,1\}$ and set
$$
E^\mathbb{N}_A
:=\{\om \in E^\mathbb{N}:  \,\, A_{\om_i\om_{i+1}}=1  \,\, \mbox{for
  all}\,\,   i \in \N
\}.
$$
The words in $E^\mathbb{N}_A$ will be called {\it $A$-admissible}. In the same way the set of {\em finite} admissible words is defined as
$$
E^n_A
:=\{w \in E^\mathbb{N}:  \,\, A_{\om_i\om_{i+1}}=1  \,\, \mbox{for
  all}\,\,  1\leq i \leq
n-1\}, \qquad n \in \N,
$$
and
$$
E^*_A:=\bigcup_{n=0}^\infty E^n_A.
$$
For every  $\om \in E^*_A$, its corresponding {\em cylinder}\index{cylinder} is
$$
[\om]:=\{\tau \in E^\mathbb{N}_A:\,\, \tau_{|_{|\om|}}=\om \}.
$$

For  $\alpha > 0$, we consider the metrics $d_\alpha$ on
$E^{\mathbb N}$ by setting
\begin{equation*}
d_\alpha(\om,\tau) ={\rm e}^{-\alpha|\om\wedge\tau|}.
\end{equation*}
It follows easily that  the  metrics $d_\alpha$ induce the same topology.  
If no metric is specifically mentioned, $E^\N$ will be treated as a topological space with the topology defined by  $d_1$.  Note that  $E^\mathbb{N}_A$ is  a closed subset of
$E^\mathbb{N}$, invariant under the shift map $\sg: E^\mathbb{N}\ra E^\mathbb{N}$.

The matrix $A:E \times E \to \{0,1\}$  is called {\it irreducible} 
if there
exists $\Lambda \subset E_A^*$ such that for all $i,j\in E$
there exists $\om\in \Lambda$ for which $i\om j\in E_A^*$. If there exists a finite set $\Lambda$ with the previous property, the matrix $A$ will be called {\it finitely irreducible}. \index{finitely irreducible matrix}
\index{matrix!finitely irreducible}
If in addition there exists a finite set $\Lambda \subset E_A^*$ consisting of words of the same lengths such that for all $i,j \in E$ there exists $\om \in \Lambda$ such that $i\om j \in E_A^\ast$, then the matrix $A$ is called  {\it finitely primitive}. 

Given a set $F\subset E$ we  put
$$
F^{\mathbb N}
:=\{\om \in E^{\mathbb N}: \, \om_i\in F
\, \mbox{for all } \,  i \in \N\},
$$
and
$$
F_A^{n}:=E^{n}_A \cap F^n
=\{\om \in F^{n}:  \,\, A_{\om_i\om_{i+1}}=1  \,\, \mbox{for
  all}\,\,  1\leq i \leq
n-1\},
$$
where $A:E \times E \to \{0,1\}$ is a matrix. Slightly abusing notation, the set $F \subset E$ will be called {\em irreducible} (with respect to the matrix $A$) if there exists a set $\Lambda \subset F_A^\ast$ such that for all $i,j\in F$
there exists $\om\in \Lambda$ for which $i\om j\in F_A^*$. If $\Lambda \subset F_A^*$ is a set witnessing irreducibilty for $F$ we will denote by $\tilde{\Lambda}$ the set of all letters from $F$ appearing in the words of $\Lambda$. Formally
\begin{equation}
\label{lambdatilde}
\tilde{\Lambda}=\{e \in F: \om_i=e \mbox{ for some } \om \in \Lambda, i=1, \dots, |\om|\}.
\end{equation}
We stress that from now on, until otherwise noted, $A$ will denote a fixed finitely irreducible matrix.


Given a function $f:F_A^\mathbb{N}\to
\mathbb{R}$, the {\it n-th partition function} with respect to $F$ and $f$ is defined as
$$
Z_n(F,f)=\sum_{\om \in F_A^n}\exp \left(\sup_{\tau\in [\om]_F}\sum_{j=0}^{n-1}f(\sg^j(\tau))\right),
$$
where $[\om]_F=\{\tau\in F_A^\mathbb{N}:\tau|_{|\om|}=\om\}$. 

The  following lemma is crucial for the definition of topological pressure which will follow shortly. For its proof see e.g. \cite[Lemma 6.3]{CTU}
\begin{lm}\label{subadd2}
The sequence $(\log Z_n(F,f))_{n=1}^\infty$ is subadditive, i.e.
$$\log Z_{m+n}(F,f) \leq \log Z_n(F,f)+\log Z_m(F,f)$$
for all $m,n \in \N.$
\end{lm}
It is well known that if $(a_n)_{n=1}^\infty$ is subadditive sequence, then $\lim_{n\to \infty}\frac{a_n}{n}$
exists and is equal to $\inf_{n\ge 1}(a_n/n)$. Therefore the following definition makes sense.
\begin{defn}
\label{toppressym} Let $F \subset E$ and a function $f:F_A^\mathbb{N}\to
\mathbb{R}$. The {\em topological pressure} of $f$ with respect to the shift map
\index{shift map} $\sg:F_A^\mathbb{N}\to F_A^\mathbb{N}$ is defined to be
\begin{equation*}
P_F^\sg(f):=\lim_{n\to\infty} \frac{1}{n} \log Z_n(F,f)
       = \inf\left\{ \frac{1}{n} \log Z_n(F,f)\right\}.
\end{equation*}
If $F=E$ we suppress the subscript $F$ and write simply $P^\sg(f)$  for $P_E^\sg(f)$ and $Z_n(f)$ for $Z_n(E,f)$.
\end{defn}
Topological pressure is a key concept in symbolic dynamics with countable alphabet. For a concise but rigorous exposition see \cite[Chapter 6]{CTU}; a more extensive treatment can be found in \cite{MUbook}. 

Our goal in this section is to obtain estimates of $P^\sg_{F \cup \{a\}}(f)$ in terms of $P^{\sg}_{F}$ and $\sup_{\om \in [a]} f(\om)$ in the case when both $E$ and $F$ are finitely irreducible.
We start by stating an important distortion lemma for H\"older continuous functions. A function $f:E_A^\N\to\R$ is  H\"older continuous
with exponent $\a >0$ if
$$
V_\a(f):=\sup_{n\ge 1}\{V_{\a,n}(f)\}<\infty,
$$
where
$$
V_{\a,n}(f)=\sup\{|f(\om)-f(\tau)|e^{\a(n-1)}:\om,\tau\in E_A^\N
\text{ and } |\om\wedge \tau|\ge n\}.
$$
For functions  $f:E_A^\N\to\R$ we will also use the standard notation 
$$S_n f =\sum_{j=0}^{n-1} f \circ \sigma^{j}, \, n\in \N.$$
The proof of the following lemma can be found in \cite[Lemma 6.25]{CTU}. 
\begin{lm}
\label{distsymn}
Let $E$ be a finitely irreducible set. If $g:E_A^\N\to\C$ such that $V_\a(g)<\infty$, then for all $n\ge 1$, for all
$\om,\tau\in E_A^\N$, and all $\rho\in
E_A^n$ with $A_{\rho_n\om_1}=A_{\rho_n\tau_1}=1$ we have
$$
\big|S_ng(\rho\om)-S_ng(\rho\tau)\big|\le \frac{V_\a(g)}{e^\a-1}d_\a(\om,\tau).
$$
\end{lm}
For any $g: E_A^\N \ra \R$ we define
$$B(g):=\left\{ M>0: |S_ng(\rho\om)-S_ng(\rho\tau)| \leq M, \mbox{ for all } n\in \N, \,
\om,\tau\in E_A^\N,\,\rho\in E_A^n \right\}$$
and we set
\begin{equation}
\label{L(g)}
L(g):=\inf B(g).
\end{equation}
Note that if that if $g$ H\"older then Lemma \ref{distsymn}  implies that $L(g)<\infty$.

We now proceed to a key technical lemma of combinatorial flavor which will allow us to obtain the desired pressure estimates. For any finite set $\Lambda \subset E^\ast$ and any $f: E_A^\N \ra \R$ we denote
$$p_\Lambda:=\max_{\lambda \in \Lambda}  |\lambda|$$
and
\begin{equation}
\label{klf}
\kappa_{\Lambda}(f)=\min_{\lambda \in \Lambda}\left\{ \inf_{[\lambda]} e^{S_{|\lambda|}f}\right\}.
\end{equation}

\begin{lm}
\label{mainlemmasym} Let $E$ be a finitely irreducible infinite countable set. Let $F\subset E$ be a finitely irreducible subset of $E$ and let $\Lambda$ be a nonempty set witnessing finite irreducibility for $F$. Then for every $a \in E$ and every $\a$-H\"older function $f: E_A^\N \ra \R$,
$$Z_{n}(F \cup \{a\},f) \leq e^{2 L(f)}\sum_{j=0}^n {n \choose j}\, \left(\frac{\sharp \Lambda \, \exp(\sup f|_{[a]}+L(f))}{ \kappa_{\Lambda}(f)} \right)^{n-j}\, \sum_{k=n}^{j+p_{\Lambda}(n-j)} Z_{k}(F,f),$$
for every $n \in \N$.
\end{lm}

\begin{proof}
We set
\begin{equation}
\label{fjn}
F^n_j=\{\om \in  (F \cup \{a\})_A^n: \mbox{ the letter $e$ appears }n-j\mbox{ times in }\om\}.
\end{equation}
Therefore 
\begin{equation}
\label{zkl}
Z_{n}(F \cup \{a\},f) \leq \sum_{j=0}^n \sum_{\om \in F^n_j} \exp( \sup S_n f |_{[\om]}).
\end{equation}
Observe that if $\om \in F^n_j$ then it is of the form
\begin{equation}
\label{omega}
\om=\rho_1 a \rho_2 \dots\rho_{n-j} a \rho_{n-j+1},
\end{equation}
where $\rho_i \in F_A^\ast \cup \{\emptyset\}$ for $i=1, \dots, n-j+1$. Since $\Lambda \subset F_A^\ast$ is finite,
$$\Lambda=\{\lambda_1,\dots, \lambda_{\sharp \Lambda}\}.$$ For any two words $\om, \upsilon \in F^{\ast}_A$ we define
$$[\om,\upsilon]=\min\{i\in\{1,\dots,\sharp \Lambda\}:\om\lambda_i\upsilon \in F_A^\ast\}.$$
We then define a map
$$g: (F \cup \{a\})_A^{\ast} \ra F_A^\ast$$
 as follows.  If $\om \in(F \cup \{a\})_A^{\ast}$ then there exists some $n,j \in \N$ such that $\om \in F^n_j$. In that case, as we discussed earlier, $\om$ is as in \eqref{omega}. We then let
$$g(\om):=\bar{\om}:=\rho_1 \alpha_1 \rho_2 \dots\rho_{n-j} \alpha_{n-j} \rho_{n-j+1}$$
where $\alpha_i=\lambda_{[\rho_i,\rho_{i+1}]}$ for $i=1,\dots,n-j$.

We now make the following key observation. The map $g: F^n_{j} \ra F_A^\ast$ is at most 
\begin{equation}
\label{key}
(\sharp \Lambda)^{n-j} {n \choose j}\mbox{ to }1.
\end{equation}
To prove our claim we fix some word $\bar{\om} \in g(F^n_j)$ and we set
$$F^n_{j}(k_1,\dots,k_{n-j})=\{\om \in F^n_{j}:\mbox{ the letter $a$ appears on the spots }k_1,\dots,k_{n-j}\},$$
where $1\leq k_i\leq n$. We will now establish an upper bound for the number of elements in the set
$$A^n_j(k_1,\dots,k_{n-j}):=(g|_{F_j^n})^{-1}(\bar{\om}) \cap F^n_{j}(k_1,\dots,k_{n-j}).$$
Let $\om,\om' \in A^n_j(k_1,\dots,k_{n-j}), \om \neq \om'$,
$$\om=\rho_1a\rho_2a \dots \rho_{n-j}a\rho_{n-j+1}$$
and
$$\om'=\rho'_1a\rho'_2a \dots \rho'_{n-j}a\rho'_{n-j+1}.$$
Observe that for all $i=1,\dots,n-j$
\begin{equation}
\label{rosamelength}
|\rho_i|=|\rho'_i|
\end{equation}
because $|\om|=|\om'|$ and $a$ occupies exactly the same spots in $\om$ and $\om'$. Moreover
since $\om \neq \om'$ there exists some $i_0 \in \{1,\dots,n-j+1\}$ such that 
\begin{equation}
\label{roneq}
\rho_{i_0} \neq \rho'_{i_0}.
\end{equation}
Observe that necessarily
\begin{equation}
\label{lambdavektor}
(\lambda_{[\rho_1,\rho_2]},\lambda_{[\rho_2,\rho_1]},\dots, \lambda_{[\rho_{n-j},\rho_{n-j+1}]})\neq (\lambda_{[\rho'_1,\rho'_2]},\lambda_{[\rho'_2,\rho'_1]},\dots, \lambda_{[\rho'_{n-j},\rho'_{n-j+1}]}).
\end{equation}
Because otherwise we would get that
\begin{align*}
g(\om)&=\rho_1\lambda_{[\rho_1,\rho_2]}\rho_2 \dots \rho_{n-j}\lambda_{[\rho_{n-j},\rho_{n-j+1}]}\rho_{n-j+1}\\
&=\rho'_1\lambda_{[\rho_1,\rho_2]}\rho'_2 \dots \rho_{n-j}\lambda_{[\rho_{n-j},\rho_{n-j+1}]}\rho'_{n-j+1}=g(\om'),
\end{align*}
which  combined with \eqref{rosamelength} implies that $\rho_i=\rho'_i$ for all $i=1,\dots,n-j$. But this is impossible by \eqref{roneq}. Therefore \eqref{lambdavektor} holds and it implies that
\begin{equation}
\label{keyest}
\sharp A^n_j(k_1,\dots,k_{n-j}) \leq (\sharp \Lambda)^{n-j}.
\end{equation}
Note that 
$$F^n_{j}=\bigcup_{\substack{(k_1,\dots,k_{n-j}):\\1\leq k_1<k_2<\dots<k_{n-j}\leq n }} A^n_j(k_1,\dots,k_{n-j}),$$
and there exist at most ${n \choose j}$ sets $A^n_j(k_1,\dots,k_{n-j})$. Hence \eqref{key} follows by \eqref{keyest}.

For simplicity we let $\kappa:=\kappa_{\Lambda}(f)$ and $L=L(f)$. The next step in the proof of Lemma \ref{mainlemma} is to show that for every $\om \in F^n_j$,
\begin{equation}
\label{claim2}
\sup S_n f|_{[\om]} \leq (n-j)(\sup f|_{[a]}+L-\log \kappa)+\sup S_{|\bar{\om}|}f|_{[\bar{\om}]_F}+2L.
\end{equation}

Let $\om=\rho_1 a \rho_2 \dots \rho_{n-j} a \rho_{n-j+1}$ where as we observed earlier some of the $\rho_i$'s might be empty words. If $\rho$ is an empty word we make the convention that $\sup S_{|\rho|}f|_{[\rho]}=0$. We will first prove that
for every $\om \in F^n_j$,
\begin{equation}
\label{c21}
\sup S_n f |_{[\om]} \leq (n-j)\sup f|_{[a]}+\sum_{i=1}^{n-j+1} \sup S_{|\rho_i|}f|_{[\rho_i]}.
\end{equation}
To see \eqref{c21}, take $\tau \in [\om]$. Then
\begin{equation*}
\begin{split}\sum_{l=0}^{n-1} f(\sg^l(\tau))=&\sum_{l=0}^{|\rho_1|-1} f \circ \sg^l(\tau)+f(\sg^{|\rho_1|}(\tau))+\sum_{l=|\rho_1|+1}^{|\rho_1|+|\rho_2|} f \circ \sg^l(\tau) \\
&+f(\sg^{|\rho_1|+|\rho_2|+1}(\tau))+\dots+\sum_{l=|\rho_1|+\dots+|\rho_{n-j}|+(n-j)}^{|\rho_1|+\dots+|\rho_{n-j}|+|\rho_{n-j+1}|+(n-j)-1} f \circ \sg^l(\om).
\end{split}
\end{equation*}
Since $\om=\rho_1 a \rho_2 \dots \rho_{n-j} a \rho_{n-j+1}$ we deduce that
$$\sum_{l=0}^{n-1} f(\sg^j(\tau)) \leq (n-j)\sup f|_{[a]}+\sum_{i=1}^{n-j+1} \sup S_{|\rho_i|}f|_{[\rho_i]},$$
and \eqref{c21} follows.

Recall that
$$g(\om)=\bar{\om}=\rho_1\lambda_{[\rho_1,\rho_2]}\rho_2 \dots \rho_{n-j}\lambda_{[\rho_{n-j},\rho_{n-j+1}]}\rho_{n-j+1}.$$ We will now prove that 
\begin{equation}
\begin{split}
\label{c212}
\sum_{i=1}^{n-j+1} \sup S_{|\rho_i|}f|_{[\rho_i]} \leq \sup S_{|\bar{\om}|}f|_{[\bar{\om}]}+(n-j)(L-\log \kappa)+L.
\end{split}
\end{equation}
Take $\upsilon \in [\bar{\om}]$. Then,
\begin{equation*}
\begin{split}\sum_{l=0}^{|\om|-1} f(\sg^l(\up))=&\sum_{l=0}^{|\rho_1|-1} f \circ \sg^l(\up)+\sum_{l=|\rho_1|}^{|\rho_1|+|\lambda_{[\rho_1,\rho_2]}|-1} f \circ \sg^l(\up)+\sum_{l=|\rho_1|+|\lambda_{[\rho_1,\rho_2]}|}^{|\rho_1|+|\lambda_{[\rho_1,\rho_2]}|+|\rho_2|-1} f \circ \sg^l(\up) \\
&+\sum_{l=|\rho_1|+|\lambda_{[\rho_1,\rho_2]}|+|\rho_2|}^{|\rho_1|+|\lambda_{[\rho_1,\rho_2]}|+|\rho_2|+\lambda_{[\rho_2,\rho_3]}-1} f \circ \sg^l(\up) \\
&\quad+\dots+\sum_{l=|\rho_1|+|\lambda_{[\rho_1,\rho_2]}|+\dots+|\rho_{n-j}| +|\lambda_{[\rho_{n-j},\rho_{n-j+1}]}|}^{|\rho_1|+|\lambda_{[\rho_1,\rho_2]}|+\dots+|\rho_{n-j}| +|\lambda_{[\rho_{n-j},\rho_{n-j+1}]}|+|\rho_{n-j+1}|-1} f \circ \sg^l(\up). 
\end{split}
\end{equation*}
Hence 
\begin{equation}
\begin{split}
\label{correstsym1}\sum_{l=0}^{n-1} f(\sg^l(\up))&= S_{|\rho_1|}f(\up)+ S_{|\lambda_{[\rho_1,\rho_2]}|}f(\sg^{|\rho_1|}(\up))\\
&\quad+S_{|\rho_2|}f(\sg^{|\rho_1|+|\lambda_{[\rho_1,\rho_2]}|}(\up))+S_{|\lambda_{[\rho_2,\rho_3]}|}f(\sg^{|\rho_1|+|\lambda_{[\rho_1,\rho_2]}|+|\rho_2|} (\up)) \\
&\qquad+\dots+S_{|\rho_{n-j+1}|}f(\sg^{|\rho_1|+|\lambda_{[\rho_1,\rho_2]}|+\dots+|\rho_{n-j}| +|\lambda_{[\rho_{n-j},\rho_{n-j+1}]}|}(\up)).
\end{split}
\end{equation}
By the definition of $\kappa$ we deduce that
\begin{equation}
\label{correstsym2}
S_{|\lambda_{[\rho_i,\rho_{i+1}]}|}f(\sg^{|\rho_1|+\dots |\lambda_{[\rho_i-1,\rho_i]}|+|\rho_i|} (\up)) \geq \log \kappa
\end{equation}
for all $i=1,\dots,n-j$. By Lemma \ref{distsymn} we also deduce that
\begin{equation}
\label{correstsym3}
S_{|\rho_{i}|}f(\sg^{|\rho_1|+\dots+|\rho_{i-1}| +|\lambda_{[\rho_{i-1},\rho_{i}]}|}(\up)) \geq \sup S_{|\rho_i|}f|_{[\rho_i]}-L
\end{equation}
for all $i=1,\dots,n-j+1$. Therefore by \eqref{correstsym1}, \eqref{correstsym2} and \eqref{correstsym3}, we deduce that
$$\sup S_{|\bar{\om}|}f|_{[\bar{\om}]} \geq \sum_{i=1}^{n-j+1} \sup S_{|\rho_i|}f |_{[\rho_i]} +(n-j)(\log \kappa-L)-L,$$
and \eqref{c212} follows. 

Observe that since $\bar{\om} \in F$, Lemma \ref{c212} implies
\begin{equation}
\label{restrcapf}
\sup S_{|\bar{\om}|}f|_{[\bar{\om}]} \leq \sup S_{|\bar{\om}|}f|_{[\bar{\om}]_F}+L.
\end{equation}
Therefore \eqref{claim2} follows by \eqref{c21}, \eqref{c212} and \eqref{restrcapf}.

By \eqref{zkl} and \eqref{claim2},
\begin{equation}
\begin{split}
\label{spec3}
&Z_n(F \cup \{a\},f) \\
&\quad\quad\leq \sum_{j=0}^n \sum_{\om \in F^n_j} \exp \left((n-j)(\sup f|_{[a]}+L-\log \kappa)+\sup S_{|\bar{\omega}|}f|_{[\bar{\om}]_F} +2L\right).
\end{split}
\end{equation}
If $\om \in F^n_j$ then
$$n \leq |\bar{\om}|\leq j +p_{\Lambda} (n-j),$$
therefore \eqref{key} implies that
\begin{equation}
\label{klemmalast}
\sum_{\om \in F^n_j} \exp(\sup S_{|\bar{\om}|}f|_{[\bar{\om}]_F}) \leq (\sharp \Lambda)^{n-j}\, {n \choose j} \,\sum_{k=n}^{j+p_{\Lambda}(n-j)} Z_{k}(F,f).
\end{equation}
The proof of the lemma follows by \eqref{zkl}, \eqref{spec3} and \eqref{klemmalast}.
\end{proof}

Using Lemma \ref{mainlemmasym} we prove two Propositions which give useful qualitative bounds for $P^{\sg}_{F \cup \{a\}}(f)$ in terms of $P^\sg_{F}(f)$ and $\sup f |_{[a]}$.

\begin{propo}
\label{ghenciu1sym}
Let $E$ be a finitely irreducible infinite countable set. Let $F\subset E$ be a finitely irreducible subset of $E$ and let $\Lambda$ be a nonempty set witnessing finite irreducibility for $F$. Then for every $a \in E$ and every $\a$-H\"older function $f: E_A^\N \ra \R$ such that $P^\sg_{F}(f)<0$,
$$e^{P^\sg_{F \cup \{a\}}(f)} \leq e^{P^\sg_{F}(f)}+ \frac{\sharp \Lambda \, e^{L(f)}}{ \kappa_\Lambda(f)}\,e^{\sup f|_{[a]}}.$$
\end{propo}

\begin{proof} As in the previous lemma we let $\kappa:=\kappa_{\Lambda}(f)$ and $L=L(f)$.  Let $\ve \in (0, |P^\sg_F(f)|)$. By the definition of the pressure function we know that there exists some $N_0 \in \N$ such that for all $n \geq N_0$,
$$Z_{n}(F,f) \leq e^{(P^\sg_F(f) +\ve)n}.$$ 
Therefore applying Lemma \ref{mainlemmasym} for $n \geq N_0$ and using that $P^\sg_F(f)+\ve <0$ we obtain that
\begin{align*}
Z_{n}(F \cup \{a\},f) & \leq e^{2 L}\sum_{j=0}^n {n \choose j}\, \left(\frac{\sharp \Lambda \,\exp(\sup f |_{[a]}+L)}{ \kappa} \right)^{n-j}\, \sum_{k=n}^{j+p_{\Lambda}(n-j)}  e^{(P^{\sg}_F(f) +\ve)k}\\
&\leq e^{2 L} p_{\Lambda} n \sum_{j=0}^n  {n \choose j}\, \left(\frac{\sharp \Lambda \,\exp(\sup f |_{[a]}+L)}{ \kappa} \right)^{n-j} e^{(P^\sg_F(f) +\ve)n} \\
&\leq  e^{2 L} p_{\Lambda} n \sum_{j=0}^n  {n \choose j}\, \left(\frac{\sharp \Lambda \,\exp(\sup f |_{[a]}+L)}{ \kappa} \right)^{n-j} e^{(P^\sg_F(f) +\ve)j}\\
&=e^{2L} p_{\Lambda} n  \left( \frac{\sharp \Lambda \,\exp(\sup f |_{[a]}+L)}{ \kappa}  +e^{(P^\sg_F(t) +\ve)}\right)^n.
\end{align*}
Taking $n$-th roots and letting $n \ra \infty$ we get,
$$e^{P^\sg_{F \cup \{a\}}(f)} \leq e^{P^{\sg}_F(f)+\ve} + \frac{\sharp \Lambda \, \exp(\sup f |_{[a]}+L)}{ \kappa} .$$
The proof now follows because $\ve$ can be taken arbitrarily small.
\end{proof}

\begin{propo}
\label{ghenciu2sym}
Let $E$ be a finitely irreducible infinite countable set. Let $F\subset E$ be a finitely irreducible subset of $E$ and let $\Lambda$ be a nonempty set witnessing finite irreducibility for $F$. Then for every $a \in E$ and every $\a$-H\"older function $f: E_A^\N \ra \R$ such that $P^\sg_{F}(f) \geq 0$,
$$e^{P^{\sg}_{F \cup \{a\}}(t)} \leq e^{P^\sg_{F}(f)}+ \frac{\sharp \Lambda \, e^{L(f)}}{ \kappa_\Lambda(f)}\, e^{p_{\Lambda}P^\sg_{F}(f)}\,e^{\sup f|_{[a]}}.$$
\end{propo}

\begin{proof} Let $\ve>0$ and let $N_0$ large enough  such that for all $k \geq N_0$,
$$Z_{k}(F,f) \leq e^{(P^\sg_F(f) +\ve)k}.$$ 
Therefore by Lemma \ref{mainlemma} for $n \geq N_0$ we have that
\begin{align*}
&e^{-p_{\Lambda} n(P^\sg_F(f)+\ve)}Z_{n}(F \cup \{a\},f)\\
 &\quad\leq   e^{2 L}\sum_{j=0}^n {n \choose j}\, \left(\frac{\sharp \Lambda \, \exp(\sup f |_{[a]}+L)}{ \kappa} \right)^{n-j}\, \sum_{k=n}^{j+p_{\Lambda}(n-j)} e^{-p_{\Lambda} n(P^\sg_F(f)+\ve)}e^{(P^\sg_F(f) +\ve)k} \\
&\quad\leq e^{2 L} \sum_{j=0}^n {n \choose j} \left(\frac{\sharp \Lambda \, \exp(\sup f |_{[a]}+L)}{ \kappa} \right)^{n-j}\ \sum_{k=n}^{j+p_{\Lambda}(n-j)} e^{-\ve (p_{\Lambda} n-k)} e^{-P^\sg_F(f)(p_{\Lambda} n-k)}.
\end{align*}
Now note that for $k=n,\dots,j+p_{\Lambda}(n-j),$ 
$$p_{\Lambda}n-k \geq p_{\Lambda}n-j-p_{\Lambda}(n-j) =(p_\Lambda-1) j.$$
Hence
\begin{align*}
e^{-p_{\Lambda} n(P^\sg_F(f)+\ve)}&Z_{n}(F \cup \{a\},f)\\
&\leq e^{2L}  \sum_{j=0}^n {n \choose j}  \left(\frac{\sharp \Lambda \, \exp(\sup f |_{[a]}+L)}{ \kappa} \right)^{(n-j)}  e^{-P^\sg_F(f)(p_{\Lambda} -1)j} \sum_{l=j(p_{\Lambda} -1)}^{n(p_{\Lambda} -1)} e^{-\ve l}\\
&\leq e^{2 L}p_{\Lambda} n \sum_{j=0}^n {n \choose j} \left(\frac{\sharp \Lambda \, \exp(\sup f |_{[a]}+L)}{ \kappa} \right)^{(n-j)}   e^{-P^\sg_F(f)(p_{\Lambda} -1)j}  e^{-\ve j(p_{\Lambda} -1) }\\
&\leq e^{2 L} p_{\Lambda} n \left( \frac{\sharp \Lambda \, \exp(\sup f |_{[a]}+L)}{ \kappa} + e^{-P^\sg_F(f)(p_{\Lambda} -1)}\right)^n.
\end{align*}
Taking $n$-th roots and letting $n \ra \infty$ we get,
$$e^{P^\sg_{F \cup \{a\}}(f)} \leq e^{p_{\Lambda}(P^\sg_F(f)+\ve)} \left(\frac{\sharp \Lambda \, \exp(\sup f |_{[a]}+L)}{ \kappa} +e^{-P^\sg_F(f)(p_{\Lambda} -1)} \right).$$
The proof now follows because $\ve$ can be taken arbitrarily small.
\end{proof}

As an immediate corollary of Proposition \ref{ghenciu1sym} and \ref{ghenciu2sym} we have the following estimate.
\begin{corollary}
\label{ghenciucombsym}
Let $E$ be a finitely irreducible infinite countable set. Let $F\subset E$ be a finitely irreducible subset of $E$ and let $\Lambda$ be a nonempty set witnessing finite irreducibility for $F$. Then for every $a \in E$ and every $\a$-H\"older function $f: E_A^\N \ra \R$,
$$e^{P^{\sg}_{F \cup \{a\}}(f)} \leq e^{P^\sg_{F}(f)}+ \frac{\sharp \Lambda \, e^{L(f)}}{ \kappa_\Lambda(f)}\,\max\{1,e^{p_{\Lambda}P^\sg_{F}(f)}\} \,e^{\sup f|_{[a]}} \,    .$$
\end{corollary}

\section{Revisiting $\theta$ parameters of Graph Directed Markov Systems}\label{sec:gdms}
The goal of this section is to clarify the role of $\theta$-parameters in the setting of Graph Directed Markov Systems. A {\it graph directed Markov system} (GDMS) \index{GDMS}
$$
\cS= \big\{ V,E,A, t,i, \{X_v\}_{v \in V}, \{\f_e\}_{e\in E} \big\}
$$
consists of
\begin{itemize}
\item a directed multigraph $(E,V)$ with a countable set of edges $E$, frequently referred to also as alphabet, and a finite set of vertices $V$,

\item an incidence matrix $A:E \times E \ra \{0,1\}$,

\item two functions $i,t:E\ra V$ such that $t(a)=i(b)$ whenever $A_{ab}=1$,

\item a family of non-empty compact metric spaces $\{X_v\}_{v\in V}$,

\item a number $s$, $0<s<1$, and

\item  a family of injective contractions $$ \{\phi_e:X_{t(e)}\to X_{i(e)}\}_{e\in E}$$ such that every $\phi_e,\, e\in E,$ has Lipschitz constant no larger than $s$.
\end{itemize}

We will usually use the more economical notation $\cS=\{\f_e\}_{e \in E}$ for a GDMS. Moreover we will  assume that for every $v \in V$ there exist $e,e' \in E$ such that $t(e)=v$ and $i(e')=v$. 

A GDMS $\cS=\big\{ V,E,A, t,i, \{X_v\}_{v \in V}, \{\f_e\}_{e\in E} \big\}$ is said to be {\it finitely irreducible} if its associated incidence matrix $A$ is finitely irreducible. Notice that if $\cS$ is a finite irreducible GDMS then it is finitely irreducible.

For $\om \in E^*_A$ we consider the map coded by $\om$:
\begin{equation}\label{phi-om}
\phi_\om=\phi_{\om_1}\circ\cdots\circ\phi_{\om_n}:X_{t(\om_n)}\to
X_{i(\om_1)} \qquad \mbox{if $\om\in E^n_A$.}
\end{equation}
For the sake of convenience we will write $t(\om) = t(\om_n)$ and $i(\om)=i(\om_1)$ for
$\om$ as in \eqref{phi-om}.

For $\om \in E^{\mathbb N}_A$, the sets
$\{\f_{\om|_n}(X_{t(\om_n)})\}_{n=1}^\infty$ form a descending
sequence of non-empty compact sets and therefore have nonempty intersection.
Since
$$
\diam(\f_{\om|_n}(X_{t(\om_n)})) \le s^n\diam(X_{t(\om_n)})\le s^n\max\{\diam(X_v):v\in V\}
$$
for every $n\in\N$, we conclude that the intersection
$$
\bigcap_{n\in  \N}\f_{\om|_n}(X_{t(\om_n)})
$$
is a singleton and we denote its only element by $\pi(\om)$.
In this way we define the coding map \index{coding map}
\begin{equation}
\label{picoding}
\pi:E^{\mathbb N}_A\to \du_{v\in V}X_v,
\end{equation}
the latter being a disjoint union of the sets $X_v$, $v\in V$.
The set
$$
J=J_\cS:=\pi(E^{\mathbb N}_A)
$$
will be called the {\it limit set} \index{limit set} (or {\it attractor}) \index{attractor} of the GDMS $\cS$.

We will be interested in conformal GDMS. 
\begin{defn}
\label{gdmsdef}\label{Carnot-conformal-GDMS}
A graph directed Markov system is called {\it conformal} if the following conditions are satisfied.
\begin{itemize}
\item[(i)] For every vertex $v\in V$, $X_v$ is a compact connected
subset of a fixed ambient space and $X_v=\overline{\Int(X_v)}$.
\item[(ii)] ({\it Open set condition} or {\it OSC}). \index{open set condition} For all $a,b\in
E$, $a\ne b$,
$$
\phi_a(\Int(X_{t(a)})) \cap \phi_b(\Int(X_{t(b)}))= \emptyset.
$$
\item[(iii)] For every vertex $v\in V$ there exists an open connected
set $W_v\spt X_v$ such that for every $e\in E$ with $t(e)=v$, the map
$\f_e$ extends to a conformal diffeomorphism of $W_v$ into $W_{i(e)}$.
\end{itemize}
\end{defn}

\begin{remark}
\label{ifsosc}
In the particular case when $V$ is a singleton and for every $e_1,e_2 \in E$, $A_{e_1e_2}=1$ if and only if $t(e_1)=i(e_2)$, the GDMS is called an {\it iterated function system} (IFS). In particular when we write that $\cS=\{ \f_e\}_{e \in E}$ is a conformal IFS, according to Definition \ref{gdmsdef}, we will assume that $\cS$ satisfies the open set condition.
\end{remark}

\begin{remark}
\label{carnot}
At this point some clarifications are needed. If the ambient space is Euclidean, either $\R^n$ or $\C$, then the term {\it conformal diffeomorphism} corresponds to its standard meaning. Recently in the monograph \cite{CTU}, the first and third author together with Tyson extended the framework of conformal graph directed Markov systems in the sub-Riemannian setting of Carnot groups. A {\it Carnot group}  is a connected and simply
connected nilpotent Lie group $\G$ whose Lie algebra $\fg$ admits a
stratification
\begin{equation}\label{fg}
\fg = \fv_1 \oplus \cdots \oplus \fv_\iota
\end{equation}
into vector subspaces satisfying the commutation rules
\begin{equation}\label{commutators}
[\fv_1,\fv_i]=\fv_{i+1}
\end{equation}
for $1\le i<\iota$ and $[\fv_1,\fv_{\iota}]=(0)$. In particular, the
full Lie algebra is generated via iterated Lie brackets of elements of
the lowest layer $\fv_1$ of the stratification. For a very concise introduction to Carnot groups see e.g. \cite[Chapter 1]{CTU}, a far more extensive source is \cite{BLU}. The simplest example of a  nonabelian Carnot group is the first (complex) Heisenberg group $\Heis$. The underlying space for  $\Heis$ is $\R^3$, which we also view as $\C\times\R$. We endow $\C\times\R$ with the group law
$$
(z;t) \ast (z';t')=(z+z';t+t'+2\Imag (z\overline{z'})),
$$
where we denote elements of $\Heis$ by either $(z;t) \in \C\times\R$ or $(x,y;t) \in \R^3$. A good reference on Heisenberg groups is \cite{CDPT}.  The Heisenberg group is the lowest-dimensional example in a class of groups, the Iwasawa groups, which are the nilpotent components in the Iwasawa decomposition of real rank one simple Lie groups. The one-point
compactifications of these groups, equipped with suitable
sub-Riemannian metrics, arise as boundaries at infinity of the
classical rank one symmetric spaces. For more details see \cite[Chapter 2]{CTU} and the references therein.

Returning to Definition \ref{Carnot-conformal-GDMS}, in the case when the ambient space is a fixed Carnot group $\G$,  a homeomorphism $f:\Omega \to \Omega'$ between domains in  $\G$ is said to be {\it conformal} if
\begin{equation*}
\lim_{r \to 0} \frac{\sup \{ d_{cc}(f(p),f(q)) \, : \, d_{cc}(p,q) = r \}}{\inf \{ d_{cc}(f(p),f(q')) \, : \, d_{cc}(p,q') = r \}} = 1
\end{equation*}
for all $p \in \Omega$. Here $d_{cc}$ denotes the Carnot-Carath\'eodory path metric in $\G$. We also remark that if $f:\Omega \ra \Omega'$ is a conformal map between domains in $\R^n$ or $\C$ we will denote its usual derivative at a point $p \in \Omega$ by $D f (p)$. In that case  $\|Df(p)\|$ will denote the norm of $D f (p)$. If on the other hand $f: \Omega \ra \Omega'$ is a Carnot conformal map and $p \in \Omega$, then 
\begin{equation*}
\|Df(p)\| = \lim_{q\to p} \frac{d(f(p),f(q))}{d(p,q)},
\end{equation*}
and it is usually referred to as  the {\em maximal stretching factor of $f$ at $p$}. As in the Euclidean case the quantity
$\|Df(\cdot)\|$ satisfies the Leibniz rule \index{Leibniz rule}
\begin{equation}\label{leibniz}
\|D(f\circ g)(p)\| = \|Df(g(p))\|\,\|Dg(p)\|.
\end{equation}
For the proof see \cite[Chapter 3]{CTU}.
\end{remark}

We also select a family $S_v, v \in V,$ of pairwise disjoint compact sets such that
$X_v \subset \Int(S_v) \subset S_v \subset W_v$. We set
\begin{equation*}
X := \bigcup_{v\in V} X_v \text{ and }S:= \bigcup_{v\in V} S_v.
\end{equation*}

Controlling the distortion of conformal maps is essential for developing the thermodynamic formalism for GDMS. We will use the following lemma repeatedly. For the proof in the Euclidean case see  \cite[Section 4.1]{MUbook}, for the proof in the case of Carnot GDMS see \cite[Lemma 4.9]{CTU}.

\begin{lm}[Bounded Distortion Property]\label{bdp} \index{bounded distortion property}
Let $S=\{\phi_e\}_{e\in E}$ be a  conformal GDMS. There exists a constant $K \geq 1$ so that $$
\biggl|\frac{\|D\f_\om(p)\|}{\|D\f_\om(q)\|}-1\biggr|\le Kd(p,q)
$$
and
$$
K^{-1}\le\frac{\|D\f_\om(p)\|}{\|D\f_\om(q)\|}\le K
$$
for every $\om\in E_A^*$ and every pair of points $p,q\in S_{t(\om)}$.
\end{lm}

For $\om \in E^*_A$ we set
$$\|D \f_\om\|_\infty := \|D \f_\om\|_{S_{t(\om)}}.$$

Lemma \ref{bdp} and \eqref{leibniz} easily imply
that if $\om \in E_A^\ast$ and $\om=\tau \upsilon$ for some $\tau, \upsilon \in E_A^\ast$, then
\begin{equation}\label{quasi-multiplicativity1}
K^{-1} \|D\f_{\tau}\|_\infty \, \|D\f_{\upsilon}\|_\infty \le
\|D\f_\om\|_\infty \le \|D\f_{\tau}\|_\infty \,
\|D\f_{\upsilon}\|_\infty.
\end{equation}

Let $\mathcal{S}=\{\f_e\}_{e\in E}$  be a finitely irreducible  conformal GDMS. For $t\ge 0$, $n \in \N$ and $F \subset E$ let
\begin{equation}\label{1_2017_12_20}
Z_{n}(F,t) = \sum_{\om\in F^n_A} \|D\phi_\om\|^t_\infty.
\end{equation}
When $F=E$, we just write $Z_n(t)$ instead of $Z_n(E,t)$. By \eqref{quasi-multiplicativity1} we easily  see that
\begin{equation}
\label{zmn}
Z_{m+n}(t)\le Z_m(t)Z_n(t),
\end{equation}
and consequently, the
sequence $(\log Z_n(t))_{n=1}^\infty$ is subadditive. Thus, the limit
$$
\lim_{n \to  \infty}  \frac{ \log Z_n(t)}{n}
$$
exists and equals $\inf_{n \in \N} (\log Z_n(t) / n)$. The value of
the limit is denoted by $P(t)$ or, if we want to be more precise, by
$P_E(t)$ or $P_\mathcal{S}(t)$. It is called the of the system $\mathcal{S}$ evaluated at the parameter $t$. 
  
Recall that a notion of topological pressure has already appeared in Section \ref{sec:preses}, Defintion \ref{toppressym}, in the context of symbolic dynamics. The next lemma whose proof can be found in \cite[Proposition 3.1.4]{MUbook}  establishes the natural relation between the two notions. Before stating it we introduce an important potential. Let $\zeta: E^\mathbb{N}_A \to \mathbb{R}$ defined by the formula
\begin{equation}\label{1MU_2014_09_10}
\zeta(\om):= \log\|D\phi_{\om_1}(\pi(\sg(\om))\|,
\end{equation}
where the coding map $\pi$ was defined in \eqref{picoding}. 
\begin{lm}\label{l1j85}
For $t \geq 0$ the function $t \zeta:E^\mathbb{N}_A \to
\mathbb{R}$ is H\"older continuous and $P^\sg(t\zeta)=P(t)$.
\end{lm}

The following well known proposition gathers some fundamental properties of topological pressure in the setting of finitely irreducible conformal GDMS.
\begin{propo}\label{p2j85}
Let $\cS$ be a finitely irreducible conformal GDMS. Then the
following conclusions hold.
\begin{enumerate}[label=(\roman*)]
\item \label{z1pfin}  $\{t\ge 0:Z_1(t)< +\infty\}=\{t\ge 0:P(t)< +\infty\}$.
\item \label{presscontconv} The topological pressure $P$ is
strictly  decreasing  on $[0,+\infty)$ with $P(t) \to -\infty$ as
$t\to+\infty$. Moreover, the function
$P$ is convex and  continuous on  $\overline{\{t\ge 0:Z_1(t)< +\infty\}}$.
\item \label{p0infty} $P(0)=+\infty$  if and only if $E$  is infinite.
\end{enumerate}
\end{propo}
The proof of \ref{z1pfin}  follows by \cite[Proposition 4.1]{CTU}. The  proofs of \ref{presscontconv} and \ref{p0infty} follow easily by the definition of topological pressure and they are omitted. 
\begin{defn}
\label{bowen}
Let $\cS=\{\f_e\}_{e \in E}$ be a finitely irreducible conformal GDMS. The number
$$
h = h({\mathcal{S}}) := \inf\{t\geq 0:  P(t)\leq 0\}
$$
is called {\it Bowen's parameter} \index{Bowen's parameter} of the  system $\mathcal{S}$.
\end{defn}

Note that if $\cS$ is a finitely irreducible conformal GDMS, then $
P(h)\le 0.
$
This follows easily by Proposition \ref{p2j85}.

\begin{defn}
\label{reguraldef}
A finitely irreducible conformal GDMS $\cS$ is:
\begin{enumerate}[label=(\roman*)]
\item {\it regular}  if $P(h)=0$,
\item {\it strongly regular} if there exists $t \geq 0$ such that $0< P(t) <+\infty$.
\end{enumerate}
\end{defn}

\begin{remark}\label{f1j87}
If $\mathcal{S}$ is a finitely irreducible conformal GDMS,
then $\theta \leq h$. If $\mathcal{S}$ is
strongly regular then $\theta< h$.
\end{remark}

We record the following fundamental theorem from thermodynamic formalism which relates Bowen's parameter to the Hausdorff dimension of dynamical systems. For the proof see \cite[Theorem 7.21]{CTU}.

\begin{thm}\label{t1j97}
If $\mathcal{S}$ is a finitely irreducible Carnot conformal GDMS, then
$$
h (\mathcal{S})
= \dim_{\mathcal{H}}(J_\mathcal{S})
= \sup \{\dim_\cH(J_F):  \, F \subset E \, \mbox{finite} \, \}.
$$
\end{thm}

The following lemma is a generalization of Lemma~3.19 from \cite{MU1} and a corrected version of Lemma 4.3.3 from \cite{MUbook}.

\begin{lm}
\label{433mu}
Let $\cS$ be a finitely irreducible Carnot conformal GDMS. The following conditions are equivalent.
\begin{enumerate}[label=(\roman*)]
\item \label{z1fin} $Z_1(t)<\infty$.
\item \label{z11cof}  There exists a finitely irreducible  cofinite subsystem $\cS_F$ of $\cS$ such that $Z_1(F,t)<\infty$.
\item \label{z1evcof}  For every  finitely irreducible  cofinite subsystem $\cS_F$ of $\cS$ it holds that $Z_1(F,t)<\infty$.
\item \label{pfin} $P(t)<\infty$.
\item \label{p1cof}  There exists a finitely irreducible cofinite subsystem $\cS_F$ of $\cS$ such that $P_{F}(t)<\infty$.
\item \label{myst}  For every  finitely irreducible cofinite subsystem $\cS_F$ of $\cS$ it holds that $P_{F}(t)<\infty$.
\end{enumerate}
\end{lm}
\begin{proof} The equivalence of \ref{z1fin} and \ref{pfin} is a restatement of  Proposition \ref{p2j85} \ref{z1pfin}. The implications \ref{z1evcof} $\implies$ \ref{z11cof}  and \ref{myst} $\implies$ \ref{p1cof}  are trivial. The implication  \ref{z1fin} $\implies$ \ref{z1evcof} is trivial because for all $F \subset E$, it holds that $Z_1(F,t) \leq Z_1(t)$. For the implication \ref{z11cof} $\implies$  \ref{z1fin}, assume by contradiction that \ref{z11cof} holds and $Z_1(t)=\infty$. Since $F$ is cofinite,
$$\sum_{e \in E \stm F} \|D \f_e\|_\infty ^t<\infty,$$
hence $\sum_{e \in  F} \|D \f_e\|_\infty^ t=\infty,$ and we have reached another contradiction.

The implication \ref{pfin} $\implies$  \ref{myst} follows because for all $F \subset E$ we have that $P_F(t) \leq P(t)$. Hence we only have to prove that \ref{p1cof} $\implies$ \ref{pfin}.  Suppose on the contrary that $P(t)=\infty$ and there exists some cofinite and finitely irreducible set $F\subset E$ such that $P_F(t)<\infty$. Then as before $Z_1(t)=\infty$ and thus $Z_1(F,t)=\infty$. Since $F$ is finitely irreducible the equivalence \ref{z1fin} and \ref{pfin} implies that $P_F(t)=\infty$ and we have reached a contradiction. The proof of the lemma is complete.  
\end{proof}
\begin{remark} Note that the equivalence \ref{z1fin} $\iff$ \ref{z11cof} $\iff$ \ref{z1evcof} in Lemma \ref{433mu} holds even when the cofinite sets $F$ are not finitely irreducible.
\end{remark}

We are now ready to introduce the main objects of study of this section. If $\cS=\{\f_e\}_{e \in E}$ is a GDMS and $F \subset E$ we will denote  by $J_F$ and $h_F$ the limit set and Bowen's parameter of the subsystem $\cS_F=\{\f_e\}_{e \in F}$.

\begin{defn}
\label{thetavariants}
Let $\cS=\{\f_e\}_{e \in E}$ be a finitely irreducible conformal GDMS. We define 
\begin{enumerate}[label=(\roman*)]
\item $\theta:=\theta(\cS)=\inf \{t \geq 0: P(t)<+\infty\}$,
\item $\theta_1:=\theta_1(\cS)=\inf \{\dim_{\cH}(J_A):A \subset E \mbox{ is cofinite}\}$,
\item $\theta_2:=\theta_2(\cS)=\inf \{\dim_{\cH}(J_A):A\subset E \mbox{ is cofinite and irreducible}\}$,
\item  $\theta_3:=\theta_3(\cS)=\inf \{\dim_{\cH}(J_A):A\subset E \mbox{ is cofinite and finitely irreducible}\}$.
\end{enumerate}
If $\Phi \subset E$ we also let
\begin{equation}
\label{theta1la}
\theta_1(\Phi):=\theta_1(\cS, \Phi)=\inf \{\dim_{\cH}(J_A):A \mbox{ is cofinite and }\Phi \subset A\}.
\end{equation}
\end{defn}
We remind the reader that if $\Lambda \subset E_A^{\ast}$, the set of all letters appearing in words of $\Lambda$ is denoted by $\tilde{\Lambda}$, see also \eqref{lambdatilde}.
\begin{thm} 
\label{thetaorderthm}
 Let $\cS=\{\f_e\}_{e \in E}$ be a finitely irreducible conformal GDMS. Then
\begin{enumerate}[label=(\roman*)]
\item \label{thetaiorder} $ \theta_1 \leq \theta_2 \leq \theta_3$,
\item \label{thetaorder} $\theta_1 \leq \theta  \leq \theta_3$,
\item\label{theta3equi} $\theta_3=\inf \{\theta_1(\tilde{\Lambda}): \Lambda \mbox{ witnesses finite irreducibility for some cofinite subset of }E\}$.
\end{enumerate}
\end{thm}
\begin{proof} We only need to prove \ref{thetaorder} and \ref{theta3equi} because \ref{thetaiorder} is obvious. We will first prove that $\theta_1 \leq \theta$. To this end let $t>\theta$. Then by Lemma \ref{433mu}
$$\sum_{e \in E} \|D \f_e\|^t_\infty<\infty.$$
Therefore there exists some finite set $M \subset E$ such that
$$\sum_{e \in E\stm M} \|D \f_e\|^t_\infty<1.$$
Hence for every finite $T$ such that $M \subset T \subset E$,  we have that
$$Z_1 (E \stm T) \leq Z_1(E \stm M, t)<1.$$
Therefore
$$P_{E\stm T}(t) \leq \log Z_1(E\stm T, t)<0,$$
which implies that $t \geq h_{E \stm T}$. Thus
$$\inf \{h_F: F\mbox{  cofinite subset of } E\} \leq \theta.$$
Inspecting the first part of the proof of \cite[Theorem 7.21]{CTU} we deduce that for any conformal GDMS $\cT$, not necessarily finitely irreducible, we have that 
\begin{equation}
\label{hupbound}
h_{\cT} \geq \dim_{\cH} J_{\cT}.
\end{equation}
In particular
$$\theta_1 \leq \inf \{h_F: F\mbox{  cofinite subset of } E\} \leq \theta.$$

We will now prove that $\theta \leq \theta_3$. We start by proving that if $F \subset E$ is cofinite and finitely irreducible then
\begin{equation}
\label{thetaf}
\theta=\theta(\cS_F).
\end{equation}
By Lemma \ref{433mu},
$$\theta=\inf \{t \geq 0: \sum_{e \in E} \|D \f_e\|^t_\infty<\infty\}$$
and
$$\theta(\cS_F)=\inf \{t \geq 0: \sum_{e \in F} \|D \f_e\|^t_\infty<\infty\}.$$
Hence $\theta(\cS_F) \leq \theta$. For the converse inequality suppose on the contrary that $\theta > \theta(\cS_F)$. Therefore there exists some $\lambda>0$ such that $\theta> \lambda> \theta(\cS_F)$. In particular $P(\lambda)=\infty$ and $P_F(\lambda)<\infty$. Since $F$ is cofinite and finitely irreducible, Lemma \ref{433mu} implies that $P(\lambda)<\infty$, but this is a contradiction. So \eqref{thetaf} has been established.

If $F \subset E$ is cofinite and finitely irreducible then by \cite[Theorem 7.21]{CTU}, Lemma \ref{433mu} and \eqref{thetaf},
\begin{equation*}
\begin{split}
\dim_{\cH}(J_{F})=h_{F}&=\inf \{t \geq 0: P_{F}(t) \leq 0\} \\
& \geq \inf \{t \geq 0: P_{F}(t)  <\infty\}\\
&=\inf \{t \geq 0: Z_1(F,t)<\infty\}=\theta(\cS_F)=\theta.
\end{split}
\end{equation*}
Hence $\theta \leq \theta_3$ and the proof of \ref{thetaorder} is complete.

We will now prove \ref{theta3equi}. Let
$$\theta_3'=\inf \{\theta_1(\tilde{\Lambda}): \Lambda \mbox{ witness finite irreducibility for some cofinite subset of }E\}.$$
We will first show that  $\theta_3\leq \theta_3'$. Suppose by way of contradiction that  there exists some $\lambda \in (\theta_3',\theta_3)$. So there exists a finite set $\Lambda \subset E_A^\ast$ and a cofinite set $F \subset E$ such that $\Lambda$ witnesses finite irreducibility for $F$ and $\theta_1(\tilde{\Lambda})<\lambda$. This means that
\begin{equation*}
\begin{split}
\lambda&>\inf \{\dim_{\cH}(J_A): \tilde{\Lambda} \subset A \mbox{ and }A\mbox{ is cofinite} \} \\
&\geq \dim_{\cH} (J_F) \\
&\geq \inf \{\dim_{\cH}(J_{B}): B \mbox{ is cofinite and finitely irreducible}\}\\
&=\theta_3,
\end{split}
\end{equation*}
which is a contradiction.

We will finally show that  $\theta'_3\leq \theta_3$. Take $t>\theta_3$. Then there exists $A \subset E$ cofinite and finitely irreducible such that 
$\dim_{\cH} (J_A)<t$. Since $A$ is finitely irreducible there exists a finite set $\Lambda \subset A$ such that $\Lambda$ witnesses finite irreducibility for $A$. Hence
$$\theta_3' \leq \theta_1(\Lambda) \leq \dim_{\cH} (J_A)<t.$$
Since $t$ was arbitrary we deduce that indeed  $\theta'_3\leq \theta_3$ and the proof is complete.
\end{proof}


We will now prove that Theorem \ref{thetaorderthm} is sharp. We stress that this phenomenon appears only in the framework of  graph directed Markov systems, since in the case of iterated function systems it is known \cite[Theorem 3.23]{MU1} that 
$$\theta=\theta_1=\theta_2=\theta_3.$$ 
\begin{thm}
\label{sharptheta}
There exist finitely irreducible conformal GDMS's  $\cS_i, i=1,\dots, 5,$ such that
\begin{enumerate}[label=(\roman*)]
\item \label{th1lessth2}
 $\theta_1(\cS_1) < \theta_2(\cS_1).$
  \item   \label{th1lessth}
 $\theta_1(\cS_2) < \theta(\cS_2).$
  \item  \label{thlessth2}
$\theta(\cS_3) < \theta_2(\cS_3).$
\item  \label{th2lessth3}
 $\theta_2(\cS_4) < \theta_3(\cS_4).$
\item \label{th2lessth}
  $\theta_2(\cS_5) < \theta(\cS_5).$
 \end{enumerate}
 \end{thm}
 \begin{proof} We begin with the proof of \ref{th1lessth2}. Let $\{a,b\} \in \R \stm \Z$ be a pair of distinct elements. Let
$$E= \N \cup \{-\N\} \cup \{a, b \}$$
 and consider any conformal GDMS $\cS=\{V,E,A,t,i, \{X_v\}_{v \in V}, \{\f_e\}_{e \in E} \}$ where the matrix $A$ is defined by
 \begin{equation*}
 \begin{split}
 A_{ma}&=1, \forall m \in -\N,\\
 A_{bm}&=1, \forall m \in -\N,\\
 A_{an}&=1, \forall n \in \N,\\
 A_{nb}&=1, \forall n \in \N,
 \end{split}
 \quad\quad\quad\quad\quad
\begin{split}
 A_{ab}&=1, \\
 A_{ba}&=1, \\
 A_{aa}&=1,\\
 A_{bb}&=1, 
 \end{split}
 \end{equation*}
 and $A_{ej}=0$ for all other $e,j \in E$. See also Figure \ref{fig1}. \begin{figure}
\centering
\includegraphics[scale = 0.37]{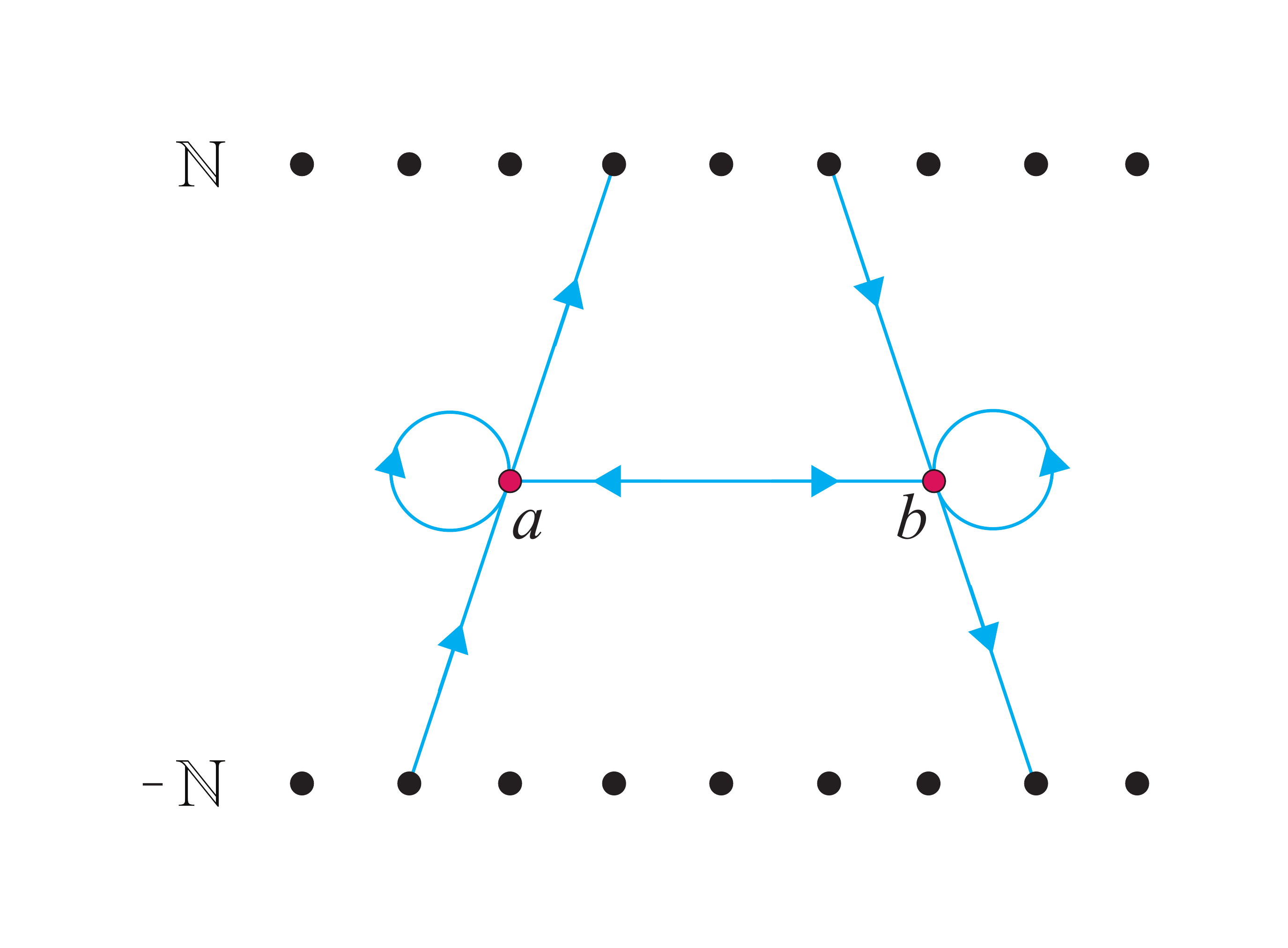}
\caption{The arrows between any two points $e,j \in E$ signify that $A_{ej}=1$.}
\label{fig1}
\end{figure} Note that the systems $\cS$ are finitely irreducible and the set $\{a,b\}$ witnesses finite irreducibility for $E$. Observe also that if $F \subset E$ is irreducible then $\{a,b\} \subset F$, therefore 
$$\theta_2 \geq \dim_{\cH} (J_F) \geq \dim_{\cH} (J_{\{a,b\}})>0.$$
On the other hand it is evident that $\theta_1=0$. Just  consider the cofinite set $I=E \stm \{a,b\}$ and notice that $I_A^{\N}=\emptyset$.

For the proof of \ref{th1lessth} it suffices to consider any system $\cS$ as in the previous example, consisting of similiarities such that for all $n \in \N,$
\begin{equation*}
\|D \f_{n}\|_{\infty}=\|D \f_{-n}\|_\infty=n^{-2},
\end{equation*}
and
\begin{equation*}
\|D \f_{a}\|_{\infty}=\|D \f_{b}\|_\infty=1/2.
\end{equation*}
Since this is just a special case of the examples considered in \ref{th1lessth2} we have that $\theta_1=0$. On the other hand
$$Z_1(t)=2 \left ((1/2)^t +\sum_{n \in \N}n^{-2t} \right),$$
and by Lemma \ref{433mu} we deduce that $\theta=1/2$.

The family introduced in \ref{th1lessth2} also contains examples which establish \ref{thlessth2}. In this case we consider similarity maps such that
 for all $n \in \N,$
\begin{equation*}
\|D \f_{n}\|_{\infty}=\|D \f_{-n}\|_\infty=e^{-n},
\end{equation*}
and
\begin{equation*}
\|D \f_{a}\|_{\infty}=\|D \f_{b}\|_\infty=1/2.
\end{equation*}
Hence
$$Z_1(t)=2 \left ((1/2)^t +\sum_{n \in \N}e^{-tn} \right),$$
and and by Lemma \ref{433mu} we deduce that $\theta=0$. Now recalling the proof of \ref{th1lessth2} we have that $\theta_2 \geq \dim_{\cH} (J_{\{a,b\}})>0$, and the proof of \ref{thlessth2} is complete.

We now move to the proof of \ref{th2lessth3}. Let the alphabet $E$ be as in the previous examples. We now consider conformal GDMS 
$\cS=\{V,E,A,t,i, \{X_v\}_{v \in V}, \{\f_e\}_{e \in E} \}$ where the matrix $A$ is defined by
 \begin{equation*}
 \begin{split}
 A_{ma}&=1, \forall m \in -\N,\\
 A_{bm}&=1, \forall m \in -\N,\\
 A_{an}&=1, \forall n \in \N,\\
 A_{nb}&=1, \forall n \in \N,
 \end{split}
 \quad\quad 
 \begin{split}
 A_{m,m-1}&=1, \forall m \in \Z \stm \{0,1\}\\
 A_{n,-n}&=1, \forall n \in \N,\\
 A_{-n,n}&=1, \forall n \in \N,\\
 \end{split}
 \quad\quad 
\begin{split}
 A_{ab}&=1, \\
 A_{ba}&=1, \\
 A_{aa}&=1,\\
 A_{bb}&=1, 
 \end{split}
 \end{equation*}
 and $A_{ej}=0$ for all other $e,j \in E$.  See also Figure \ref{fig2}. \begin{figure}
\centering
\includegraphics[scale = 0.365]{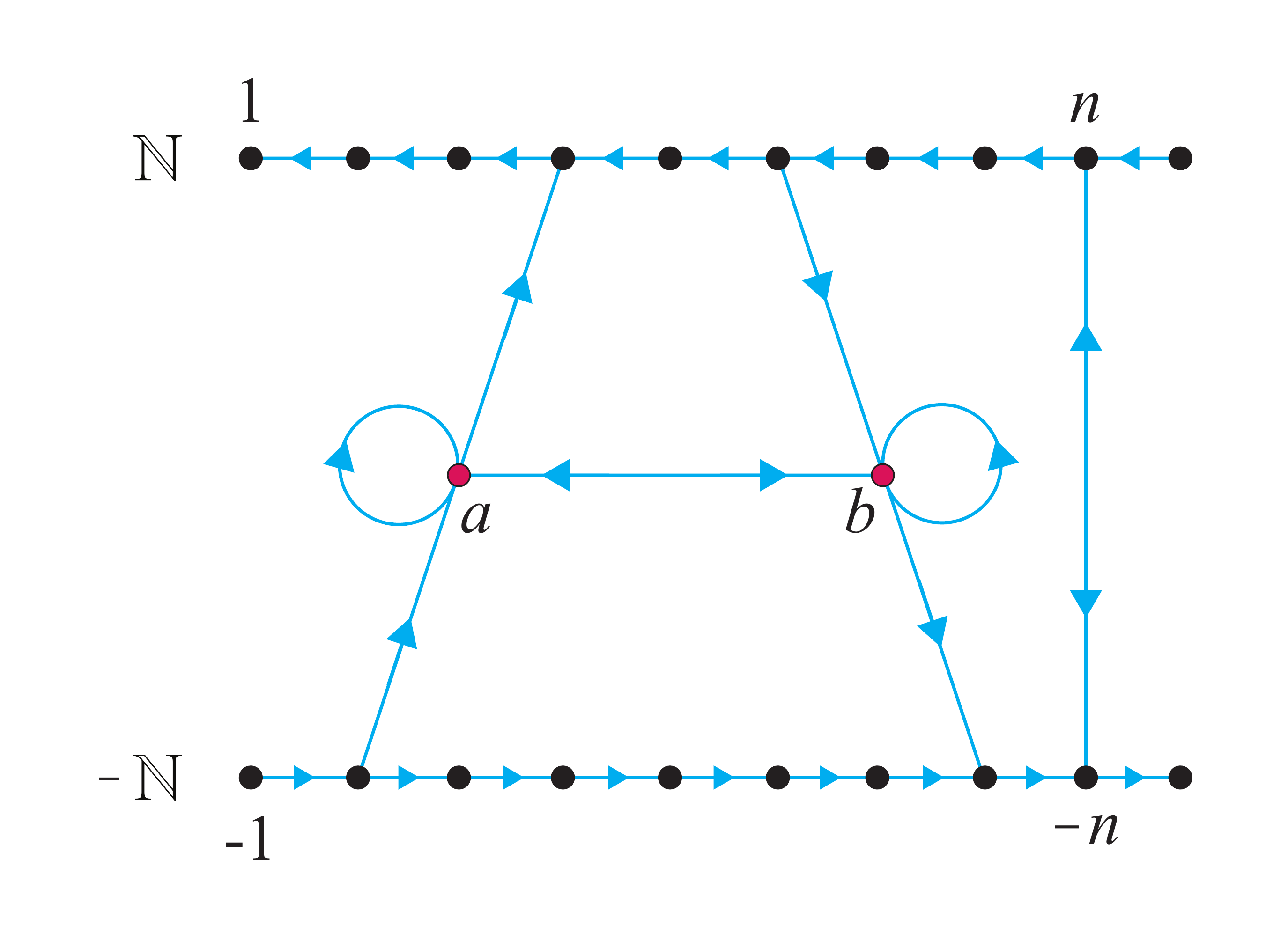}
\caption{}
\label{fig2}
\end{figure} As before we consider similarity maps such that
 for all $n \in \N,$
\begin{equation*}
\|D \f_{n}\|_{\infty}=\|D \f_{-n}\|_\infty=e^{-n},
\end{equation*}
and
\begin{equation*}
\|D \f_{a}\|_{\infty}=\|D \f_{b}\|_\infty=1/2.
\end{equation*}
The system $\cS$ is finitely irreducible and  if $F$ is any finitely irreducible subset of $E$ then $\{a,b\} \subset F$ and
$$\theta_3 \geq \dim_{\cH} (J_F) \geq \dim_{\cH} (J_{\{a,b\}})>0.$$
We now consider the cofinite sets
$$F_q=\{j \in \Z: |j| \geq q\}$$
for $i \in \N$. Note that the sets $F_q$ are irreducible and moreover
$$Z_{1}(F_q,t)=2 \sum_{j=q}^{\infty} e^{-t j}=2 e^{-t q}(1-e^{-t})^{-1}.$$
Hence for all $t>0$ there exists some $q(t) \in \N$ such that
$Z_1(F_{q(t)},t)<1$. Therefore for all $t>0$,
$$P_{F_{q(t)}}(t) \leq \log Z_1(F_{q(t)},t)<0,$$
and
$$ h_{F_{q(t)}}\leq t.$$
Therefore  
$$\theta_2 \leq \dim_{\cH} (J_{F_{q(t)}}) \leq t.$$
Since $t>0$ was arbitrary we deduce that $\theta_2=0$. 

For the proof of \ref{th2lessth3} we let $a \in \R \stm \N$ and we consider the alphabet $E=\{n \in \N: n \geq 2\} \cup \{a\}$.  Let  
$\cS=\{V,E,A,t,i, \{X_v\}_{v \in V}, \{\f_e\}_{e \in E} \}$ be a conformal GDMS  where the matrix $A$ is defined by 
 \begin{equation*}
 \begin{split}
 A_{m,m+1}&=1, \forall m \geq 2,\\
 A_{m+1,m}&=1, \forall m \geq 2,\\
 A_{am}&=1, \forall m \geq 2,\\
 A_{ma}&=1, \forall m \geq 2,\\
 A_{aa}&=1,
 \end{split}
 \end{equation*}
 and $A_{ej}=0$ for all other $e,j \in E$. See also Figure \ref{fig3}. \begin{figure}
\centering
\includegraphics[scale = 0.4]{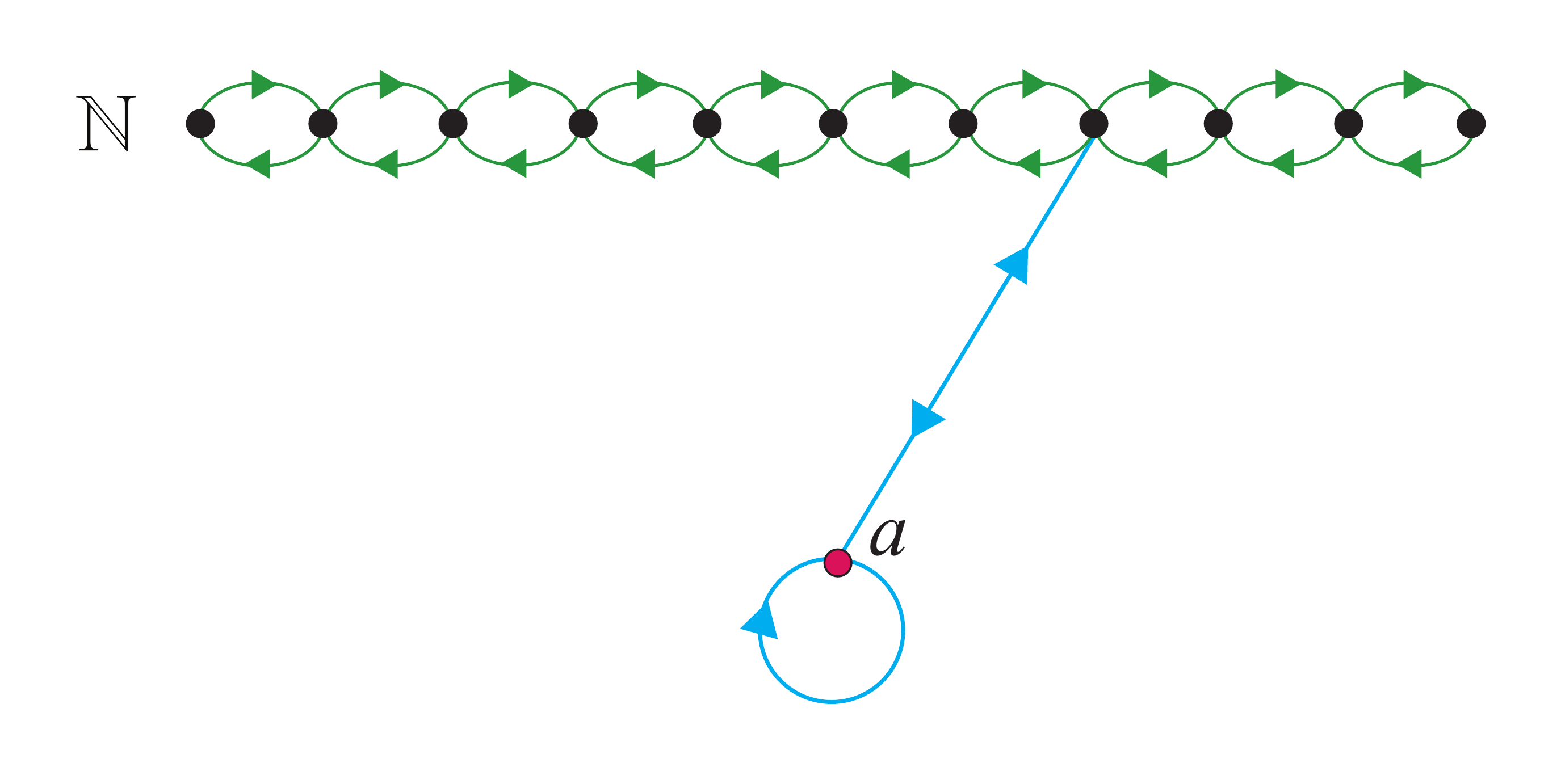}
\caption{}
\label{fig3}
\end{figure}The system $\cS$  consists  of similarities
 such that for all $m \in \N, m \geq 2,$
\begin{equation*}
\|D \f_{m}\|_{\infty}=m^{-1}
\mbox{ and }
\|D \f_{a}\|_{\infty}=1/2.
\end{equation*}
Observe that $\cS$ is finitely irreducible and 
$$Z_1(t)= 2^{-t}+\sum_{m \geq 2} m^{-t}.$$
Hence $\theta=1$. On the other hand for $q \geq 2$ consider the irreducible sets
$$F(q)=\{m \in \N: m \geq q \}.$$
Then  $\om \in {F(q)}^2_A$ if and only if $\om=(n,n+1)$ or if $\om=(m+1,m)$ for some $m,n \geq q$. Note that using \eqref{quasi-multiplicativity1} we get that
$$Z_2 (F(q), t) =\sum_{\om \in {F(q)}^2_A} \|D \f_{\om}\|_\infty^t \leq 2  \sum_{n \geq q} (n(n+1))^{-t}.$$
Hence if  $t>1/2$ there exists some $q \in \N$ large enough such that
$Z_2 (F(q),t) <1$. Therefore
$$P(t) \leq\frac{ \log Z_2 (F(q),t)}{2} <0,$$
and consequently $h_{F(q)} \leq t.$ Thus by \eqref{hupbound}
$$\theta_2 \leq \dim_{\cH} (J_{F(q)}) \leq  h_{F(q)} \leq t.$$
Since $t>1/2$ was chosen arbitrarily we deduce that $\theta_2\leq 1/2$ and the proof is complete.
\end{proof}

The following theorem is the corrected version of Lemma 4.3.10 from \cite{MUbook}. It extends some results from Section~3 of \cite{MU1} in the setting of GDMS.

\begin{thm}
\label{4310mu}
Let $\mathcal{S}=\{\f_e\}_{e \in E}$ be a finitely irreducible conformal GDMS. Then the following conditions are equivalent.
\begin{enumerate}[label=(\roman*)]
\item \label{ctu(i)} $\cS$ is strongly regular.
\item \label{ctu(ii)} $h({\cS})>\theta({\cS})$.
\item \label{ctu(iii)}  There exists a proper cofinite and finitely irreducible subsystem $\cS' \subset \cS$ such that $h({\cS'})<h({\cS})$.
\item \label{ctu(iv)}  For every proper and finitely irreducible subsystem $\cS' \subset \cS$ it holds that $h({\cS'})<h({\cS})$.
\end{enumerate}
\end{thm}
\begin{proof} The implications \ref{ctu(iv)}$\Rightarrow$\ref{ctu(iii)} and \ref{ctu(ii)}$\Rightarrow$\ref{ctu(i)} are immediate. In order to prove the implication \ref{ctu(iii)}$\Rightarrow$\ref{ctu(ii)} suppose by way of contradiction that $h(\cS)=\theta(\cS)$. Let $\cS'$ be a cofinite and finitely irreducible subsystem of $\cS$ such that $h({\cS'})<h({\cS})$. By Theorem \ref{thetaorderthm} \ref{thetaorder} and Theorem \ref{t1j97} we deduce that $$h({\cS'}) \geq \theta_3({\cS}) \geq \theta (\cS).$$ Hence by our assumption $h(\cS') \geq \theta(\cS)=h(\cS),$
which contradicts \ref{ctu(iii)}.

For the remaining implication \ref{ctu(i)}$\Rightarrow$\ref{ctu(iv)} let $E' \subsetneq E$ be finitely irreducible and consider the corresponding proper subsystem of $\cS$,  $\cS'=\{\f_e\}_{e \in E'}$. If $\cS'$ is not regular then by Proposition \ref{p2j85} we deduce that  $P_{\cS'}(\theta({\cS'}))<0$. Since $\cS$ is strongly regular, recalling Remark \ref{f1j87}, there exists some $\alpha \in (\theta(\cS), h(\cS))$. Therefore since $\alpha > \theta(\cS) \geq \theta({\cS'})$ and the pressure function is strictly decreasing 
we deduce that $P_{\cS'}(\alpha)<0$. Thus by the definition of the parameter $h({\cS'})$, we get that $h({\cS'}) \leq \alpha <h(\cS)$ and we are done in the case when $\cS'$ is  not regular.

Now by way of contradiction assume that $\cS'$ is regular and
$$h(\cS)=h({\cS'}):=h.$$
By Theorem \cite[Theorem 7.4]{CTU}  there exist unique measures $\tilde\mu_h$ on $E_A^\N$ and $\tilde\mu'_h$ on ${E'_A}^\N$,  which are  ergodic and shift invariant with respect to $\sigma:E_A^\N \ra E_A^\N$ and $\sigma':{E'_A}^\N\ra {E'_A}^\N$  respectively. Moreover,  by \cite[Proposition 7.19]{CTU} we have that
\begin{equation}
\label{confequiv}
\tilde m_h \ll \tilde\mu_h \ll \tilde m_h \text{ and }\tilde m'_h \ll \tilde\mu'_h \ll \tilde m'_h,
\end{equation}
where $\tilde m_h, \tilde m'_h$ stand for the $h$-conformal measures corresponding to $\cS$ and $\cS'$ respectively. We record that for every  $\om \in {E_A}^\ast$,
\begin{equation}\label{preequicyl}\tilde  m_h([\om]_A)  \approx \|D \f_\om\|^h_\infty,\end{equation}
and for every $\om \in {E'_A}^\ast$
\begin{equation}\label{preequicyl1}\tilde  m'_h([\om]'_A) \approx \|D \f_\om\|^h_\infty,\end{equation}
where $[\om]'_A=[\om]_A \cap {E'_A}^\N$. These estimates are the only things that we will need from conformal measures, for their definition as well as the proof of \eqref{preequicyl} see \cite[Section 7]{CTU}. Note that \eqref{confequiv}, \eqref{preequicyl} and \eqref{preequicyl1} imply that for every $\om \in {E'_A}^\ast$,
\begin{equation}\label{equicyl}
\tilde\mu_h([\om]_A)\approx\tilde\mu'_h([\om]'_A).
\end{equation}
Now in the obvious way we can extend $\tilde\mu'_h$ to a Borel measure in $E_A^\N$, defined by
$$\tilde\nu_h(B):=\tilde\mu'_h(B \cap {E'_A}^\N)$$
for Borel sets $B \subset E_A^\N$. We will first prove that
$$\tilde\nu_h \ll \tilde\mu_h.$$ 
It suffices to show that if $B$ is a Borel subset of $E_A^\N$ such that $B \subset {E'_A}^\N$ and $\tilde\mu_h(B)=0$ then $\tilde\nu_h(B)=0$. Let $\ve>0$ and let $U \subset E_A^\N$ be an open set such that $U \supset B$ and $\tilde\mu_h(U)<\ve$. Since the collection of cylinders forms a countable base for the topology of $E_A^\N$  there exists a countable set $I \subset {E_A}^\ast$ consisting of mutually incomparable words such that $$U=\bigcup_{\om \in I} [\om]_A.$$ 
Hence $U'=\bigcup_{\om \in I} [\om]_A \cap {E'_A}^\N=\bigcup_{\om \in I} [\om]'_A$ and by \eqref{equicyl},
\begin{equation*}
\begin{split}
\tilde\nu_h(B) \leq \tilde\nu_h(U')&=\sum_{\om \in I}\nu_h ([\om]'_A)= \sum_{\om \in I}\tilde\mu'_h ([\om]'_A) \\
&\approx \sum_{\om \in I}\tilde\mu_h ([\om]_A)=\tilde\mu_h (U)<\ve.
\end{split}
\end{equation*}
Since $\ve>0$ was arbitrary, we deduce that $\tilde\nu_h(B)=0$ and the proof is complete.

We will now show that $\tilde\nu_h$ is shift invariant with respect to $\sigma:E_A^\N \ra E_A^\N$. First notice that
\begin{equation*}
\begin{split}
\tilde\nu_h(\sigma^{-1}({E'_A}^\N))&=\tilde\nu_h\left(\bigcup_{j \in E}\{j\om: \om \in {E'_A}^\N\}\right)\\
&=\tilde\mu'_h \left(\bigcup_{j \in E}\{j\om: \om \in {E'_A}^\N\} \cap {E'_A}^\N\right)\\
&=\tilde\mu'_h({E'_A}^\N)=\tilde\nu_h({E'_A}^\N).
\end{split}
\end{equation*}
Let $F$ be a Borel  subset of $E_A^\N$. Then
\begin{equation}
\label{shifinv1}
\begin{split}
\tilde\nu_h(\sigma^{-1}(F))&=\tilde\nu_h(\sigma^{-1}(F)\cap {E'_A}^\N)=\tilde\nu_h(\sigma^{-1}(F) \cap \sigma^{-1}({E'_A}^\N))\\
&=\tilde\mu'_h(\sigma^{-1}(F \cap {E'_A}^\N)\cap {E'_A}^\N).
\end{split}
\end{equation}
We will now show that
\begin{equation}
\label{shifinv2}
\sigma^{-1}(F \cap {E'_A}^\N)\cap {E'_A}^\N=\sigma'^{-1}(F \cap {E'_A}^\N).
\end{equation}
Recall that $\sigma'$ stands for the shift map in ${E'_A}^\N$. We have,
\begin{equation*}
\begin{split}
\sigma^{-1}(F \cap {E'_A}^\N)\cap {E'_A}^\N&=\bigcup_{j \in E}\{j\om: \om \in F \cap {E'_A}^\N\}\cap {E'_A}^\N\\
&=\bigcup_{j \in E'}\{j\om: \om \in F \cap {E'_A}^\N\}\cap {E'_A}^\N\\
&=\sigma'^{-1}(F \cap {E'_A}^\N),
\end{split}
\end{equation*}
and \eqref{shifinv2} follows. Now using \eqref{shifinv1}, \eqref{shifinv2} and the $\sigma'$-invariance of $\tilde \mu'_h$ we get,
\begin{equation*}
\begin{split}
\tilde\nu_h(\sigma^{-1}(F))&=\tilde\mu'_h(\sigma'^{-1}(F \cap {E'_A}^\N))\\
&=\tilde \mu'_h(F \cap {E'_A}^\N)\\
&=\tilde\nu_h(F).
\end{split}
\end{equation*}
Hence $\tilde\nu_h$ is $\sg$-invariant.

We will now show that $\tilde\nu_h$ is ergodic with respect to $\sigma$. By way of contradiction suppose that there exists some Borel subset $F$ of $E_A^\N$ such that $\sigma^{-1}(F)=F$ and $0<\tilde \nu_h (F)<1$. Let
$$F_1=F \cap {E'_A}^\N \text{ and }F_2=F \setminus F_1.$$
Since $\sigma'^{-1}(F_1) \subset \sigma^{-1}(F_1) \subset F_1 \cup F_2$ and $\sigma'^{-1}(F_1) \cap F_2=\emptyset$ we deduce that
\begin{equation}
\label{shifergo}
\sigma'^{-1}(F_1) \subset F_1.
\end{equation}
Moreover,
\begin{equation*}
\begin{split}
F_1 \subset F= \sigma^{-1}(F) \subset \bigcup_{j \in E'} \{jf: f \in F\} \cup \bigcup_{j \in E \stm E'} \{jf: f \in F\}.
\end{split}
\end{equation*}
Therefore
$$F_1 \subset \bigcup_{j \in E'} \{jf: f \in F_1\} \cap {E'_A}^\N= \sigma'^{-1}(F_1),$$
which combined with \eqref{shifergo} implies that
\begin{equation}
\label{shifergo1}
F_1=\sigma'^{-1}(F_1).
\end{equation}
Since $\tilde\mu'_h$ is ergodic with respect to $\sigma'$ we deduce that either $\tilde\mu'_h(F_1)=0$ or $\tilde\mu'_h(F_1)=1.$ Therefore, since $\tilde\nu_h(F)=\tilde\mu_h(F_1)$,
$$\tilde\nu_h(F)=0 \text{ or }\tilde\nu_h(F)=1,$$
and we have reached a contradiction. Thus $\tilde\nu_h$ is ergodic with respect to $\sigma$.

Hence we have shown that there exist two probability Borel measures on $E_A^\N$, $\tilde \mu_h$ and $\tilde\nu_h$, which are shift invariant and ergodic with respect to $\sigma$ and they are absolutely continuous with respect to $\tilde m_h$. Now Theorem \cite[Theorem 7.4]{CTU} implies that
\begin{equation}
\label{equivnumu}
\tilde \mu_h\equiv\tilde\nu_h.
\end{equation}
If $j \in E \stm E'$, then $\tilde\nu_h([j])=0$. On the other hand, since $\tilde\mu_h$ is equivalent to $\tilde m_h$, by \eqref{preequicyl}  we deduce that $\tilde\mu_h([j])>0$. Therefore \eqref{equivnumu} cannot hold and we have reached a contradiction. The proof of the theorem is complete.
\end{proof}

 \section{Dimension spectrum of conformal graph directed Markov systems}\label{sec:gdmsspec}
In this section we investigate several properties of the dimension spectrum of graph directed Markov systems. We start by defining the notion of dimension spectrum and some related concepts.
\begin{defn}
\label{dspec} Let $\cS=\{\f_e\}_{e \in E}$ be a finitely irreducible conformal GDMS. The {\em dimension spectrum} of $\cS$ is
$$DS (\cS):=\{\dim_{\cH}(J_F): F \subset E \}.$$
If $L \subset E$ we set
$$DS_{L} (\cS):=\{\dim_{\cH}(J_F): L \subset F \subset E \},$$
and we define
\begin{enumerate}[label=(\roman*)]
\item $DS_1(\cS)=\bigcup \{DS_{\tilde{\Lambda}}(\cS): \Lambda \mbox{ witnesses finite irreducibility for }E\},$
\item $DS_2(\cS)=\{\dim_{\cH}(J_F): F \subset E, F \mbox{ is finitely irreducible}\},$
\item $DS_3(\cS)=\{\dim_{\cH}(J_F): F \subset E, F \mbox{ is irreducible}\}$.
\end{enumerate}
Note that if $L=\emptyset$ then $DS_{L} (\cS)=DS(\cS)$.

 \end{defn}
\begin{defn}
\label{hzero}
Let $\cS=\{\f_e\}_{e \in E}$ be a finitely irreducible conformal GDMS. We let 
$$h_0:=\inf \{\dim_{\cH} (J_{\tilde{\Lambda}}): \Lambda \mbox{ witnesses finite irreducibility for }E\}.$$
\end{defn}

The pressure estimates obtained in Section \ref{sec:preses} will play a crucial role in the following. We now record their translation  to the setting of finitely irreducible conformal GDMS. Recall that for any finite $\Lambda \subset E_A^\ast$ we denote
$$p_\Lambda:=\max \{ |\lambda|: \lambda \in \Lambda\}.$$
Now let $t \geq 0$ and recalling \eqref{1MU_2014_09_10} we consider the potential 
$$f(\om):=t\zeta (\om)=t\log \|D \f_{\om_1}(\pi(\sg(\om)))\|.$$ For $\om \in E^\N_A$,
\begin{equation*}\begin{split}
S_n  f (\om)&=t\sum_{j=0}^{n-1} \log \|D \f_{\om_{j+1}}(\pi(\sg^{j+1}(\om)))\|\\
&=t\log\left(\|D \f_{\om_{1}}(\pi(\sg^{1}(\om)))\|\,\|D \f_{\om_{2}}(\pi(\sg^{2}(\om)))\|\cdots \|D \f_{\om_{n}}(\pi(\sg^{n}(\om)))\| \right).
\end{split}\end{equation*}
By \eqref{leibniz} and the fact that  $\pi(\sg^k(\om))= \f_{\om_{k+1}}\circ \dots \circ \f_{\om_n}(\pi(\sg^n(\om)))$ for $1 \leq k \leq n$ we deduce that
\begin{equation}
\label{logleib}
S_n  f (\om)=\log \|D \f_{\om|_n}(\pi(\sg^n(\om)))\|^t.
\end{equation}
Therefore by \eqref{logleib} and Lemma \eqref{bdp} we deduce that
$$K^{-t} \min\{\|D \f_\lambda\|^t_\infty: \lambda \in \Lambda\} \leq\kappa_{\Lambda}(f) \leq \min\{\|D \f_\lambda\|^t_\infty: \lambda \in \Lambda\},$$
where $\kappa_{\Lambda}(f)$ was defined in \eqref{klf}. For the rest of this section we will let
$$\kappa_{\Lambda}=\min\{\|D \f_\lambda\|_\infty: \lambda \in \Lambda\}.$$
Using \eqref{logleib} and Proposition \ref{bdp} we get that for all $\om,\tau\in E_A^\N$, and  $\rho\in
E_A^n, n \in \N$,
\begin{equation*}
\label{diffpot}
\big|S_nf(\rho\om)-S_nf(\rho\tau)\big|=t|\log \|D \f_{\rho}(\pi(\om))\|-\log \|D \f_{\rho}(\pi(\tau))\|| \leq \log K^t.
\end{equation*}
Therefore recalling \eqref{L(g)} we deduce that $L(f) \leq \log K^t$. It also follows that the function $f$ is $\log(1/s)$-H\"older, see e.g.\cite[Lemma 4.22]{MUbook} or \cite[Lemma 4.16]{CTU}, therefore we can directly apply Lemma \ref{mainlemmasym} and Propositions \ref{ghenciu1sym} and \ref{ghenciu2sym} to our setting. For the convenience of the reader we provide the statements in the setting of conformal GDMS.
 \begin{lm}
\label{mainlemma} Let $\cS=\{\f_e\}_{e \in E}$ be a finitely irreducible conformal GDMS. Let $F\subset E$ be a finitely irreducible subset of $E$ and let $\Lambda$ be a nonempty set witnessing finite irreducibility for $F$. Then for every $e \in E$ and every $t>0$
$$Z_{n}(F \cup \{e\},t) \leq K^{2 t}\sum_{j=0}^n \sum_{k=n}^{j+p_{\Lambda}(n-j)} \, {n \choose j} \left( \frac{K \|D \f_e\|_\infty}{\kappa_{\Lambda}}\right)^{t(n-j)} (\sharp \Lambda)^{n-j}  Z_{k}(F,t),$$
for every $n \in \N$.
\end{lm}

\begin{propo}
\label{ghenciu1}
Let $\cS=\{\f_e\}_{e \in E}$ be a finitely irreducible conformal GDMS. Let $F\subset E$ be a finitely irreducible subset of $E$ and let $\Lambda$ be a nonempty set witnessing finite irreducibility for $F$. Then for every $e \in E$ and every $t>\dim_{\cH} (J_F)$,
$$e^{P_{F \cup \{e\}}(t)} \leq e^{P_{F}(t)}+ \sharp \Lambda (K \kappa_{\Lambda}^{-1})^t \, \|D \f_e\|_{\infty}^t.$$
\end{propo}

\begin{propo}
\label{ghenciu2}
Let $\cS=\{\f_e\}_{e \in E}$ be a conformal GDMS. Let $F\subset E$ be a finitely irreducible subset of $E$ and let $\Lambda$ be a nonempty set witnessing finite irreducibility for $F$. Then for every $e \in E$ and every $t$ such that $P_F(t) \geq 0$,
$$e^{P_{F \cup \{e\}}(t)} \leq e^{P_{F}(t)}+ \sharp \Lambda (K \kappa_{\Lambda}^{-1})^t \,  e^{p_{\Lambda}P_{F}(t)} \|D \f_e\|_{\infty}^t .$$
\end{propo}
As an immediate corollary of Proposition \ref{ghenciu1} and \ref{ghenciu2} we have the following estimate.
\begin{corollary}
\label{ghenciucomb}
Let $\cS=\{\f_e\}_{e \in E}$ be a conformal GDMS. Let $F\subset E$ be a finitely irreducible subset of $E$ and let $\Lambda$ be a nonempty set witnessing finite irreducibility for $F$. Then for every $e \in E$ and every $t \geq 0$
$$e^{P_{F \cup \{e\}}(t)} \leq e^{P_{F}(t)}+ \sharp \Lambda (K \kappa_{\Lambda}^{-1})^t \, \max \{1, e^{p_{\Lambda}P_{F}(t)}\} \, \|D \f_e\|_{\infty}^t .$$
\end{corollary}

In the case of IFS one can obtain estimates as in Corollary \ref{ghenciucomb} using a more straightforward argument. The following proposition essentially appears in  \cite[Lemma 2.1]{KZ}, nevertheless we include the short proof for convenience of the reader.

\begin{propo}
\label{ifspe}
Let $\cS=\{\f_e\}_{e \in E}$ be a conformal IFS. Let $F \subset E$ and $e \in E \stm F$. Then for all $t >0$,
\begin{equation}
\label{presest}
e^{P_{F}(t)} +K^{-t} \|D \f_e\|^t_\infty\leq e^{P_{F \cup \{e\}}(t)} \leq e^{P_{F}(t)} +K^{t} \|D \f_e\|^t_\infty.
\end{equation}
\end{propo} 

\begin{proof}
Note that for all $n \in \N$,
\begin{equation*}
\begin{split}Z_n(F \cup \{e\},t) &\leq \sum_{j=0}^n {n \choose j} (K \|D \f_e\|_\infty)^{jt} Z_{n-j}(F,t) \\
&= \sum_{j=0}^n {n \choose j} (K \|D \f_e\|_\infty)^{jt} e^{P_F(t)(n-j)} e^{n \frac{n-j}{n} \ve_{n-j}},
\end{split}
\end{equation*}
where for $m \in \N$,
$$\ve_{m}:= \frac{1}{m} \log Z_m (F, t)-P_F(t).$$
Since $\ve_m \ra 0$ as $m \ra \infty$, it follows easily that
$$\gamma_m:= \max_{1 \leq j \leq m} \frac{m-j}{m} \ve_{m-j} \ra 0,$$
as $m \ra \infty$. Hence 
$$Z_n(F \cup \{e\},t) \leq e^{n \gamma_n} \left( K^t \|D \f_e\|^t_\infty +e^{P_F(t)}\right)^n,$$
and consequently
$$e^{P_{F \cup \{e\}}(t)} \leq e^{P_{F}(t)} +K^{t} \|D \f_e\|^t_\infty.$$

Therefore we only have to prove the left hand part of \eqref{presest}. As before,
\begin{equation*}
Z_n(F \cup \{e\},t) \geq 
\sum_{j=0}^n {n \choose j} K^{-tj} \|D \f_e\|_\infty^{jt} Z_{n-j}(F,t).
\end{equation*}
Since $P_{F}(t)= \inf_{m \in \N} \frac{\log Z_m(F,t)}{m},$ we have that $Z_{n-j} (F,t) \geq e^{(n-j) P_F(t)}.$
So
$$Z_n(F \cup \{e\},t) \geq 
\sum_{j=0}^n {n \choose j} K^{-tj} \|D \f_e\|_\infty^{jt} e^{(n-j) P_F(t)}=(K^{-t} \|D \f_e\|_\infty^{t}+e^{P_F(t)})^n,$$
and consequently
$$e^{P_{F \cup \{e\}}(t)} \geq e^{P_{F}(t)} +K^{-t} \|D \f_e\|^t_\infty.$$
The proof is complete.
\end{proof}
Using Lemma \ref{mainlemma} we can prove the following proposition which will be used to derive information about the size of the dimension spectrum of GDMSs.
\begin{propo}
\label{addone}
Let $\cS=\{\f_e\}_{e \in E}$ be a finitely irreducible conformal GDMS. Let $F \subset E$ such that
\begin{enumerate}[label=(\roman*)]
\item $E \stm F$ is infinite,
\item $F$ is finitely irreducible.
\end{enumerate}
Then for all $\ve>0$ and for all, except finitely many, $e \in E \stm F$,
$$\dim_{\cH}(J_{F \cup \{e\}})<\dim_{\cH} (J_{F})+\ve.$$
\end{propo}
\begin{proof}
Fix some $\ve>0$ and let $h=\dim_{\cH}(J_{F})$. By Theorem \ref{t1j97} $h=h_{F}$ where $h_F$ is Bowen's parameter of the subsystem $\cS_F$. By Proposition \ref{p2j85} we then deduce that $P_{F}(h+\ve)<0$. We are going to show that there exists some $\alpha \in (0,1)$ and $j_0 \geq 1$ such that
\begin{equation}
\label{eqn1}
Z_{j} (F,h+\ve)<\alpha^j, \mbox{ for all }j \geq j_0.
\end{equation}
By way of contradiction assume that \eqref{eqn1} fails. Then for all $\alpha \in (0,1)$ there exists a sequence $(j_m)_{m \in \N}$ such that
$$Z_{j_m} (F,h+\ve)\geq\alpha^{j_m}.$$
Therefore for all $\alpha \in (0,1)$
\begin{align*}
P_F(h+\ve)&=\lim_{n \ra \infty}\frac{1}{n} \log Z_{n}(F,h+\ve)\\
& = \lim_{m \ra \infty} \frac{1}{j_m} \log Z_{j_m}(F,h+\ve)\\
& \geq  \limsup_{m \ra \infty} \frac{1}{j_m} \log \alpha^{j_m}=\log \alpha.
\end{align*}
Hence we have shown that $P_F(h+\ve) \geq 0$ which is a contradiction. Thus \eqref{eqn1} holds.

Let $e \in E \stm F$ and let $\Lambda$ be a set witnessing finite irreducibility for $F$. Assuming that $n>j_0$ we have by \eqref{eqn1} that
$$ Z_{k}(F,h+\ve)<\alpha^k$$
for $k=n, \dots, j+p_{\Lambda}(n-j)$. Since $\alpha \in (0,1)$ and $n \geq j$ we have that $\alpha^n \leq \alpha^j$ and $j+p_{\Lambda} (n-j) \leq p_{\Lambda}n$. Therefore 
by Lemma \ref{mainlemma},
\begin{equation}
\label{est1}
\begin{split}
Z_{n}(F \cup \{e\},h+\ve)
&  \leq K^{2(h+\ve)} p_{\Lambda} \,n \, \sum_{j=0}^n \, {n \choose j} \left( \frac{K \|D \f_e\|_\infty}{\kappa_{\Lambda}}\right)^{(h+\ve)(n-j)} (\sharp \Lambda)^{n-j} \,\alpha^j \\
&= K^{2(h+\ve)} p_{\Lambda} \, n \left(  \sharp \Lambda\left( \frac{K \|D \f_e\|_\infty}{\kappa_{\Lambda}}\right)^{h+\ve} +\alpha\right)^n.
\end{split}
\end{equation}
Since $E \stm F$ is infinite, \cite[Lemma 5.17]{CTU} implies that for all, except finitely many, $e \in E \stm F$ 
$$\sharp \Lambda\left( \frac{K \|D \f_e\|_\infty}{\kappa_{\Lambda}}\right)^{h+\ve} +\alpha<1.$$
Therefore by \eqref{est1} we conclude that for such $e$ there exists some $N_0 \in \N$ such that 
$$Z_{n}(F \cup \{e\},h+\ve)<1$$
for all $n \geq N_0$. Therefore 
$$P_{F \cup \{e\}}(h+\ve)=\lim_{n \ra \infty} \frac{1}{n} \log Z_{n}(F \cup \{e\},h+\ve)<0.$$
Hence $$\dim_{\cH} (J_{F \cup \{e\}})<h+\ve.$$
The proof of Proposition \ref{addone} is complete.
\end{proof}
We are now in place to prove Theorem \ref{mainspectintro} which we restate for the convenience of the reader. Recall that the parameters $\theta_3$ and $h_0$ were introduced respectively in Definitions \ref{thetavariants} and \ref{hzero}. 
\begin{thm} 
\label{mainspect}Let $\cS=\{\f\}_{e \in E}$ be an infinite finitely irreducible conformal GDMS. For every $t \in (h_0, \theta_3)$ there exists $F \subset E$ such that
$$\dim_{\cH} (J_F)=t.$$
In other words $(h_0, \theta_3) \subset DS(\cS)$.
\end{thm}
\begin{proof} Let $t \in (h_0,\theta_3)$. Recalling Definition \ref{hzero} there exists some $\Lambda \subset E_A^{\ast}$ which witnesses finite irreducibility for $E$ and $$\dim_{\cH} (J_{F_1}) \in [h_0, t),$$
where $F_1=\tilde{\Lambda}$. Without loss of generality we can assume that $E=\N$. 
Proposition \ref{addone} allows us to  choose the minimal $k_1 \in \N$ such that
\begin{enumerate}[label=(\roman*)]
\item $k_1> \max F_1$,
\item $\dim_{\cH}(J_{F_1 \cup \{k_1\}})<t$.
\end{enumerate}
We set $F_2:=F_1 \cup \{k_1\}$ and we proceed inductively to obtain a sequence of finite subsets of $E$, which we label as $\{F_n\}_{n \in \N}$, such that
$$\dim_{\cH}(J_{F_n})<t.$$
We now set $F_t=\cup_{n=1}^\infty F_n$.
Notice that $F_t$ is infinite and it is finitely irreducible because $F_1 \subset F_t$. Hence by Theorem \ref{t1j97},
\begin{equation}
\label{spethe1}
\dim_{\cH}(J_{F_t})= \sup_{n \in \N} \{\dim_{\cH}(J_{F_n})\} \leq t.
\end{equation}
Now notice that $\N \setminus F_t$ is infinite. Because  if $F_t$ is cofinite just by the definition of $\theta_3$ we have that
$$\dim_\cH (J_{F_t}) \geq \theta_3>t,$$
and this contradicts \eqref{spethe1}. Now if $\dim_{\cH}(J_{F_t})=t$ we are done. If not, since $\N \setminus F_t$ is infinite, we can apply Proposition \ref{addone} once more in order to find some $q \in \N \stm F_t$ such that
\begin{equation}
\label{dimft}
\dim_\cH (J_{F_t \cup \{q\}})<t.
\end{equation}
Notice that either $q<k_1$ or $q \in (k_m, k_{m+1})$ for some $m \in \N$. But this is not possible because \eqref{dimft} implies that $\dim_\cH (J_{F_{n} \cup \{q\}})<t$ for all $n \in \N$ and if such $q$ existed it would contradict the minimality of $k_{1}$ or $k_{n+1}$, depending on the location of $q$. Therefore we have reached a contradiction and the proof of Theorem \ref{mainspect} is complete.
\end{proof}

Theorem \ref{mainspect} generalizes results from \cite{MU1} and \cite{CTU} in the setting of GDMS. More specifically it was proved in \cite{MU1} and \cite{CTU} that if $\cS$ is an infinite conformal IFS then $[0,\theta) \in DS(\cS)$. While $\theta=\theta_3$ when $\cS$ is a conformal IFS, it is remarkable that the lower bound $h_0$ for the spectrum of a conformal GDMS is sharp. This is proved in our next theorem.

\begin{thm}
\label{sharpexample}
There exists an infinite finitely irreducible conformal GDMS $\cS=\{\f_e\}_{e \in E}$ with the property that every subset $I \subset E$ such that $\dim_{\cH}(J_I)>0$ satifies 
$$\dim_{\cH}(J_I) \geq h_0.$$
\end{thm}

\begin{proof} It is enough to consider the conformal GDMS used to prove \ref{th1lessth2} of Theorem \ref{sharptheta}, see also Figure \ref{fig1}. It is immediate that
$$h_0=\dim_{\cH} (J_{\{a,b\}})>0.$$
We will now show that if $I \subset E$ and $\{a,b\} \nsubseteq I$ then
\begin{equation}
\label{dim0}
\dim_{\cH}(J_I)=0.
\end{equation} As we noted in the proof of Theorem \ref{sharptheta} \ref{th1lessth2}, if $I$ does not contain $a$ or $b$ then $I_A^\N=\emptyset$. Hence without loss of generality we can assume that $a \in I$ and $b \notin I$. Note that for every $\om \in I_A^\N$ we have that $\om_i \in -\N \cup \{a\}$ for all $i \in \N$. This is simply because if $\om_i \in \N$ for some $i \in \N$ then $\om_{i+1}=b$, which is impossible. Moreover if $i \geq 2$ then $\om_i=a$, because if there exists some $i \geq 2$ such that $\om_i \in -\N$ then $\om_{i-1}=b$. Hence 
$$I_A^\N=\{\om \in (-\N \cup \{a\})^\N: \om_i=a \mbox{ for all }i \geq 2\},$$
which is countable. Therefore $\dim_{\cH}(J_I)=0$ and \eqref{dim0} has been proven. The proof of Theorem \ref{sharpexample} is complete.
\end{proof}

 We will now investigate topological properties of the dimension spectrum of a GDMS.
 \begin{thm}
 \label{compactspectrum}
 Let $\cS=\{\f_e\}_{e \in E}$ be an infinite and finitely irreducible conformal GDMS.  If $\Lambda$ is a set witnessing finite irreducibility for $E$ then $DS_{\tilde{\Lambda}}(\cS)$ is compact.
 \end{thm}
 \begin{proof} Without loss of generality we can identify $E$ with $\N$. Let $\{E_n\}_{n \in \N}$ be a sequence of subsets of $E$ such that $\tilde{\Lambda} \subset E_n$ for all $n \in \N$ and 
 \begin{equation}
 \label{limhen}\lim_{n \ra \infty} \dim_{\cH}(J_{E_n})=\alpha.
 \end{equation}
 It suffices to show that there exists some  set $F$ such that $\tilde{\Lambda} \subset F \subset E$ and
 \begin{equation}
\label{dimjfeqa1}
 \dim_{\cH}(J_F)=\alpha.
 \end{equation}
 We will first consider the case when $\alpha>\theta_3$. 
 
We first assume that no $e \in E$ appears in infinitely many sets $E_n$. Observe that this is possible only if the set $\Lambda$ is empty. 
Recalling Definition \ref{thetavariants}, for every $m \in \N$  there exists a cofinite set $C_m \subset E$ such that
\begin{equation}
\label{dimhcm}
\dim_{\cH}(J_{C_m}) \leq \theta_3 +\frac{1}{m}.
\end{equation}
Since $C_m$ is cofinite there exists some $j_m \in \N$ such that
\begin{equation}
\label{jcmtails}
\{j_m, j_m+1,\dots\} \subset C_m.
\end{equation}
Now notice that since  no $e \in E$ appears in infinitely many $E_n$, if $k \in \{1,\dots,j_m-1\}$ there exist finitely many  sets $E_n$ such that $k \in E_n$. Hence there exist finitely many sets sets $E_n$ such that
$$E_n \cap \{1,\dots,j_m-1\} \neq \emptyset.$$
Therefore there exists some $n(m) \in \N$ such that
\begin{equation}
\label{entails}
E_{n} \subset \{j_m, j_m+1,\dots\}
\end{equation}
for all $n \geq n(m)$. Combining \eqref{dimhcm}, \eqref{jcmtails} and \eqref{entails} we conclude that there exists a strictly increasing sequence of natural numbers
$$n(1)<n(2)<\dots<n(m)<n(m+1)<\dots$$
such that
$$\dim_{\cH}(E_{n(m)}) \leq \theta_3 +\frac{1}{m}.$$
Letting $m \ra \infty$ and using \eqref{limhen} we obtain that 
$$\alpha=\lim_{m \ra \infty} \dim_{\cH}(E_{n(m)}) \leq \theta_3 <\alpha,$$
which is a contradiction.

Therefore we know that there exist natural numbers which appear in infinitely many $E_n$. Let $\theta_3 <\alpha'<\alpha$. We  set
\begin{equation}
\label{k1def}
k_1=\min \{e \in \N: e \in E_n \mbox{ for infinitely many }n \in \N\},
\end{equation}
and we define
$$G_1= \{n \in \N: k_1 \in E_n\}$$
and 
$$G_1^k=\{n \in G_1:k \in E_n\},$$
for $k \in \N$. We will now show that there exist $k \in \N \stm \{k_1\}$ such that $G_1^k$ is infinite. Suppose on the contrary that for all $k \in \N \stm \{k_1\}$ the sets $G^k_1$ are finite. Then note that
\begin{equation}
\label{limtrun0}
\lim_{\substack{n \ra \infty\\ n \in G_1}} \min \{ E_n \stm \{k_1\}\}=+\infty.
\end{equation}
In order to prove  \eqref{limtrun0} assume by way of contradiction that there exists some $m_0 \in \N$ and some strictly increasing sequence $(n_j)_{j \in \N}, n_j \in G_1,$ such that
\begin{equation*}
\lim_{j \ra \infty} \min \{ E_{n_j} \stm \{k_1\}\}=m_0.
\end{equation*} 
This implies that there exists some $j_0$ such that for all $j \geq j_0$,
$$\min \{ E_n \stm \{k_1\}\}=m_0.$$
In particular $m_0 \in E_{n_j} \stm \{k_1\}$ and $n_j \in G_1$ for all $j \geq j_0$. Hence $G_1^{m_0}$ is infinite and this contradicts the assumption that for all $k \in \N \stm \{k_1\}$ the sets $G^k_1$ are finite. 

Since $\alpha' >\theta_3$, Lemma \ref{433mu} and Theorem \ref{thetaorderthm} \ref{thetaorder}  imply that for all $n \in \N$,
$$Z_1(E_n, \alpha') \leq Z_1(E, \alpha')<\infty.$$
Note that
$$Z_1(E_n, \alpha')=\|D \f_{k_1}\|_\infty^{\alpha'}+\sum_{j \in E_n \stm \{k_1\}} \|D  \f_j\|_\infty^{\alpha'}<\infty,$$
so \eqref{limtrun0} implies that for $n \in G_1$ large enough $Z_1(E_n, \alpha')\leq 1$. Therefore for such $n$,
$$P_{E_n}(\alpha') \leq \log Z_1 (E_n, \alpha')<0,$$
and consequently
$$\dim_{\cH} (J_{E_n}) \leq \alpha'.$$ So,
$$\alpha=\lim_{\substack{n \ra \infty\\ n \in G_1}} \dim_{\cH} (J_{E_n}) \leq \alpha'<\alpha$$
and we have a reached a contradiction. Hence there exist $k \in \N \stm \{k_1\}$ such that the sets $G^{k}_{1}$ are infinite. 

Let
$$k_2=\min \{k \in \N: k \in E_n \stm \{k_1\} \mbox{ for infinitely many }n \in G_1\},$$
and
$$G_2:=G^{k_2}_1.$$
Observe that $k_2>k_1$, because $k_2$ appears in infinitely many $E_n$ and recalling \eqref{k1def} $k_1$ is the minimal integer with that property.
Continuing inductively we obtain a collection of strictly increasing natural numbers $\{k_i\}_{i=1}^p$ (where it might happen that $p=+\infty$) and a corresponding family of infinite sets $\{G_i\}_{i=1}^p$ such that for all $i=1, \dots, p-1$
$$G_{i+1}=\{n \in G_i: k_{i+1} \in E_n\},$$
and
$$k_{i+1}=\min \{k \in \N: k \in E_n \stm \{k_1, \dots, k_i\} \mbox{ for infinitely many }n \in G_i\},$$

Let $F= \{k_i\}_{i=1}^p \cup \tilde{\Lambda}$. We will prove that
\begin{equation}
\label{dimjfeqa}
\dim_{\cH} (J_F) = \alpha.
\end{equation} 
We will first show that 
\begin{equation}
\label{dimjfleqa}
\dim_{\cH} (J_F) \leq \alpha.
\end{equation}
For any finite $l \leq p$ let
$F_l=\{k_i\}_{i=1}^l \cup \tilde{\Lambda}$. Since
$$G_1 \supset G_{2} \supset \dots \supset G_l \supset \dots,$$
we deduce that for every finite $l \leq p$ and every $j \in G_l $ we have that $F_l \subset E_j$. 
Hence for all finite $l \leq p$ there exists a strictly increasing sequence $\{j^l_m\}_{m \in \N}$ such that 
$$F_l \subset E_{j^l_m}$$
for all $m \in \N$. So for all finite $l \leq p$,
$$\dim_{\cH}(J_{F_l}) \leq \lim_{m \ra \infty}\dim_{\cH}(J_{E_{j^l_m}})=\alpha.$$
Therefore Theorem \ref{t1j97} implies that,$$\dim_{\cH} (J_F)= \sup_{l \leq p}  \{\dim_{\cH} (J_{F_l})\} \leq \alpha.$$
Thus \eqref{dimjfleqa} has been proven. 

In order to complete the proof of \eqref{dimjfeqa}, suppose by way of contradiction that 
\begin{equation}
\label{dimjflessa}
\dim_{\cH} (J_F) < \alpha.
\end{equation}
Choose $\alpha'$ such that 
$$\max\{\dim_{\cH}(J_F), \theta_3\} < \alpha' <\alpha.$$
We will first assume  that $p=\infty$. 
Let $l \in \N$ and  observe that if $\rho \in \{1,2,\dots, k_l\} \stm \{k_1,k_2,\dots ,k_l\}$ then it appears in finitely many $E_j$ for $j \in G_l$. To see this suppose by way of contradiction that there exists some $\rho \in \{1,2,\dots, k_l\} \stm \{k_1,k_2,\dots ,k_l\}$ such that the set 
$$\sharp \{j \in G_l: \rho \in E_j\}=+\infty.$$
Recall that $k_l$ is the smallest integer which appears in $E_j \stm \{k_1,\dots, k_{l-1}\}$ for infinitely many $j \in G_{l-1}$. Since $G_l \subset G_{l-1}$, $\rho$ appears in infinitely many $E_j$ for $j \in G_{l-1}$. Moreover since $\rho \not\in \{k_1,k_2,\dots, k_l\}$, we deduce that $\rho \in E_j \stm \{k_1,\dots, k_{l-1}\}$ for infinitely many $j \in G_{l-1}$. But this contradicts the minimality of $k_l$. 
 
 Hence for every $l \in \N$ there exists some $N_l$  such that for every $j \in G_l \cap [N_l,\infty)$,
\begin{equation*}
\begin{split}
E_j \stm (F \cap \{1,\dots,k_l\}) \subset E_j \stm \{k_1,\dots, k_{l}\}  \subset \{k_l+1,k_l+2,\dots\}. \\
\end{split}
\end{equation*}
Since $\{k_l\}_{l \in \N}$ is strictly increasing and $p=+\infty$, for $l$ large enough $k_l \geq \max \tilde{\Lambda}$. Therefore there exists some $l_0 \in \N$ and a strictly increasing sequence $\{j_l \}_{l \geq l_0}$ such that 
\begin{equation}
\label{incluejl}
E_{j_l} \subset (F \cap \{1,\dots,k_l\}) \cup \{k_l+1,k_l+2,\dots\}:=T_l.
\end{equation}
Since $k_l \geq \max \tilde{\Lambda}$ for all $l \geq l_0$ we deduce that $\tilde{\Lambda} \subset F \cap \{1,\dots,k_l\}$. Thus for every $A \subset E$, the set
$$F \cap \{1,\dots,k_l\} \cup A$$
is finitely irreducible.
Hence by Corollary \ref{ghenciucomb},
\begin{equation*}
\begin{split}e^{P_{T_l}(\alpha')} &\leq e^{P_{T_l \stm \{k_l+1\}}(\alpha')}+\sharp \Lambda (K \kappa_{\Lambda}^{-1})^{\alpha'} \, \max\{1, e^{p_{\Lambda} P_{T_l}(\alpha')} \} \, \|D \f_{k_l+1}\|_{\infty}^{\alpha'} \\
&\leq  e^{P_{T_l \stm \{k_l+1\}}(\alpha')}+\sharp \Lambda (K \kappa_{\Lambda}^{-1})^{\alpha'} \, \max\{1, e^{p_{\Lambda} P_{E}(\alpha')} \} \, \|D \f_{k_l+1}\|_{\infty}^{\alpha'}.
\end{split}
\end{equation*}
Since $\alpha'>\theta_3 \geq \theta$, we have that 
$$\max\{1, e^{p_{\Lambda} P_{E}(\alpha')} \}:=c_{\Lambda}<\infty.$$Continuing inductively we obtain that for every $l \geq l_0$
\begin{equation}
\label{ghenciuspec}
\begin{split}
e^{P_{T_l}(\alpha')} &\leq e^{P_{F \cap \{1,\dots,k_l\}}(\alpha')}+c_{\Lambda} \,\sharp \Lambda (K \kappa_{\Lambda}^{-1})^{\alpha'} \, \sum_{j=k_l+1}^\infty\|D \f_{j}\|_{\infty}^{\alpha'} \\
&\leq e^{P_{F }(\alpha')}+c_{\Lambda} \,\sharp \Lambda (K \kappa_{\Lambda}^{-1})^{\alpha'} \, \sum_{j=k_l+1}^\infty\|D \f_{j}\|_{\infty}^{\alpha'} .
\end{split}
\end{equation}
Since $\alpha'>\theta$, and $E=\N$ is finitely irreducible Lemma \ref{433mu} implies that 
\begin{equation}
\label{z1aprfin}
Z_1(\alpha')=\sum_{j \in \N}\|D \f_{j}\|_{\infty}^{\alpha'}<\infty.
\end{equation}
Moreover since $\alpha'>\dim_{\cH} (J_F)$ and $F$ is finitely irreducible 
\begin{equation}
\label{lfaprfin}
e^{P_F(\alpha')} <1.
\end{equation}
Since $k_l \ra \infty$, combining \eqref{incluejl}, \eqref{ghenciuspec}, \eqref{z1aprfin} and \eqref{lfaprfin} we deduce that for $l$ large enough $e^{P_{E_{j_l}}(\alpha')}<1,$ and consequently
$$\dim_{\cH} (J_{E_{j_l}}) \leq \alpha'.$$
Therefore
$$\alpha=\lim_{l \ra \infty}\dim_{\cH} (J_{E_{j_l}}) \leq \alpha'<\alpha$$
and we have reached a contradiction. Hence we proved \eqref{dimjfeqa} in the case when $p=+\infty$.

We now consider the case when $p<+\infty$. In that case the sets 
$$G_p^k=\{j \in G_{p-1}:k \in E_j\}$$
are finite for every $k \in \N \stm \{k_1,k_2,\dots,k_p\}$. Therefore
\begin{equation*}
\lim_{\substack{j \ra \infty\\ j \in G_p}} \min \{ E_j \stm \{k_1,k_2,\dots,k_p\}\}=+\infty.
\end{equation*}
The proof is identical to the proof of \eqref{limtrun0} and it is omitted. Since $F=\{k_i\}_{i=1}^p \cup \tilde{\Lambda}$ we also have that
\begin{equation}
\label{fincaseliminfty}
\lim_{\substack{j \ra \infty\\ j \in G_p}} \min \{ E_j \stm F\}=+\infty.
\end{equation}
Hence there exists a strictly increasing sequence $\{b_j\}_{j \in G_p}$ such that for all $j \in G_p$,
\begin{equation}
\label{incluej}
E_{j}  \subset F \cup \{b_j,b_j+1,\dots\}:=B_j.
\end{equation}
Using \eqref{incluej} and employing Corollary \ref{ghenciucomb} as in \eqref{ghenciuspec} we deduce that for all $j \in G_p$,
\begin{equation}
\label{ghenciuspecfin}
\begin{split}
e^{P_{E_j}(\alpha')} \leq e^{P_{B_j}(\alpha')} \leq e^{P_{F }(\alpha')}+c_{\Lambda} \,\sharp \Lambda (K \kappa_{\Lambda}^{-1})^{\alpha'} \, \sum_{i=b_j}^\infty\|D \f_{i}\|_{\infty}^{\alpha'} .
\end{split}
\end{equation}
Exactly as in the case of $p=+\infty$, \eqref{ghenciuspecfin} implies that for all $j \in G_p$ large enough
$$\dim_{\cH} (J_{E_{j}}) \leq \alpha'.$$
Which leads to a contradiction because
$$\alpha=\lim_{\substack{j \ra \infty\\ j\in G_p}}\dim_{\cH} (J_{E_{j}}) \leq \alpha'<\alpha.$$
We have thus proved \eqref{dimjfeqa1} in the case when $\alpha>\theta_3$

Now consider the case when $\dim_{\cH} (J_{\tilde{\Lambda}}) \geq \theta_3$. In this case, since always $\alpha \geq \dim_{\cH} (J_{\tilde{\Lambda}})$, we only have to consider two cases: $\alpha=\theta_3$ and $\alpha >\theta_3$. If  $\alpha=\theta_3$ then $\alpha=\dim_{\cH} (J_{\tilde{\Lambda}})$ and \eqref{dimjfeqa1} follows trivially by choosing $F= \tilde{\Lambda}$. The case $\alpha>\theta_3$ has been already settled without any extra assumption on $\dim_{\cH} (J_{\tilde{\Lambda}})$

Hence we only have to consider the case when $\dim_{\cH} (J_{\tilde{\Lambda}}) < \theta_3$. If $\alpha \in [\dim_{\cH} (J_{\tilde{\Lambda}}), \theta_3)$,   Theorem \ref{mainspect} implies that there exists some $F \supset \tilde{\Lambda}$ such that $$\dim_{\cH}(J_F)=\alpha,$$ hence \eqref{dimjfeqa1} follows.  If $\alpha>\theta_3$, \eqref{dimjfeqa1} follows by \eqref{dimjfeqa}. Therefore in order to finish the proof of \eqref{dimjfeqa1}, and thus the proof of the theorem, we are left with examining the case when $\dim_{\cH} (J_{\tilde{\Lambda}}) < \theta_3$ and $\alpha=\theta_3$. If $\dim_{\cH} (J_{\cS})=\theta_3$ we are done, so suppose that $\dim_{\cH} (J_{\cS})>\theta_3$. It suffices to construct a set $F \subset \N$ such that $\tilde{\Lambda} \subset F$ and
$$\dim_{\cH} (J_F)=\theta_3.$$
Let $n_1 \in \N$ be the smallest integer such that $$\dim_{\cH} (J_{\tilde{\Lambda} \cup \{1,\dots,n_1\}})< \theta_3$$
and $$\dim_{\cH} (J_{\tilde{\Lambda} \cup \{1,\dots,n_1, n_1+1\}}) \geq \theta_3.$$ If $\dim_{\cH} (J_{\tilde{\Lambda} \cup \{1,\dots,n_1,n_1+1\}})= \theta_3$, we are done therefore we can assume that $$\dim_{\cH} (J_{\tilde{\Lambda} \cup \{1,\dots,n_1,n_1+1\}})>\theta_3.$$ By Proposition \ref{addone} we know that there exists some $n_2>n_1+1$ such that
$$\dim_{\cH} (J_{\tilde{\Lambda} \cup \{1,\dots,n_1,n_2\}})< \theta_3.$$
By Theorem \ref{thetaorderthm} \ref{theta3equi} we know that $\theta_1(\tilde{\Lambda}) \geq \theta_3$. Recalling \eqref{theta1la} we deduce that 
$$\dim_{\cH} (J_{\tilde{\Lambda} \cup \{m, m+1,\dots\}}) \geq \theta_1(\tilde{\Lambda}) \geq \theta_3$$
for all $m \in \N$. In particular,
$$\dim_{\cH} (J_{\tilde{\Lambda} \cup \{1,\dots,n_1,n_2,n_2+1,\dots\}}) \geq \theta_3.$$
If $\dim_{\cH} (J_{\tilde{\Lambda} \cup \{1,\dots,n_1,n_2,n_2+1,\dots\}}) = \theta_3,$ we can take $F=\tilde{\Lambda} \cup \{1,\dots,n_1,n_2,n_2+1,\dots\}$ and we are done. Otherwise we choose the smallest integer $n_3>n_2$ such that
$$\dim_{\cH} (J_{\tilde{\Lambda} \cup \{1,\dots,n_1,n_2,\dots,n_3\}}) < \theta_3,$$
and $\dim_{\cH} (J_{\tilde{\Lambda} \cup \{1,\dots,n_1,n_2,\dots,n_3,n_3+1\}}) \geq \theta_3.$ Exactly as before we can assume that 
$$\dim_{\cH} (J_{\tilde{\Lambda} \cup \{1,\dots,n_1,n_2,\dots,n_3,n_3+1\}}) > \theta_3.$$ 
We continue the process inductively. For $k \geq 0$ let
$$I_k=(n_{2k}, n_{2_k}+1, \dots, n_{2k+1}) $$
where  $n_0=1$. If there exists some $l \in \N$ such that
$$\dim_{\cH} (J_{\tilde{\Lambda} \cup  I_0 \cup I_1 \dots \cup I_l \cup \{n_{2l+1}+1\}}) = \theta_3$$
the process trivially terminates because we can choose 
$$F=\tilde{\Lambda} \cup  I_0 \cup I_1 \cup \dots \cup I_l \cup \{n_{2l+1}+1\}.$$
Hence we can assume that there exists an increasing sequence of natural numbers $\{n_k\}_{k\in \N}$ such that for all $k\in \N$
\begin{equation}
\label{fdef1}
\dim_{\cH} (J_{\tilde{\Lambda} \cup I_0 \cup I_1 \cup\dots \cup I_k}) < \theta_3
\end{equation}
and
\begin{equation}
\label{fdef2}
\dim_{\cH} (J_{\tilde{\Lambda} \cup I_0 \cup I_1 \cup \dots \cup I_k \cup \{n_{2k+1}+1\}}) > \theta_3.
\end{equation}
 We define
$$F=\tilde{\Lambda} \cup \bigcup_{k=0}^\infty I_k.$$
We will show that $\dim_{\cH} (J_F)=\theta_3.$

First note that \eqref{fdef1} and Theorem \ref{t1j97} imply that 
\begin{equation}
\label{jfleq}
\dim_{\cH} (J_F) \leq \theta_3.
\end{equation}  
Let
$$F_k=\tilde{\Lambda} \cup \bigcup_{j=0}^k I_j.$$
Note that by \eqref{fdef1},
$$\dim_{\cH} (J_{F_k})<\theta_3,$$
and \eqref{fdef2} implies that for all $k \in \N$
$$P_{F_k \cup \{n_{2k+1}+1\}} (\theta_3) \geq 0.
$$
Hence by Proposition \ref{ghenciu1}, 
\begin{equation*}
\begin{split}
1 \leq e^{P_{F_k \cup \{n_{2k+1}+1\}}(\theta_3)} &\leq e^{P_{F_k}(\theta_3)}+ \sharp \Lambda (K \kappa_{\Lambda}^{-1})^{\theta_3} \, \|D \f_{n_{2k+1}+1}\|_{\infty}^{\theta_3}\\
&\leq e^{P_{F}(\theta_3)}+ \sharp \Lambda (K \kappa_{\Lambda}^{-1})^{\theta_3} \, \|D \f_{n_{2k+1}+1}\|_{\infty}^{\theta_3}.
\end{split}
\end{equation*}
Hence
$$e^{P_{F}(\theta_3)} \geq 1-\sharp \Lambda (K \kappa_{\Lambda}^{-1})^{\theta_3} \, \|D \f_{n_{2k+1}+1}\|_{\infty}^{\theta_3}.$$
Since $\{n_k\}_{k \in \N}$ is a strictly increasing sequence, \cite[Lemma  5.17]{CTU} implies that $$\|D \f_{n_{2k+1}+1}\|_{\infty}^{\theta_3} \ra 0\mbox{ as }k \ra \infty,$$ therefore $e^{P_{F}(\theta_3)} \geq 1$. Hence $P_{F}(\theta_3) \geq 0$ and consequently,
\begin{equation}
\label{jfgeq}\dim_{\cH} (J_F) \geq \theta_3.
\end{equation}
Combining \eqref{jfleq} and \eqref{jfgeq} we deduce that $\dim_{\cH} (J_F) = \theta_3$. The proof of the theorem is complete.
\end{proof}

\begin{remark}
\label{thetains}  Let $\cS=\{\f_e\}_{e \in E}$ be an infinite and finitely irreducible conformal GDMS. Notice that by the last part of the proof of Theorem \ref{compactspectrum} if $h_0<\theta_3$  then $\theta_3 \in DS(\cS)$. Therefore in that case we get that $(h_0,\theta_3] \in DS(\cS)$. Notice also that if  $\cS=\{\f_e\}_{e \in E}$ is an infinite conformal IFS then $h_0=0$, $\theta_3=\theta$ and if $\Lambda=\emptyset$, $DS_{\tilde{\Lambda}}(\cS)=DS(\cS)$. Hence Theorems \ref{mainspect} and \ref{compactspectrum} imply that $[0, \theta] \subset DS(\cS)$ improving the corresponding result from \cite{MU1} which only guarantees that $[0,\theta) \in DS(\cS)$.
\end{remark}

\begin{thm}
 \label{perfectspectrum}
 Let $\cS=\{\f_e\}_{e \in E}$ be an infinite and finitely irreducible conformal GDMS.  If $\Lambda$ is a set witnessing finite irreducibility for $E$ then $DS_{\tilde{\Lambda}}(\cS)$ is perfect.
 \end{thm}
 
 \begin{proof}  Without loss of generality we can assume that $E= \N$. Let $F \subset E$ such that $\tilde{\Lambda} \subset F$. Then $F$ is finitely irreducible. Therefore if $F$ is infinite then Theorem \ref{t1j97} implies that 
$$\dim_{\cH} (J_F)=\lim_{n \ra \infty} \dim_{\cH} (J_{F \cap \{1,\dots,n\}}).$$
For $n \in \N$ we denote $F_n= F \cap \{1,\dots,n\}$ and we let $l_0=\max \tilde{\Lambda}$. Note that if $n \geq l_0$ then the sets $F_n$ are (finitely) irreducible. Note also that for every $n \in \N$, 
\begin{equation}
\label{notisol}
\dim_{\cH} (J_{F_n})< \dim_{\cH} (J_F).
\end{equation}
In order to prove \eqref{notisol}, suppose by way of contradiction that there exists some $n_0 \in \N$ such that $$\dim_{\cH} (J_{F_{n_0}})= \dim_{\cH} (J_F).$$ Then for all $n \geq n_0$, we have that $\dim_{\cH} (J_{F_n})= \dim_{\cH} (J_F).$ Let $m=2 n_0+l_0$. In that case  $\cS_{F_m}$ is strongly regular since $F_m$ is finite, see e.g. \cite[Proposition 7.11]{CTU}. Moreover $\tilde{\Lambda} \subset F_m$ hence $F_m$  is irreducible and by Theorem \ref{4310mu} we deduce that $\dim_{\cH} (J_G)<\dim_{\cH} (J_{F_m})$ for every irreducible  $G \subsetneq F_m$ . Nevertheless $F_{n_0+l_0}$ is an irreducible subset of $F_m$ and it also contains $F_{n_0}$, therefore $\dim_{\cH} (J_{F_{{n_0+l_0}}})=\dim_{\cH} (J_{F_m})$. Thus we have reached a contradiction and \eqref{notisol} holds. Therefore if $F$ is infinite, $\dim_{\cH} (J_F)$ is not an isolated point in $DS _{\tilde{\Lambda}}(\cS)$.

Now assume that $F\subset \N$ is finite. Recall that $\tilde{\Lambda} \subset F$, hence $F$ is irreducible. Exactly as in the previous case, invoking Theorem \ref{4310mu} we have that 
$$\dim_{\cH} (J_{F \cup \{n\}})> \dim_{\cH} (J_F).$$
We will now show that 
$$\dim_{\cH} (J_F)=\lim_{n \ra \infty} \dim_{\cH} (J_{F \cup \{n\}}).$$
Let $t>\dim_{\cH} (J_F)$. Proposition \ref{ghenciu1} implies that for all $n \in \N$,
$$e^{P_{F \cup \{n\}}(t)} \leq e^{P_{F}(t)}+ \sharp \Lambda (K \kappa_{\Lambda}^{-1})^t \, \|D \f_n\|_{\infty}^t.$$
Hence by \cite[Lemma 5.17]{CTU} there exists some $n_t \in \N$ such that for all $n \geq n_t$,
$$e^{P_{F \cup \{n\}}(t)}<1.$$
Therefore for such $n$, $P_{F \cup \{n\}}(t) \leq 0$ and thus $t \geq \dim_{\cH} (J_{F \cup \{n\}})$. Hence $\dim_{\cH} (J_F)$ is not  an isolated point in $DS _{\Lambda}(\cS)$ and the proof is complete.
 \end{proof}
 Notice that Theorem \ref{topospectrumintro} is an immediate corollary of Theorems \ref{compactspectrum} and \ref{perfectspectrum} in the case of conformal iterated function systems. We restate it here for completeness.
\begin{thm}
 \label{compactspectrumcifs}
If $\cS=\{\f_n\}_{n \in \N}$ is an infinite CIFS then $DS(\cS)$ is compact and perfect. \end{thm}

While $DS_{\tilde{\Lambda}}(\cS)$ is always compact if $\Lambda$ witnesses finite irreducibility for $E$, this is not the case for $DS_1(\cS)$. This is the content of our next theorem.
 
\begin{thm} There exists a finitely irreducible infinite conformal GDMS $\cS=\{\f_e\}_{e \in E}$ such that $DS_1(\cS)$ is not compact.
\end{thm}
\begin{proof} Let $a,b,c \in \R \stm \N$ be distinct and set
$$E=\{a,b,c\} \cup \N.$$ Consider any conformal GDMS $\cS=\{V,E,A,t,i, \{X_v\}_{v \in V}, \{\f_e\}_{e \in E} \}$ where the matrix $A$ is defined by
 \begin{equation}
 \label{matrixforcount}
\begin{split}
 A_{aa}&=1, \\
 A_{bb}&=1, \\
 A_{cc}&=0,\\
 \end{split}
 \quad\quad
 \begin{split}
 A_{ab}&=1, \\
 A_{ba}&=1, \\
 A_{ca}&=0, \\
 \end{split} \quad\quad
 \begin{split}
 A_{ac}&=0, \\
 A_{bc}&=0, \\
 A_{cb}&=0, \\
 \end{split}
 \end{equation}
 and $A_{en}=A_{ne}=1$ for $e \in E$ and $n \in \N$. Note that the last condition implies that for every $n \in \N$ the sets $\{n\}$ witness finite irreducibility for $E$. See also Figure \ref{fig4}.
 \begin{figure}
\centering
\includegraphics[scale = 0.4]{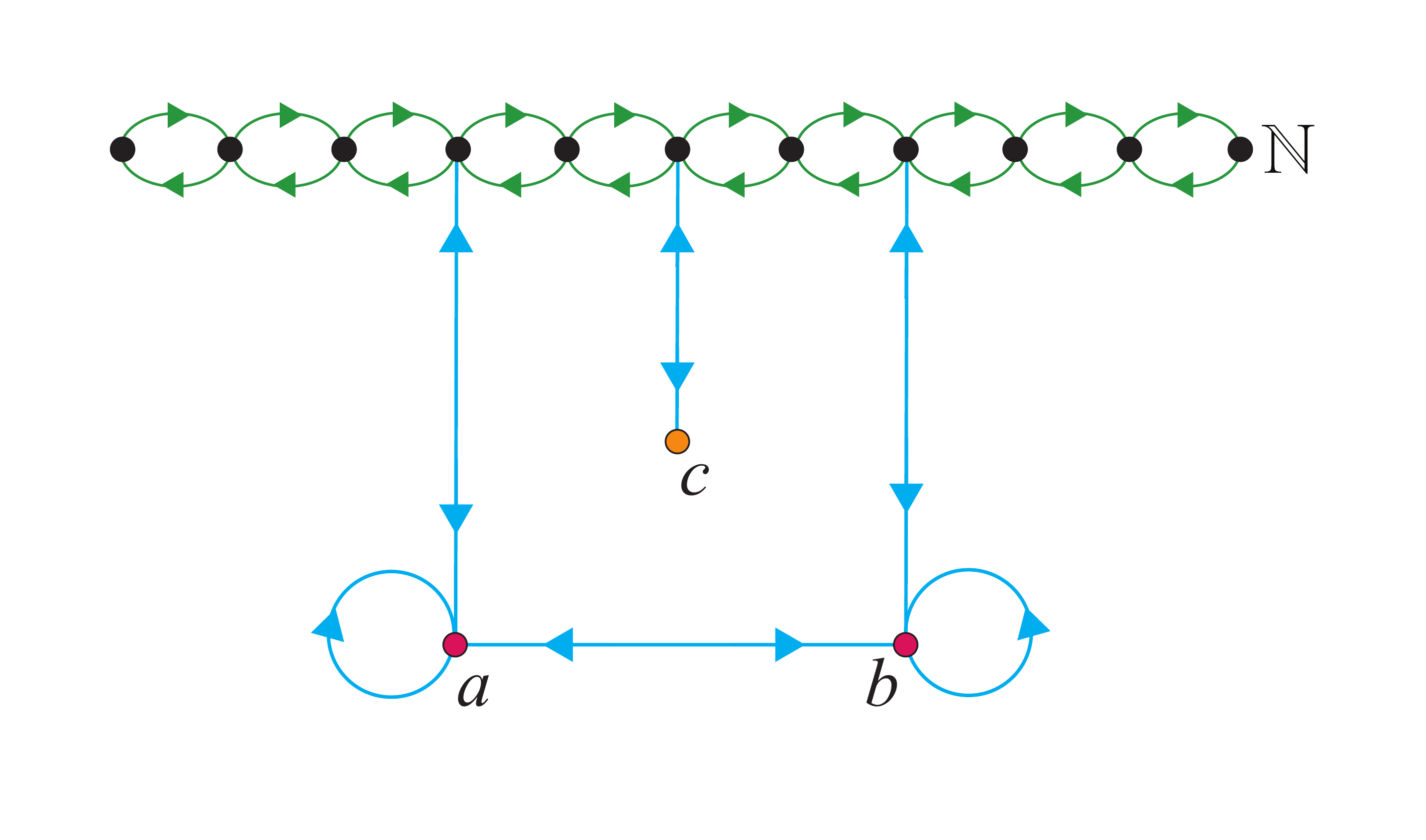}
\caption{}
\label{fig4}
\end{figure}
 
 The maps $\f_a, \f_b, \f_c$ are similarities such that 
 $$\|D \f_c\|_\infty<\|D \f_a\|_\infty=\|D \f_b\|_\infty<1$$
 and
 $$s:=\dim_{\cH}( J_{\{a,c\}})=\dim_{\cH}( J_{\{b,c\}})<\dim_{\cH} (J_{\{a,b\}}).$$
By \cite[Corollary 7.22]{CTU} if $t \in (s, \dim_{\cH}( J_{\{a,b\}}))$,
\begin{equation}
\label{dfcdfa}
\|D \f_c \|_\infty^t+\|D \f_a\|_\infty^t<1.
\end{equation}
Hence we can let $\{\f_n\}_{n \in \N}$ be any sequence of similarities such that
\begin{equation}
\label{dfmdfcdfa}
\sum_{n \in \N} \|D \f_n\|^t_\infty+\|D \f_c \|_\infty^t+\|D \f_a\|_\infty^t <1.
\end{equation}

If $\Lambda$ witness finite irreducibility for $E$ then $\tilde{\Lambda} \cap \N \neq \emptyset$. This is easy to because if $\tilde{\Lambda} \subset \{a,b,c\}$ then by the way the matrix $A$ was defined, see \eqref{matrixforcount}, there is no $ \om \in \Lambda$ such that $a\om c \in E_A^{\ast}$. Now let $\Lambda$ be any set witnessing finite irreducibility for $E$. We distinguish two cases. First assume that $\{a,b\} \nsubseteq \tilde{\Lambda}$.
Observe that by \eqref{dfmdfcdfa} and \cite[Corollary 7.20]{CTU},
$$\dim_{\cH} (J_{\{a,c\} \cup \N})\leq t <\dim_{\cH}( J_{\{a,b\}}),$$
and in the same manner (since $\|D \f_a\|_\infty=\|D \f_b\|_\infty$)
$$\dim_{\cH} (J_{\{b,c\} \cup \N})<\dim_{\cH} (J_{\{a,b\}}).$$
Therefore if $\tilde{\Lambda} \subset F$ then $\dim_{\cH}(J_F)<\dim_{\cH}( J_{\{a,b\}})$.

We now consider the case when $\{a,b\} \subset \tilde{\Lambda}$.
Moreover as in the proof of Theorem \ref{perfectspectrum}, since $\{a,b\}$ is irreducible, an application of Theorem \ref{4310mu} implies that for every $n \in \N$,
$$\dim_{\cH} (J_{\{a,b\}})<\dim_{\cH} (J_{\{a,b,n\}}).$$ Therefore if $\tilde{\Lambda} \subset F$ then $\dim_{\cH}(J_F)>\dim_{\cH}( J_{\{a,b\}})$. Hence we conclude that 
\begin{equation}
\label{notinspec}
\dim_{\cH} (J_{\{a,b\}}) \notin DS_1(\cS).
\end{equation}

Notice that $\Phi=\{a\}$ witness irreducibility for $\{a,b\}$, hence Proposition \ref{ghenciu1} implies that for $t> \dim_{\cH} (J_{\{a,b\}})$,
$$e^{P_{\{a,b,n\}}(t)} \leq e^{P_{\{a,b\}}(t)}+  (K \kappa_{\Phi}^{-1})^t \, \|D \f_n\|_{\infty}^t.$$
Since $\|D \f_n\|_\infty \ra 0$ as $n \ra \infty$, there exists some $n_0 \in \N$ such that for every $n \geq n_0$,
$$e^{P_{\{a,b,n\}}(t)} <1.$$
Therefore for $n \geq n_0$ we have that $\dim_{\cH} (J_{\{a,b,n\}})<t$. Hence
$$\lim_{n \ra \infty}\dim_{\cH} (J_{\{a,b,n\}})=\dim_{\cH} (J_{\{a,b\}}), $$
that is 
\begin{equation}
\label{closureofspec}
\dim_{\cH} (J_{\{a,b\}}) \in \overline{ DS_1 (\cS)}.
\end{equation}
Combining \eqref{notinspec} and \eqref{closureofspec} we deduce that $DS_1 (\cS)$ is not compact. The proof is complete.
\end{proof}

\section{Hausdorff dimension asymptotics of subsystems}\label{sec:hein}

Before proving the main result in this section we need to introduce some fundamental concepts of thermodynamic formalism. If $f:E_A^\N\to\R$ is a continuous function, then following \cite{MUbook} (see also the references therein), a Borel probability
measure $\tilde{m}$ on $E_A^\N$ is called a {\it Gibbs state}  for $f$ if there exist constants $Q_g\ge 1$ and $P_{\tilde{m}}\in\R$ such that for every $\om\in E_A^*$ and every
$\tau\in [\om]$
\begin{equation*}
Q_g^{-1}\le \frac{\tilde{m}([\om])}{\exp\left(S_{|\om|}f(\tau)-P_{\tilde{m}}|\om|\right)}
\le Q_g.
\end{equation*}
If additionally $\tilde{m}$ is shift-invariant, then $\tilde{m}$ is called an
{\it invariant Gibbs state}.

Let $\cS=\{\f_e\}_{e \in E}$ be a finitely irreducible conformal GDMS and fix some $t \in \R$ such that $Z_1(t)<\infty$, where, we recall, $Z_1(t)$ comes from \eqref{1_2017_12_20}. 
According to the theory developed in \cite[Chapter 2]{MUbook} and in \cite[Chapter 4]{CTU} the operator
\begin{equation}\label{1j89}
\mathcal{L}_t g(\om)= \sum_{i:\, A_{i \om_1}=1}
g(i \om)\| D\phi_i(\pi (\om))\|^t, \qquad \mbox{for $\om
  \in E^\mathbb{N}_A$.}
\end{equation}
is bounded in $C_b(E^\mathbb{N}_A)$, the Banach space of all real-valued bounded continuous functions on $E_A^\N$ endowed with the supremum norm $\|\cdot\|_\infty$.  We denote by $$\mathcal{L}_t^*: C_b^*(E^\mathbb{N}_A)\to C_b^*(E^\mathbb{N}_A)$$ 
 the dual operator for $\mathcal{L}_t$. 

We now state the following theorem comprising several results from \cite{CTU} and \cite{MUbook}. Recall that the function $\zeta:E_A^\N\ra \R$ was defined in \eqref{1MU_2014_09_10}. 

\begin{thm}\label{thm-conformal-invariant}
Suppose that $\cS=\{\f_e\}_{e \in E}$ is a finitely irreducible conformal GDMS such that $Z_1(t)<+\infty$. Then
\begin{enumerate}[label=(\roman*)]

\item There exists a unique Borel probability eigenmeasure $\tilde{m}_t$ of the conjugate Perron--Frobenius operator $\mathcal{L}_t^*$ and the corresponding eigenvalue is equal to $e^{P(t)}$.

\item The eigenmeasure $\tilde{m}_t$ is a Gibbs state for $t\zeta$.

\item The function $t\zeta:E_A^\N\to\R$ has a unique $\sigma$-invariant Gibbs state $\tilde{\mu}_t$.

\item The measure $\tilde{\mu}_t$ is ergodic, equivalent to $\tilde{m}_t$ and $\log(d\tilde\mu_t/d\tilde m_t)$ is uniformly bounded.

\item If $\int\zeta\,d\tilde{\mu}_t>-\infty$, then the $\sigma$--invariant Gibbs state $\tilde{\mu}_t$ is the unique equilibrium state for the potential $t\zeta$.

\item The Gibbs state $\tilde{\mu}_t$ is ergodic, and in case the system $\cS$ is finitely primitive, it is completely ergodic.
\end{enumerate}
\end{thm}

Finally recall that the {\it characteristic Lyapunov exponent} of a Borel probability $\sg$--invariant measure $\mu$ on $E_A^\N$ with respect to 
the conformal GDMS $\cS$, is defined as
$$
\chi_\mu(\sigma):=-\int_{E_A^\N}\zeta d\mu > 0.
$$

We shall now prove a theorem providing an effective tool for calculating the Hausdorff dimension of the limit set of any finitely irreducible and strongly regular conformal GDMS up to any desired accuracy. Indeed, it was proven in 
\cite{MUbook} and \cite{CTU} that
\begin{equation}\label{2_2017_12_20}
\dim_{\cH}(J_{\cS})=\sup\{\dim_{\cH} (J_{F})\},
\end{equation}
where the supremum is taken over all finite subsets $F$ of $E$. We will now give an explicit estimate for  $\dim_{\cH} (J_{\cS})-\dim_{\cH} (J_{F})$. This results generalizes the one for IFSs from \cite{HU} to the setting of GDMSs simultaneously giving a substantially simplified proof. We would also like to mention that nowadays there are a bunch of good and quite fast algorithms to calculate the Hausdorff dimension of the limit set of finitely irreducible conformal GDMS with high accuracy, for example \cite{mcmullen, jp1,jp2, nusscf, nuss}. Although some of these papers use a different language, their results can be easily translated to speak about GDMSs.
\begin{thm}
\label{hein}
Let $\cS=\{\f_e\}_{e \in E}$ be a finitely irreducible and strongly regular conformal GDMS  and let $\Lambda \subset E$ be a set witnessing finite irreducibility for $E$. If $F_0\subset E$ is a finite set containing $\tilde{\Lambda}$ such that $h_{F_0} \geq \theta_{\cS}$, then for all finite sets $F \supset F_0$ , we have that
$$
\dim_{\cH} (J_{\cS})-\dim_{\cH} (J_{F}) \leq  \frac{ \sharp \Lambda (K \kappa_{\Lambda}^{-1})^{h_F}}{\chi_{\tilde{\mu}_h}} \sum_{E \stm F} \, \|D \f_e\|^{h_F}_{\infty}.
$$
\end{thm}


\begin{proof} First notice that by Theorems \ref{t1j97} and \ref{4310mu} (since $\cS$ is strongly regular) always exist sets $F_0 \subset E$ such that $h_{F_0} \geq \theta$. Let $F \supset F_0$. Since convex functions are a.e. differentiable Proposition \ref{p2j85} implies that
\begin{equation}
\label{ftcpres}
e^{P(h_F)}-1=e^{P(h_F)}-e^{P(h)}=\int_h^{h_F} P'(t)e^{P(t)} dt=-\int_{h_F}^h P'(t)e^{P(t)} dt.
\end{equation}
Since $h_F \geq \theta$, we have that $P(t) \in [0,+\infty]$ for $t \in [h_F, h]$. Hence $e^{P(t)} \geq 1$ for such $t$. Recalling  Proposition \ref{p2j85} the pressure function is convex on $(\theta, +\infty)$ hence $P'$ is non-decreasing on $[h_F, h]$. Therefore 
$P'(t) \leq P'(h)$
for $t \in [h_F,h]$. Thus by \eqref{ftcpres},
\begin{equation}
\label{ftcpres1}e^{P(h_F)}-1=\int_{h_F}^h (-P'(t))e^{P(t)} dt \geq -P'(h)(h-h_F).
\end{equation}
It follows from \cite[Proposition 2.6.13]{MUbook} (for $f= \zeta$ and $\psi=0$) that $$P'(t)=-\chi_{\tilde{\mu}_t}$$
for $t \in [h_F, h]$. Hence \eqref{ftcpres1} can be rewritten as 
\begin{equation}
\label{ftcpres1}e^{P(h_F)}-1 \geq \chi_{\tilde{\mu}_h}(h-h_F).
\end{equation}
By \cite[Proposition 7.12]{CTU}, since $F$ is finite and irreducible we deduce that $\cS_F$ is strongly regular. In particular this implies that $P_F(h_F)=0$. By Corollary \ref{ghenciucomb} we then deduce that
\begin{equation}
\begin{split}
\label{hein1}e^{P(h_F)}-1&=e^{P(h_F)}-e^{P_F(h_F)} \\
&\leq   \sharp \Lambda (K \kappa_{\Lambda}^{-1})^{h_F} \, \max \{1, e^{p_{\Lambda}P_{F}(h_F)}\} \, \sum_{e \in E \stm F} \|D \f_e\|_{\infty}^{h_F} \\
&= \sharp \Lambda (K \kappa_{\Lambda}^{-1})^{h_F} \,  \sum_{e \in E \stm F} \|D \f_e\|_{\infty}^{h_F}.
\end{split}
\end{equation}
Now the proof follows by \eqref{ftcpres1} and \eqref{hein1}.
\end{proof}

\begin{rem} Note that under the assumptions of Theorem \ref{hein} we get  that for all finite sets $F \supset F_0$,
$$\dim_{\cH} (J_{\cS})-\dim_{\cH} (J_{F}) \leq  \frac{ \sharp \Lambda (K \kappa_{\Lambda}^{-1})^{\theta}}{\chi_{\tilde{\mu}_h}} \sum_{E \stm F} \, \|D \f_e\|^{h_0}_{\infty},$$
where $h_0=\dim_{\cH} (J_{F_0})$.
\end{rem}

\section{Dimension spectrum of conformal iterated functions systems}
\label{sec:dimspect}
In this section we only deal with iterated function systems, hence we will assume that $\cS=\{\f_e\}_{e \in E}$ is a conformal IFS and $E$ is a countable infinite index set. In this and the next section we will be primarily preoccupied with the question of when the spectrum $DS(\cS)$ is full, i.e. equal to $[0,h(\cS)]$.  We would like to repeat that although always (see \cite{MU1} and \cite{CTU}) $DS(\cS)\supset [0,\theta(\cS)]$ there exist IFSs, see e.g. \cite[Example~6.4]{MU2}, whose dimension spectrum  $DS$ is {\em not} full. For this reason we will also consider other relevant properties of $DS(\cS)$ that we will introduce shortly in Definition \ref{d1fs2}.

\begin{defn}
\label{d1fs2}
Let $\cS=\{\f_e\}_{e \in E}$ be a conformal IFS. We say that:
\begin{enumerate}[label=(\roman*)]
\item $\cS$ is of {\em full spectrum} if $DS(\cS)=[0,h(\cS)]$.
\item $\cS$ is of {\em cofinite full spectrum} if there exists a finite set $G \subset E$ such that $\cS_{E\stm G}$ is of full spectrum.
\item \label{SCoFS}$\cS$ is of {\em strong cofinite full spectrum} if there exists a finite set $G \subset E$ and an enumeration of $E \stm G$, say $E \stm G=\{e_i\}_{i=1}^\infty$, such that for every $n \geq 1$ the subsystem $\{\f_{e_i}\}_{i \geq n}$ is of full spectrum.
\item $\cS$ is of {\em strong full spectrum} if condition \ref{SCoFS} is satisfied with $G=\{\emptyset\}$.
\end{enumerate}
\end{defn}

Somewhat informally note that if $\cS$ is of strong full spectrum then after some enumeration each tail of $\cS$ is of full spectrum. Also observe that if $\cS$ is not strongly regular and of full spectrum, then $\cS$ is of strong full spectrum. To see this let $\cS' \subset \cS$ be a cofinite subsystem of $\cS$. Note that in that case $\theta(\cS')=h(\cS')$. If not, then there exists some $t \in (\theta(\cS'), h(\cS'))$. By Lemma \ref{433mu} we deduce that $\theta(\cS')=\theta(\cS)$.
Recall that since $\cS$ is not strongly regular Theorem \ref{4310mu} \ref{ctu(ii)} implies that $\theta(\cS)=h(\cS)$, therefore  $h(\cS')>h(\cS)$ and this is a contradiction. Hence we have shown that $\theta(\cS')=h(\cS')$ for all cofinite subsystems of $\cS$, and using Theorem \ref{mainspect} we deduce that $\cS$ is of strong full spectrum.


We stress that from now on we will endow the countable alphabet $E$ with an order of type $\N$. If $E=\{e_i\}_{i \in \N}$ is an enumeration of $E$ with respect to that order we denote
$$E_m=\{e_m,e_{m+1},e_{m+2}, \dots\}$$
and
$$I(m)=\{e_1,\dots,e_m\}$$
for $m \in \N$. Moreover if $e \in E$ then $|e|$ denotes the unique natural number such that $e=e_{|e|}$, hence $|e|$ is simply the order of $e$ with respect to the enumeration of $E$. 

\begin{defn} \label{d2fs3} Let $\cS=\{\f_e\}_{e \in E}$ be a conformal IFS.  If $F \subset E$ is nonempty and finite, then the {\em positive replacement} of $F$ is
$$F_\infty^+=(F \stm \max (F)) \cup E_{|\max(F)|+1}.$$
\end{defn}

We now state a fundamental result from \cite{KZ} which provides a sufficient condition for $t\in [0,h(\cS)]$ to belong in $DS(\cS)$.

\begin{thm}
\label{t3fs3} Let $\cS=\{\f_e\}_{e \in E}$ be a conformal IFS. If $t \in [0, h(\cS)]$ and if for every nonempty and finite subset $F \subset E$
$$P_F(t)>0 \implies P_{F_\infty^+}(t) \geq 0,$$
then $t \in DS(\cS)$.
\end{thm}
For the proof see \cite[Theorem 2.2]{KZ}.  We now provide the first technical consequence of Theorem \ref{t3fs3}. 

\begin{thm}
\label{t1fs3}
Let $\cS=\{\f_e\}_{e \in E}$ be a conformal IFS and let $t \in [0, h(\cS)]$. Assume that there exists some $d \geq 1$ such that $P_{I(d)}(t) \leq 0$. Assume also that there are two nonnegative sequences $(\alpha_k(t))_{k=d+1}^\infty$ and $(\beta_k(t) )_{k=d+1}^\infty$ with $\alpha_{d+1}(t)=0$ and the following properties:
\begin{enumerate}[label=(\roman*)]
\item \label{consa}If $F$ is a finite subset of $E$ and $e_k \notin F \cup I(d)$, then 
$$\alpha_k(t) \leq \exp (P_{F \cup \{e_k\}}(t))-\exp(P_F(t)) \leq \beta_{k}(t).$$
\item If $k \geq d+1$, then 
$$ \label{consb}\sum_{n=k+1}^\infty \alpha_n(t) \geq \beta_k(t).$$
\end{enumerate}
Then $t \in DS(\cS)$.
\end{thm}

\begin{proof} We want to apply Theorem \ref{t3fs3} and to this end we assume that 
\begin{equation}
\label{1fs4}
	P_F(t)>0.
\end{equation}Let
$$F^-=F \stm \max(F).$$
Since $P_F(t)>0$ the set $F$ must contain at least one element more than $I(d)$. Therefore $$|\max (F)|+1 \geq d+2.$$ 
Now it follows from \ref{consa} that 
$$e^{P_{\fpli}(t)} \geq e^{P_{F^{-}}(t)}+\sum_{n=|\max(F)|+1}^\infty \alpha_n(t)$$
and
$$e^{P_F(t)} \leq e^{P_{F^-}(t)}+ \beta_{|\max(F)|}(t).$$
Therefore
$$e^{P_{\fpli}(t)} \geq e^{P_F(t)}+\sum_{n=|\max(F)|+1}^\infty \alpha_n(t)- \beta_{|\max(F)|}(t).$$
Hence it follows from \ref{consb} and \eqref{1fs4} that
$$e^{P_{\fpli}(t)} \geq e^{P_F(t)}>1.$$
Equivalently $P_{\fpli}(t)>0$, and the proof follows by Theorem \ref{t3fs3}.
\end{proof}

The two following theorems are consequences of Theorem \ref{t1fs3}. 
\begin{thm}
\label{t1fs4}
Let $\cS=\{\f_e\}_{e \in E}$ be a conformal IFS. Assume that for every $t \in (0,h(\cS))$ there exist two nonnegative sequences $(\alpha_k(t))_{k=2}^\infty$ and $(\beta_k(t) )_{k=2}^\infty$ with $\alpha_{2}(t)=0$ and the following properties:
\begin{enumerate}[label=(\roman*)]
\item \label{consa4}If $F$ is a finite subset of $E$ and $e_k \notin F \cup \{e_1\}$, then 
$$\alpha_k(t) \leq \exp (P_{F \cup \{e_k\}}(t))-\exp(P_F(t)) \leq \beta_{k}(t).$$
\item If $k \geq 2$, then 
$$ \label{consb4}\sum_{n=k+1}^\infty \alpha_n(t) \geq \beta_k(t).$$
\end{enumerate}
Then $\cS$ is of strong full spectrum.
\end{thm}

Theorem \ref{t1fs4} follows immediately from Theorem \ref{t1fs3} after observing that $P_{\{e_n\}}(t) \leq 0$ for every $t \in (0, h(\cS))$ and every $n \in \N$.

\begin{thm}
\label{t1fs5}
Let $\cS=\{\f_e\}_{e \in E}$ be a strongly regular conformal IFS. Assume that there exist $q \geq 1$ and $\eta>0$ such that for every $t \in (0,\theta(\cS)+\eta)$ there are two nonnegative sequences $(\alpha_k(t))_{k=q+1}^\infty$ and $(\beta_k(t) )_{k=q}^\infty$ with $\alpha_q(t)=0$ and the following properties:
\begin{enumerate}[label=(\roman*)]
\item \label{consa5}If $F$ is a finite subset of $E_q$ and $e_k \notin F \cup \{e_q\}$, then 
$$\alpha_k(t) \leq \exp (P_{F \cup \{e_k\}}(t))-\exp(P_F(t)) \leq \beta_{k}(t).$$
\item If $k \geq q$, then 
$$ \label{consb5}\sum_{n=k+1}^\infty \alpha_n(t) \geq \beta_k(t).$$
\end{enumerate}
Then $\cS$ is of strong cofinite full spectrum.
\end{thm}

Concerning Theorem \ref{t1fs5} we note that it suffices to assume that $t \in (0, \theta(\cS)+\eta)$, rather than $t \in (0, h(\cS))$, because by \cite[Theorem 7.23]{CTU} we have that
$$\lim_{n \ra \infty} h(\cS_{E_n})=\theta_{\cS}.$$
 
As an immediate consequence of Proposition \ref{presest} and Theorems \ref{t1fs3}, \ref{t1fs4} and \ref{t1fs5} we get the following three corollaries.
\begin{corollary}
\label{c1fs6}
Let $\cS=\{\f_e\}_{e \in E}$ be a conformal IFS. If there exists some $d \geq 1$ such that $P_{I(d)}(t) \leq 0$ and for every integer $k \geq d+1$,
\begin{equation}
\label{1fs6}
\sum_{n=k+1}^\infty \|D \f_{e_n}\|^t_\infty \geq K^{2t} \|D \f_{e_k}\|^t_\infty,	
\end{equation}
then $t \in DS(\cS)$.
\end{corollary}
 
\begin{corollary}
\label{c2fs6}
Let $\cS=\{\f_e\}_{e \in E}$ be a conformal IFS. If for every $t \in (0,h(\cS))$ and every $k \geq 2$
\begin{equation*}
\sum_{n=k+1}^\infty \|D \f_{e_n}\|^t_\infty \geq K^{2t} \|D \f_{e_k}\|^t_\infty,	
\end{equation*}
then $\cS$ is of strong full spectrum.
\end{corollary}

\begin{corollary}
\label{c3fs6}
Let $\cS=\{\f_e\}_{e \in E}$ be strongly regular conformal IFS. If there exist  $q \in \N$ and $\eta>0$ such that for every $t \in [0,\theta(\cS)+\eta)$ and every integer $k \geq q$,
\begin{equation*}
\sum_{n=k+1}^\infty \|D \f_{e_n}\|^t_\infty \geq K^{2t} \|D \f_{e_k}\|^t_\infty,	
\end{equation*}
then $\cS$ is of strong cofinite full spectrum.
\end{corollary} 
Using the last corollary we can prove the following quite general theorem.

\begin{thm}
\label{t4fs6}
Let $\cS$ be a strongly regular conformal IFS.
\begin{enumerate}[label=(\roman*)]
\item \label{thenonzero} If $\theta(\cS) \neq 0$ and 
\begin{equation}
\label{liminfder}
\liminf_{n \ra \infty} \frac{\|D \f_{e_{n+1}}\|_\infty}{\|D \f_{e_n}\|_\infty}> \frac{K^2}{(1+K^{2 \theta(\cS)})^{1/\theta(\cS)}},
\end{equation}
then $\cS$ is of strong cofinite full spectrum. 
\item \label{thezero} If $\theta(\cS) = 0$ and $$\liminf_{n \ra \infty} \frac{\|D \f_{e_{n+1}}\|_\infty}{\|D \f_{e_n}\|_\infty} > 0,$$
then $\cS$ is of strong cofinite full spectrum.
\end{enumerate}
\end{thm}
\begin{proof} 
We are only going to prove \ref{thenonzero}, the proof of \ref{thezero} is similar although simpler. For simplicity of notation we let $\theta:=\theta(\cS)$. By \eqref{liminfder} there exist $q \in \N$ and $\delta>0$ such that for all $n \geq q$,
\begin{equation}
\label{liminfcons}
\frac{\|D \f_{e_{n+1}}\|^\theta_\infty}{\|D \f_{e_n}\|^\theta_\infty}>\frac{K^{2 \theta}}{1+K^{2 \theta}}+\delta.
\end{equation}
Note that by continuity there exists $\eta>0$ such that for every $t \in [\theta,\theta+\eta)$
$$\left(\frac{K^{2 \theta}}{1+ K^{2 \theta}}+\delta \right)^{t/\theta}>\frac{K^{2t}}{1+K^{2t}}.$$
Hence by \eqref{liminfcons}, for every $n \geq q$
\begin{equation}
\label{ratiodev}
\frac{\|D \f_{e_{n+1}}\|^t_\infty}{\|D \f_{e_n}\|^t_\infty}=\left(\frac{\|D \f_{e_{n+1}}\|^\theta_\infty}{\|D \f_{e_n}\|^\theta_\infty}\right)^{t/\theta}>\frac{K^{2 t}}{1+K^{2 t}}.
\end{equation}
Therefore using \eqref{ratiodev}, we get that for every $k \geq q$ and every $t \in [\theta,\theta+\eta)$
\begin{equation*}
\begin{split}
\sum_{n=k+1}^\infty \|D \f_{e_n}\|^t_\infty &\geq \|D \f_{e_k}\|^t_\infty	 \sum_{n=k+1}^\infty \left( \frac{K^{2t}}{1+K^{2t}} \right)^{n-k} \\
&= \|D \f_{e_k}\|^t_\infty	   \frac{K^{2t}}{1+K^{2t}} \left(1- \frac{K^{2t}}{1+K^{2t}}\right)^{-1} \\
&=\|D \f_{e_k}\|^t_\infty \frac{K^{2t}}{1+K^{2t}} (1+K^{2t})=K^{2t} \|D \f_{e_k}\|^t_\infty.
\end{split}
\end{equation*}
Now the proof follows by an application of Corollary \ref{c3fs6}.
\end{proof}
Theorem \ref{t4fs6} has the following two immediate consequences. First of all note that  $K^2 \leq (1+K^{2 \theta(\cS)})^{1/\theta(\cS)}$ when $\theta(\cS) \neq 0$. Therefore:
\begin{corollary}
\label{c1fs7}
Let $\cS$ be a strongly regular conformal IFS. If 
$$\liminf_{n \ra \infty} \frac{\|D \f_{e_{n+1}}\|_\infty}{\|D \f_{e_n}\|_\infty} \geq 1,$$
then $\cS$ is of strong cofinite full spectrum.
\end{corollary}
Recall also that when $\cS$ consists of similarities $K=1$. Therefore:
\begin{corollary}
\label{c1fs8}
Let $\cS$ be a strongly regular conformal IFS consisting of similarities. If 
$$\liminf_{n \ra \infty} \frac{\|D \f_{e_{n+1}}\|_\infty}{\|D \f_{e_n}\|_\infty} > 2^{-1/\theta(\cS)},$$
then $\cS$ is of strong cofinite full spectrum.
\end{corollary}

We will now introduce a special order on $E$ which will have pleasant theoretical properties and it is well suited for numerical calculations.
\begin{defn}
\label{d1fs8}
Let $E$ be an infinite countable set and let $\cS=\{\f_e\}_{e \in E}$ be a conformal IFS. We say that a well ordering $\prec$ which induces an order type of $\N$ on $E$ is {\em natural} (with respect to $\cS$) if
$$\|D \f_a\|_\infty \geq \|D \f_b\|_\infty$$
whenever $a \prec b$.
\end{defn}
If $\cS$ is as in Definition \ref{d1fs8} it follows, see e.g. \cite[Lemma 4.18]{CTU}, that there are at most finitely many letters $e \in E$ with mutually equal derivative norms $\|D \f_e\|_\infty$. Note also that the system $\cS$ admits always at least one natural ordering and it admits exactly one if and only if the function
$$E \ni e \mapsto \|D \f_e\|_\infty \in (0,1)$$
is injective.
We record the following consequences of imposing a natural order on the alphabet. In the following for any $e \in E$ in the alphabet we will denote by $e^*$ the successor of $e$ with respect to the natural order $\prec$, i.e. $e^*$ is a the least element of $E$ larger than $e$.  We will also denote by $\1$ be the first element of this order.
\begin{lm}
\label{lmc1fs9}
Let $\cS=\{\f_e\}_{e \in E}$ be a conformal IFS with its alphabet $E$ endowed with a natural order $\prec$. If there exist $t \geq 0$ and $k \in E$ such that
\begin{equation}
\label{pre1fs9}
\sum_{n \succ k^\ast} \|D \f_n\|_\infty^t \geq K^{2t} \|D \f_k\|^t_\infty,
\end{equation}
then for every $0 \leq s \leq t$
\begin{equation}
\label{pre1fs9s}
\sum_{n \succ k^\ast} \|D \f_n\|_\infty^s \geq K^{2s} \|D \f_k\|^s_\infty.
\end{equation}
\end{lm}
\begin{proof}
Since for all $n,k \in E$ such that $n \succ k$ it holds that
$$\frac{\|D\f_n\|_\infty}{\|D \f_k\|_\infty} \leq 1,$$
the function 
$$[0,+\infty)\ni s \mapsto \left( \frac{\|D\f_n\|_\infty}{\|D \f_k\|_\infty}\right)^s $$
is monotone decreasing. Therefore the function
$$[0,+\infty)\ni s \mapsto \sum_{n \succ k^\ast}\left( \frac{\|D\f_n\|_\infty}{\|D \f_k\|_\infty}\right)^s $$
is also monotone decreasing. Since on the other hand the function $$[0,+\infty)\ni s \mapsto K^s$$ is monotone increasing it follows from \eqref{pre1fs9} that for every $s \in [0,t]$,
$$\sum_{n \succ k^\ast}\left( \frac{\|D\f_n\|_\infty}{\|D \f_k\|_\infty}\right)^s \geq \sum_{n \succ k^\ast}\left( \frac{\|D\f_n\|_\infty}{\|D \f_k\|_\infty}\right)^t \geq K^{2t} \geq K^{2s}.$$
\end{proof}
\begin{corollary}
\label{c1fs9}
Let $\cS=\{\f_e\}_{e \in E}$ be a conformal IFS with its alphabet $E$ endowed with a natural order $\prec$. Let $t \geq 0$. If  for every $k \succ \1^{\ast}$
\begin{equation}
\label{1fs9}
\sum_{n \succ k^\ast} \|D \f_n\|_\infty^t \geq K^{2t} \|D \f_k\|^t_\infty,
\end{equation}
then $$[0, \min\{t,h(\cS)\}] \subset DS(\cS).$$
\end{corollary}
\begin{proof}
The proof follows by Lemma \ref{lmc1fs9} and Corollary \ref{c1fs6} after noting that $P_{\{\1\}}(s) \leq 0$ for all $s \geq 0$ because trivially $h=\dim_{\cH} (J_{\{\1\}})=0$.
\end{proof}
We also record the following corollary.
\begin{corollary}
\label{c1fs9.1}
Let $\cS=\{\f_e\}_{e \in E}$ be a conformal IFS with its alphabet $E$ endowed with a natural order $\prec$. If there exists $t \geq h(\cS)$ such that for every $k \succ 1^{\ast}$
\begin{equation*}
\sum_{n \succ k^\ast} \|D \f_n\|_\infty^t \geq K^{2t} \|D \f_k\|^t_\infty,
\end{equation*}
then $\cS$ is of strong full spectrum.
\end{corollary}
We conclude this section by noting that Corollary \ref{c1fs6} and Lemma \ref{lmc1fs9} imply the following proposition.
\begin{propo}
\label{keypropoint}
Let $\cS=\{\f_e\}_{e \in E}$ be a conformal IFS. Let $E=\{e_n\}_{n \in \N}$ be an enumeration of $E$ according to a natural order. Suppose that there exist $t \leq s \leq h(\cS)$ and $d \in \N$ such that 
\begin{enumerate}[label=(\roman*)]
\item $P_{I(d)}(t) \leq 0$,
\item $\sum_{n \geq k+1} \|D \f_{e_n}\|_\infty^s \geq K^{2s} \|D \f_{e_k}\|^s_\infty$ for all $k \geq d+1$.
\end{enumerate}
Then $[t,s] \subset DS(\cS)$.
\end{propo}

\section{Dimension spectrum of complex continued fractions}\label{sec:fullcf}



In this section we study the dimension spectrum of the conformal iterated function system generated by complex continued fractions. Let
\begin{enumerate}[label=(\roman*)]
\item $E=\{m+ni:(m,n) \in \N \times \Z \},$
\item $X=\bar{B}(1/2,1/2)$,
\item $W=B(1/2,3/4)$.
\end{enumerate}
For $e \in E$ we define the maps $\f_e: W \ra W$ by 
\begin{equation*}
\f_e(z)=\frac{1}{e+z}.
\end{equation*}
It was proved in \cite{MU1} that for every $e \in E$
\begin{enumerate}[label=(\roman*)]
\item $\f_e(W) \subset B(0, 4 |e|^{-1})$,
\item $4^{-1} |e|^{-2} \leq |\f'_e(z)|\leq 4 |e|^{-2},$ for all $z \in W$,
\item \label{derestcf}$4^{-1} |e|^{-2} \leq \diam (\f_e (W)) \leq 4 |e|^{-2}$.
\end{enumerate}
We record that the best distortion constant is  $K=4$, see \cite[Remark 6.7]{MU1}. Note that formally $\{\f_e \}_{e \in E}$ is {\em not} a conformal IFS because $\f'_1(0)=1$. Nevertheless the family $\{ \f_e \circ \f_j:(e,j) \in E \times E\}$ is indeed a conformal IFS. Slightly abusing notation we will treat $\cf:=\{\f_e\}_{e \in E}$ as a conformal IFS and we will call it, the {\em complex continued fractions} IFS. Notice that \ref{derestcf} and Proposition \ref{p2j85} imply that $\theta(\cf)=1$.

It was proved in \cite[Theorem 6.6]{MU1} that for all $e=m+ni \in E$,
\begin{equation}
\label{1fs10}
\| \f'_e\|_{\infty}=\sup_{z \in X} |\f'_e(z)|=\frac{1}{(|e+1/2|-1/2)^2}=\frac{1}{\left(((m+1/2)^2+n^2)^{1/2}-1/2\right)^2}.
\end{equation}
In fact, this formula can be seen quite easily geometrically. Indeed, we have 
$$
|\f_e'(z)|=\frac1{|e+z|^2}.
$$
To find $\|\f'_e\|_{\infty}$ thus reduces to find the infimum of $|e+z|$ if $z$ ranges over $\overline B(1/2,1/2)$. This is of course the same as the 
infimum of $|z|$ if $z$ ranges over $\overline B(e+1/2,1/2)$. And this of course is $|e+1/2|-1/2$.

We record that  \eqref{1fs10} implies that
$
\| \f'_e\|_{\infty}=\| \f'_{\bar{e}}\|_{\infty}.
$ 
This also can be seen geometrically by noticing that 
$$
\f_{\bar e}=J\circ\f_e\circ J,
$$
where $J:\C\to\C$ is the complex conjugacy map ($J(z)=\bar z$), thus an isometry.

We endow the alphabet $E$ a with natural ordering $``\prec"$ (recall Definition \ref{d1fs8}) and we denote by $\1$ be the first element of this order. By \eqref{1fs10} it is immediate that $\1=1$. Recall that for any $e \in E$  we will denote by $e^*$ the successor of $e$.
\begin{rem} 
\label{o1fs10}
If the alphabet $E$ is endowed with the natural ordering $``\prec"$, then for every $e=m+ni \in E$ and every $k \in \N$
$$e^{\ast} \prec \min\{e+k,\bar{e}+k,m+|n| i+ki,m-|n|i-ki\}.$$
\end{rem}
We now state and prove our first theorem regarding the dimension spectrum of complex continued fractions.
\begin{thm}
\label{ccfscofs}The complex continued fractions IFS $\cf$ is of strong cofinite full spectrum.
\end{thm}
\begin{proof} Observe that by \eqref{1fs10} it follows easily that for any $e \in E$,
$$e^{\ast} \prec e+1.$$
Therefore
$$\|\f'_{e}\|_\infty \geq \|\f'_{e^\ast}\|_\infty \geq \|\f'_{e+1}\|_\infty,$$
and consequently
\begin{equation}
\label{ratioccf}
1 \geq \frac{\|\f'_{e^\ast}\|_\infty}{\|\f'_{e}\|_\infty} \geq \frac{\|\f'_{e+1}\|_\infty}{\|\f'_{e}\|_\infty}.
\end{equation}
Let $(e_n)_{n \in \N}$ be an enumeration of $E$ respecting the natural order $\prec$, i.e. $(e_n)^{\ast}=e_{n+1}$  for all $n \in \N$. It then follows by \eqref{ratioccf} and \eqref{1fs10} that
$$\lim_{n \ra \infty} \frac{\|\f'_{e_{n+1}}\|_\infty}{\|\f'_{e_{n}}\|_\infty}=1.$$
By \cite[Theorem 3.20 and Proposition 6.1]{MU1} we know that $\cf$ is strongly regular, hence Corollary \ref{c1fs7} implies that $\cf$ is of strong cofinite full spectrum.
\end{proof}
We will now state and prove two technical lemmas.
\begin{lm}
\label{amonoton}
Given $m \geq 0$ and $t>0$, the function
$$\N \cup \{0\} \ni n \ra \frac{\sum_{k=m+1}^\infty \left(((k+1/2)^2+n^2)^{1/2}-1/2 \right)^{-2t}}{\left( ((m+1/2)^2+n^2)^{1/2}-1/2\right)^{-2t}}$$
is increasing.
\end{lm}
\begin{proof} Let 
$$A(n) =\sum_{k=m+1}^\infty \left( \frac{ ((m+1/2)^2+n^2)^{1/2}-1/2}{((k+1/2)^2+n^2)^{1/2}-1/2 }\right)^{2t}.$$
It suffices to show that for every $k \geq m+1$ the functions
$$n \ra A_k(n):= \frac{ ((m+1/2)^2+n^2)^{1/2}-1/2}{((k+1/2)^2+n^2)^{1/2}-1/2 }$$
are increasing. Treating $n$ as a real (continuous) variable, and taking derivatives we have that
\begin{equation*}
\begin{split}
\frac{d A_k(n)}{d\,n}&=\frac{n\left((m+1/2)^2+n^2\right)^{-1/2}\left(((k+1/2)^2+n^2)^{1/2} -1/2\right)}{\left(((k+1/2)^2+n^2)^{1/2}-1/2 \right)^2}\\
&\quad-\frac{n\left((k+1/2)^2+n^2\right)^{-1/2}\left(((m+1/2)^2+n^2)^{1/2} -1/2\right)}{\left(((k+1/2)^2+n^2)^{1/2}-1/2 \right)^2}.
\end{split}
\end{equation*}
Therefore the sign of $\frac{d A_k(n)}{d\,n}$ coincides with the sign of 
$$\a_k(n):=\frac{((k+1/2)^2+n^2)^{1/2} -1/2}{\left((m+1/2)^2+n^2\right)^{1/2}}-\frac{((m+1/2)^2+n^2)^{1/2} -1/2}{\left((k+1/2)^2+n^2\right)^{1/2}}.$$
Let
$$b_k:=((k+1/2)^2+n^2)^{1/2} \mbox{ and }a=((m+1/2)^2+n^2)^{1/2}.$$  
Then 
$$\a_k(n)=\frac{b_k-1/2}{a}-\frac{a-1/2}{b_k}=\frac{b_k(b_k-1/2)-a(a-1/2)}{ab_k}.$$
Since $b_k>a \geq 1/2$ for all $k \geq m+1$ we deduce that $\a_k(n) \geq 0$ for such $k$ and the proof is complete. 
\end{proof}

\begin{lm}
\label{bmonoton}
Given $n \geq 0$ and $t>0$, the function
$$\N \cup \{0\} \ni m \ra \frac{\sum_{l=n+1}^\infty \left(((m+1/2)^2+l^2)^{1/2}-1/2 \right)^{-2t}}{\left( ((m+1/2)^2+n^2)^{1/2}-1/2\right)^{-2t}}$$
is increasing.
\end{lm}
\begin{proof} Let 
$$B(m) =\sum_{l=n+1}^\infty \left( \frac{ ((m+1/2)^2+n^2)^{1/2}-1/2}{((m+1/2)^2+l^2)^{1/2}-1/2 }\right)^{2t}.$$
It suffices to show that for every $l \geq n+1$ the functions
$$m \ra B_l(m):= \frac{ ((m+1/2)^2+n^2)^{1/2}-1/2}{((m+1/2)^2+l^2)^{1/2}-1/2 }$$
are increasing. Treating $m$ as a real (continuous) variable, and taking derivatives we have that
\begin{equation*}
\begin{split}
\frac{d B_l(m)}{d\,m}&=\frac{\left((m+1/2)^2+n^2\right)^{-1/2}\left(((m+1/2)^2+l^2)^{1/2} -1/2\right)(m+1/2)}{\left(((m+1/2)^2+l^2)^{1/2}-1/2 \right)^2}\\
&\quad-\frac{\left((m+1/2)^2+l^2\right)^{-1/2}\left(((m+1/2)^2+n^2)^{1/2} -1/2\right)(m+1/2)}{\left(((m+1/2)^2+l^2)^{1/2}-1/2 \right)^2}.
\end{split}
\end{equation*}
Therefore the sign of $\frac{d B_l(m)}{d\,m}$ coincides with the sign of 
$$\beta_l(m):=\frac{((m+1/2)^2+l^2)^{1/2} -1/2}{\left((m+1/2)^2+n^2\right)^{1/2}}-\frac{((m+1/2)^2+n^2)^{1/2} -1/2}{\left((m+1/2)^2+l^2\right)^{1/2}}.$$
Let
$$c_l:=((m+1/2)^2+l^2)^{1/2} \mbox{ and }a=((m+1/2)^2+n^2)^{1/2}.$$  
Then 
$$\beta_l(m)=\frac{c_l-1/2}{a}-\frac{a-1/2}{c_l}=\frac{c_l(c_l-1/2)-a(a-1/2)}{ac_l}.$$
Since $c_l>a \geq 1/2$ for all $l \geq n+1$ we deduce that $\beta_l(m) \geq 0$ for such $l$ and the proof is complete. 
\end{proof}
We now make the following technical observation which incorporates the previous monotonicity lemmas.

\begin{lm}
\label{4rays}
Let $e=m+ni \in E$ and $ t >0$. Then
$$\sum_{j \succ e^{\ast}} \frac{ \|\f'_j\|^t_\infty}{\|\f'_e\|^t_\infty} \geq 2\frac{\sum_{k=m+1}^\infty k^{-2t}}{m^{-2t}}+2 \frac{\sum_{l=n+1}^\infty ((l^2+1/4)^{1/2}-1/2)^{-2t}}{((n^2+1/4)^{1/2}-1/2)^{-2t}}.$$
\end{lm}
\begin{proof}
Recalling  Remark \ref{o1fs10} we get that if $k \geq m+1$ then
$$k+ni \succ e^{\ast}\mbox{ and }k-ni \succ e^{\ast},$$
and if $l \geq |n|+1$,
$$m+li \succ e^{\ast} \mbox{ and }m-li\succ e^{\ast}.$$
Therefore
\begin{equation*}
\begin{split}
\sum_{j \succ e^{\ast}}  \|\f'_j\|^t_\infty &\geq \sum_{k=m+1}^\infty \|\f'_{k+ni}\|^t+ \sum_{k=m+1}^\infty \|\f'_{k-ni}\|^t+\sum_{l=|n|+1}^\infty \|\f'_{m+li}\|^t_\infty+\sum_{l=|n|+1}^\infty \|\f'_{m-li}\|^t_\infty \\
&=2\sum_{k=m+1}^\infty \|\f'_{k+ni}\|^t+2\sum_{l=|n|+1}^\infty \|\f'_{m+li}\|^t_\infty.
\end{split}
\end{equation*}
Now by Lemmas \ref{amonoton} and \ref{bmonoton} we get that
\begin{equation*}
\begin{split}
\sum_{j \succ e^{\ast}} \frac{ \|\f'_j\|^t_\infty}{\|\f'_e\|^t_\infty} &\geq 2\sum_{k=m+1}^\infty \frac{(((k+1/2)^2+n^2)^{1/2}-1/2)^{-2t}}{(((m+1/2)^2+n^2)^{1/2}-1/2)^{-2t}} \\
&\quad\quad+2\sum_{l=|n|+1}^\infty \frac{(((m+1/2)^2+l^2)^{1/2}-1/2)^{-2t}}{(((m+1/2)^2+n^2)^{1/2}-1/2)^{-2t}}\\
&\geq 2 \frac{\sum_{k=m+1}^\infty k^{-2t}}{m^{-2t}}+2 \frac{\sum_{l=|n|+1}^\infty ((l^2+1/4)^{1/2}-1/2)^{-2t}}{((n^2+1/4)^{1/2}-1/2)^{-2t}}.
\end{split}
\end{equation*}
\end{proof}
\begin{propo}
\label{outofnumbox}
Let $e=m+ni \in E$ such that $m \geq 2^9$ or $|n| \geq 2^{12}$ then 
\begin{equation}
\label{condoutbox}
\sum_{j \succ e^{\ast}}^\infty \| \f'_j\|^h_\infty \geq 16^h \| \f'_e\|^h_\infty
\end{equation}
where $h=\dim_{\cH} (J_{\cf})$.
\end{propo}
\begin{proof} 
Assume first that $m \geq 2^9$. By the integral test,
\begin{equation}
\label{mbox}
\begin{split}
m^{4} \sum_{k=m+1}^\infty k^{-4} &\geq m^{4} \int_{m+1}^\infty x^{-4} dx\\
&= m^{4}\frac{(m+1)^{-3}}{3} \geq \frac{1}{3} \left( \frac{m}{m+1}\right)^{4} m.
\end{split}
\end{equation}
Therefore by \eqref{mbox},
$$2 \,m^{4} \sum_{k=m+1}^\infty k^{-4} \geq \frac{2}{3} \left( \frac{2^9}{2^9+1} \right)^4 2^9 \geq \frac{1}{3} \frac{2^{46}}{2^{36}+2^{32}}.$$
Since $2^{6} \geq 3 (2^{4}+1)$ we get that
$$\frac{1}{3} \frac{2^{46}}{2^{36}+2^{32}} \geq 2^8.$$
Hence we have shown that
$$2 \,m^{4} \sum_{k=m+1}^\infty k^{-4} \geq 2^8,$$ 
which by Lemma \ref{4rays} implies that
$$\sum_{j \succ e^{\ast}} \frac{ \|\f'_j\|^2_\infty}{\|\f'_e\|^2_\infty} \geq 16^2.$$
By \cite[Theorem 6.6]{MU1} we know that 
\begin{equation}
\label{hausccf}
h=\dim_{\cH} (J_{\cf}) \leq 1.9,
\end{equation}
therefore Lemma \ref{lmc1fs9} implies that \eqref{condoutbox} holds.

Now assume that $n \geq 2^{12}$. By the integral test,
$$\sum_{l=n+1}^\infty ( (l^2+1/4)^{1/2}-1/2)^{-4} \geq \sum_{l=n+1}^\infty l^{-4} \geq  \frac{(n+1)^{-3}}{3} .$$
On the other hand it is easy to check that
$$(n^2+1/4)^{1/2}-1/2 \geq \frac{\sqrt{5}-1}{2}n.$$
Therefore,
$$2 \frac{\sum_{l=n+1}^\infty ((l^2+1/4)^{1/2}-1/2)^{-4}}{((n^2+1/4)^{1/2}-1/2)^{-4}} \geq 2 \left( \frac{\sqrt{5}-1}{2}\right)^{4} \left(\frac{n}{n+1} \right)^{4} \frac{n+1}{3}.$$
As in the previous case, recalling Lemma \ref{lmc1fs9} and Lemma \ref{4rays}, it suffices to show that for $n\geq 2^{12}$, 
$$\frac{2}{3} \left( \frac{\sqrt{5}-1}{2}\right)^{4} \left(\frac{n}{n+1} \right)^{4} n \geq 16^{2}. $$
This reduces to checking the condition
$$\frac{2}{3} \left( \frac{\sqrt{5}-1}{2}\right)^{4}\frac{2^{11}}{2^7+1}  \geq  1,$$
which is easily verified to be valid. 

The only remaining case is when $n \leq -2^{12}$. The proof follows as in the previous case using that $\| \f'_z\|=\|\f'_{\bar{z}}\|$ and Lemma \ref{4rays}. The proof is complete.
\end{proof}
We are now ready to prove Theorem \ref{fullspec}, which we restate for convenience.
\begin{thm} 
\label{full7}
The conformal iterated function system associated to the complex continued fractions has full spectrum.
\end{thm}
\begin{proof}
By Proposition \ref{outofnumbox} we know that \eqref{condoutbox} holds for any $e=m+ni \in E$ such that $m \geq 2^9$ or $|n| \geq 2^{12}$. We use \texttt{Matlab} to check condition \eqref{condoutbox} for the finitely many remaining points in the grid $E$. Let 
$$E=\{e_1,e_2, e_3,\dots\}$$ be an enumeration of the grid $E$ according to the natural order $\prec$ introduced in Definition \ref{d1fs8}. Moreover for $k \in \N$ we let
$$D_k=\{p_1,\dots, p_k\}$$
be an enumeration of the first $k$ points in the nonnegative grid $\N \times \Z^+$ according to the natural order $\prec$. Here as usual $\Z^+$ denotes the set of nonnegative integers. For any $A \subset \C$ we let 
$$\tilde{A}=\{z \in \C: z \in A \mbox{ or } \overline{z} \in A\}.$$
For $k \in \N$ we are going to denote
$$\tilde{I}_k=\tilde{D_k}$$
Recalling Section \ref{sec:dimspect}, since $\| \f'_z\|=\|\f'_{\bar{z}}\|$ for all $z \in E$, it follows that for every $k \in \N$ there exists some $n(k) \in \N$ such that
$$\tilde{D_k}=I_{n(k)},$$
this simply means that $\tilde{I}_k$ is an initial block of $E$.
\begin{figure} 
\centering
\includegraphics[scale=0.265]{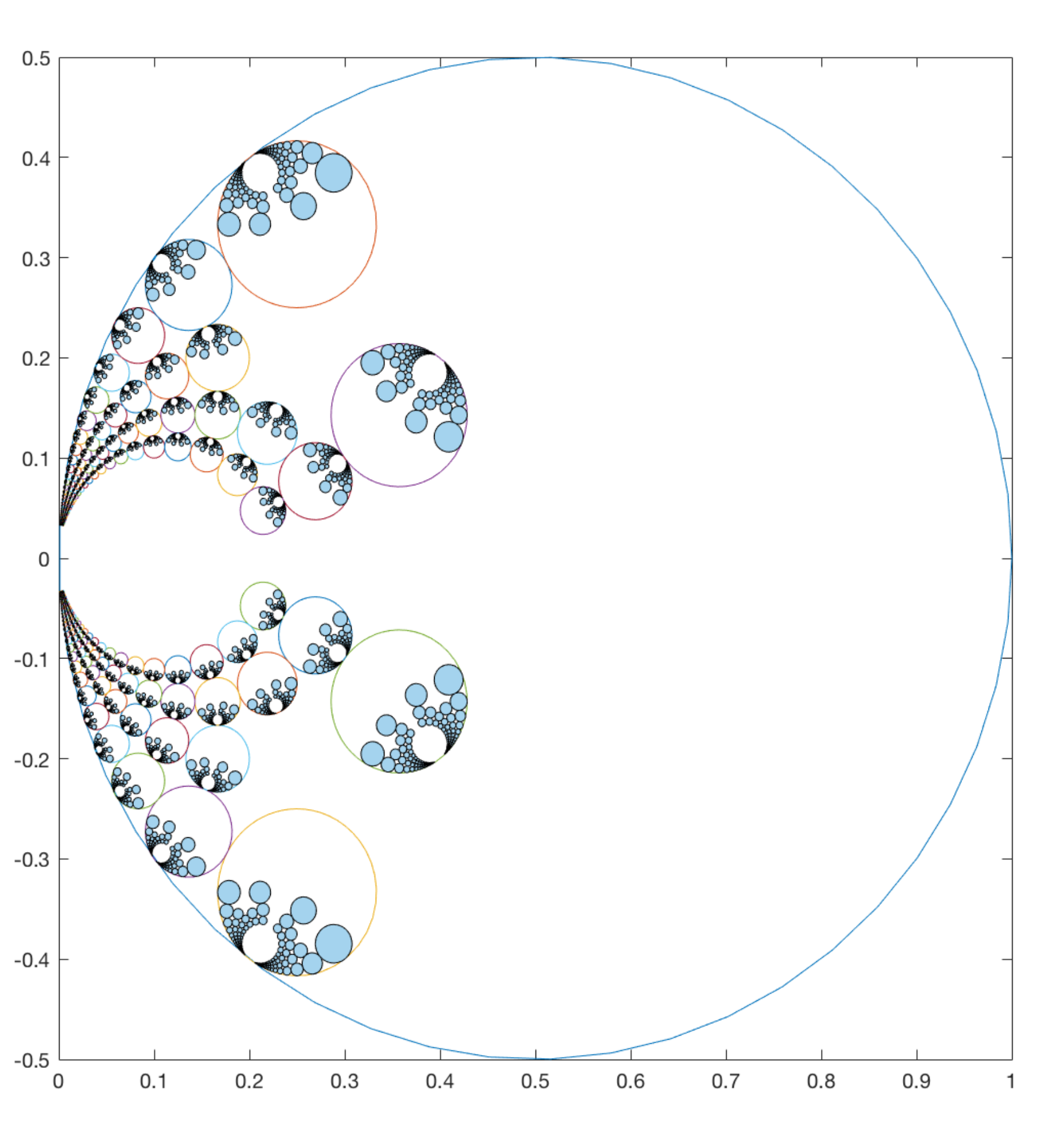}
\includegraphics[scale=0.265]{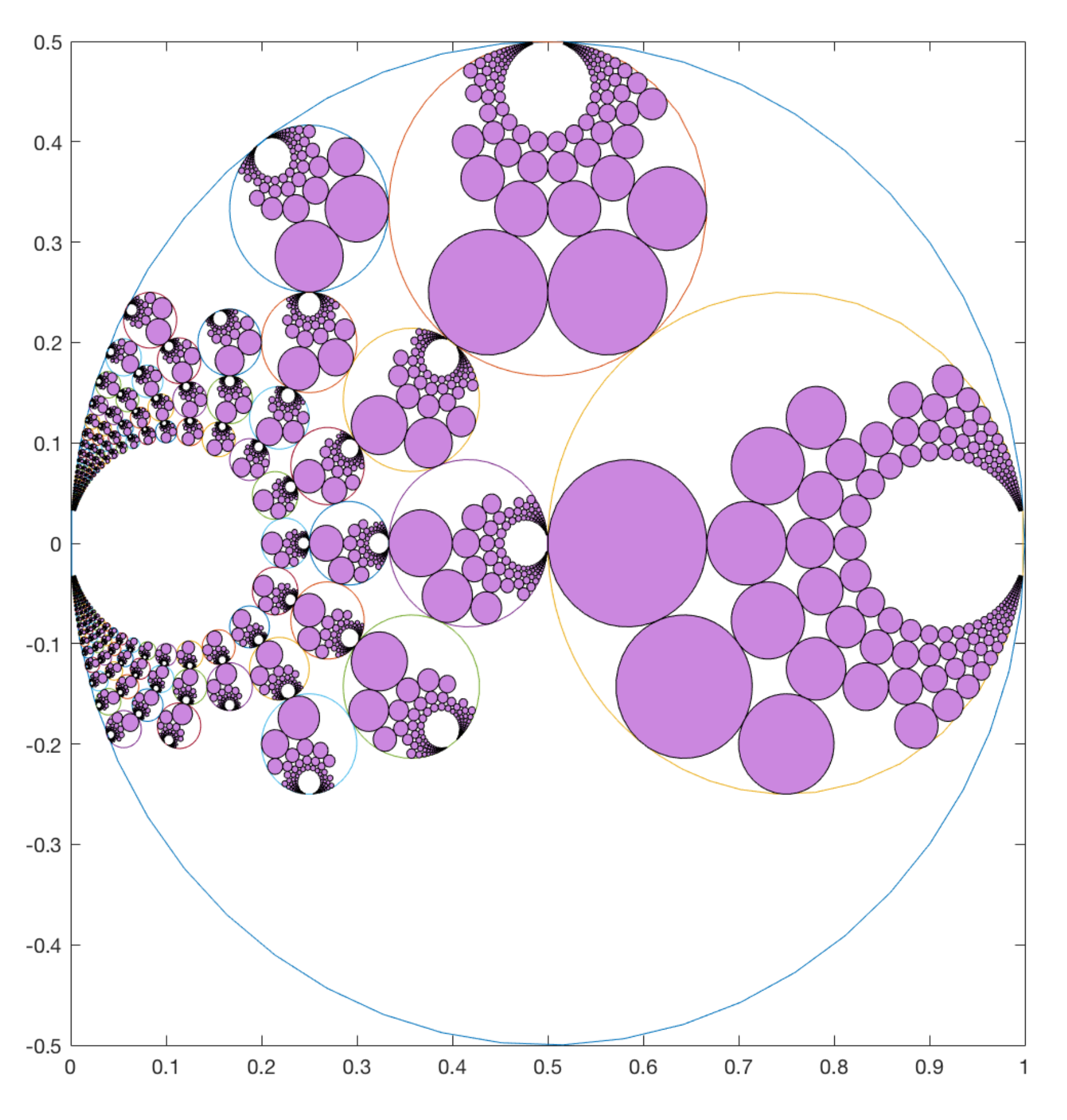}
\includegraphics[scale=0.265]{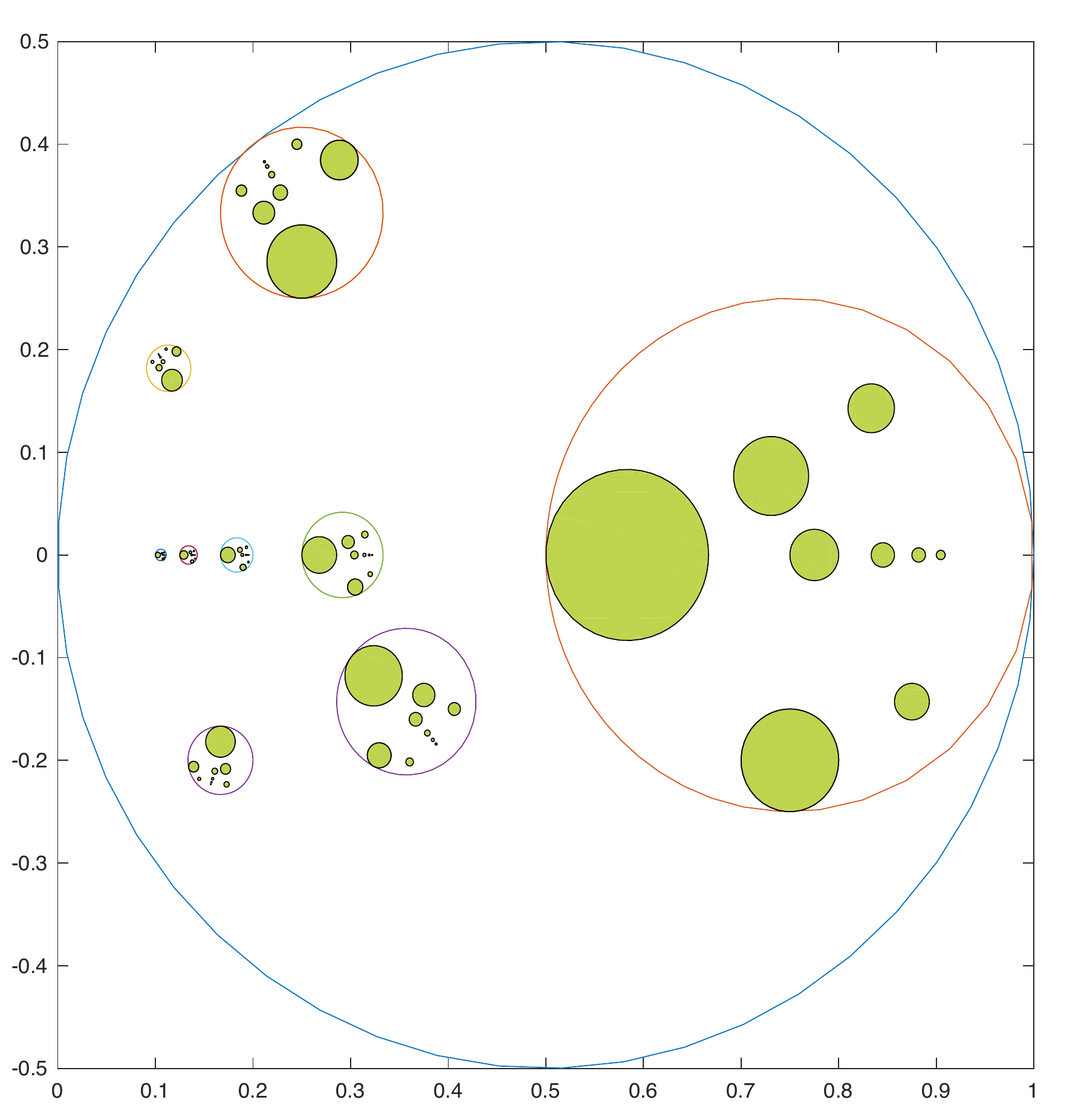}
\caption{Approximation of the limit sets of the three subsystems of $\cf$ after two iterations. See also Figure \ref{figccf} for an approximation of $J_{\cf}$.}
\end{figure}
Consider the following function $f: \N \times \R^{+} \ra \R^{+}$
$$f(k,t):=\frac{\sum_{n=k+1}^{\infty} \|\f_{e_n}'\|^t_\infty}{\|\f_{e_k}'\|^t_\infty}.$$
Notice that
$$f(k,t)=\frac{\|\f_{e_{k+1}}'\|^t_\infty}{\|\f_{e_k}'\|^t_\infty}(1+ f(k+1,t)),$$
and consequently we obtain the following recursive formula
\begin{equation}
\label{recurs}
f(k+1,t)=\alpha(k)^{-t} f(k,t)-1,
\end{equation}
where $\alpha(k)=\frac{\|\f_{e_{k+1}}'\|_\infty}{\|\f_{e_k}'\|_\infty}$. 

Let 
$$B =\{ z \in \C: \Rea (z) \in [1,2^9] \mbox{ and  }\Imm(z) \in [0, 2^{12}] \}.$$ Given $t>0$ we would like to check $k \in \N$
\begin{equation}
\label{reformcondoutbox}
e_k \in B \cap E \mbox { and }f(k,t) \geq 16^t.
\end{equation}
Observe that if $\tilde{f}(k,t)$ is an approximation of $f(k,t)$ such that $\tilde{f}(k,t) \leq f(k,t)$ then 
$$\tilde{f}(k+1,t) :=\alpha(k)^{-t} \tilde{f}(k,t)-1 \leq \alpha(k)^{-t} f(k,t)-1 = f(k+1,t).$$
Now let 
$$B_0=\{ z \in \C: \Rea (z) \in [1,40000] \mbox{ and  }\Imm(z) \in [0, 40000] \}$$
and set 
\begin{equation}
\label{apr38}
\tilde{f}(1,t):=\frac{\sum_{\{n \geq 2 :e_n \in B_0\}}  \|\f_{e_n}'\|^t_\infty}{\|\f_{e_1}'\|^t_\infty}.
\end{equation}
Obviously $\tilde{f}(1,t) \leq f(1,t)$. We employ \texttt{Matlab} to estimate explicitly the quantity $\tilde{f}(1,t)$, and using the recursive formula \eqref{recurs} we derive approximations $\tilde{f}(k,t) \leq f(k,t)$ for all $e_k \in B$. We remark that the choice of $40000$ in the definition of $B_0$ does not have any particular significance; it was simply sufficient for our calculations.

As a starting exponent $t$ we will choose an upper bound for $\dim_{\cH} (J_{\cf})$. It was proved in \cite{MU1} that
\begin{equation}
h:=\dim_{\cH} (J_{\cf}) \leq 1.885.
\end{equation}
We also record that by a recent result of Priyadarshi \cite{pri},
\begin{equation}
\label{prifull}
h \geq 1.825.
\end{equation}
In particular for $\bar{h}:=1.885$ we obtain that 
\begin{equation*}
f(k,\bar{h}) \geq 16^{\bar{h}} \mbox{ for all }e_k \in B \stm \tilde{I}_{16}.
\end{equation*}
Hence by Lemma \ref{lmc1fs9} and Proposition \ref{outofnumbox} we deduce that 
\begin{equation}
\label{firstintsum}
\frac{\sum_{n=k+1}^{\infty} \|\f_{e_n}'\|^h_\infty}{\|\f_{e_k}'\|^h_\infty} \geq 16^h \mbox{ for all }e_k \in E \stm \tilde{I}_{16}.
\end{equation}
\begin{figure}
\label{fig:3grid}
\centering
\includegraphics[scale = 0.17]{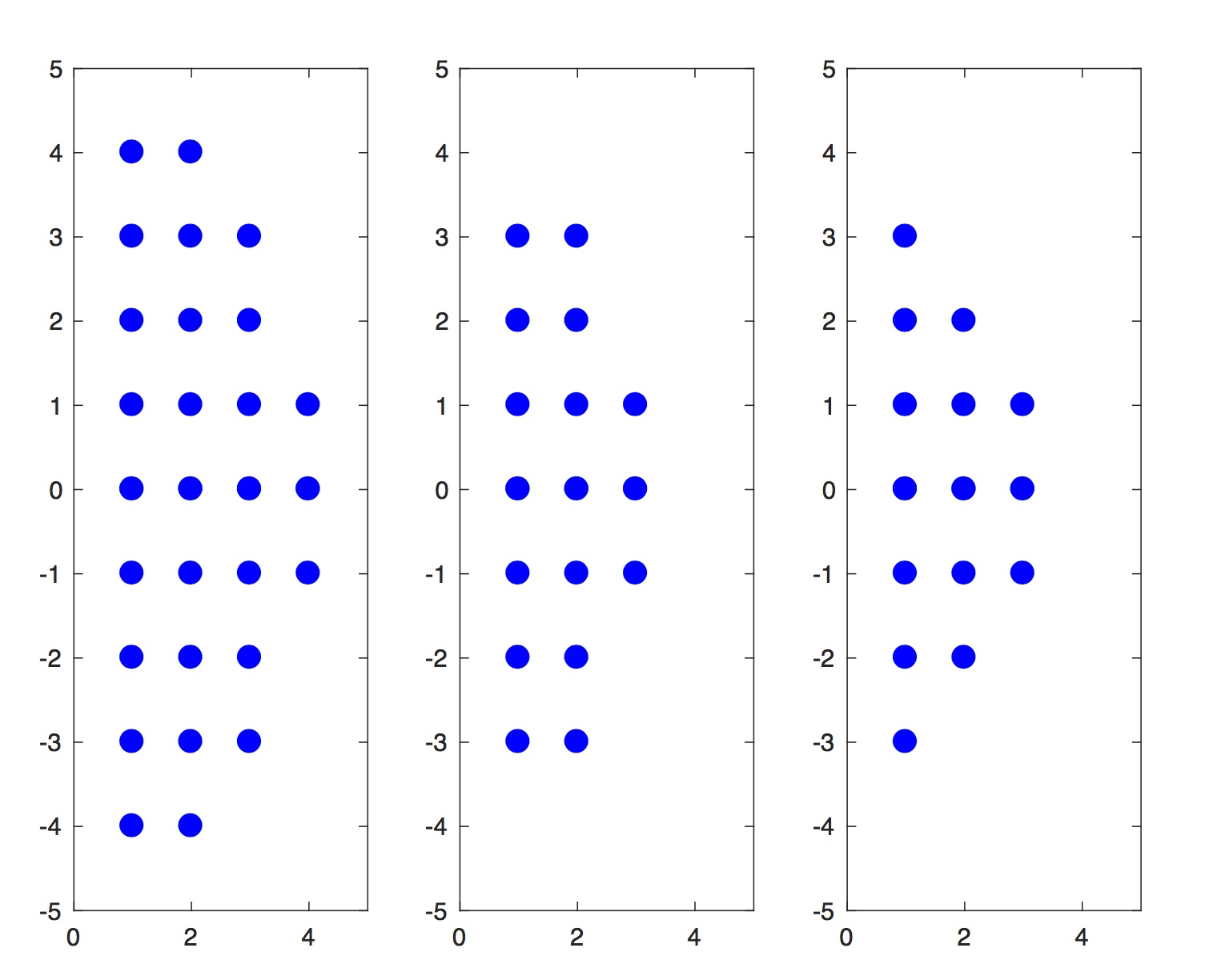}
\caption{The sets $\tilde{I}_{16},\tilde{I}_{10}$ and $\tilde{I}_{9}$.}
\end{figure}
Now we consider the finite conformal iterated function system 
$$\mathcal{F}_{16}=\{\f_e: e \in \tilde{I}_{16}\}.$$
Then
$$Z_2(\tilde{I}_{16}, t)=\sum_{\om \in \tilde{I}_{16}^2} \|\f'_\om\|^t_\infty.$$
We now recall the following recursive formula for $\|\f'_\om\|_\infty, \om \in E^n$, appearing in \cite[Theorem 6.6]{MU1}. For any $\om=(\om_1,\dots,\om_n) \in E^n$, set 
$$q_0(\om)=1, q_1(\om)=\om_1,$$
and
$$q_k(\om)=\om_k q_{k-1}(\om)+q_{k-2}(\om) \mbox{ for }2 \leq k \leq n.$$
Then
\begin{equation}
\label{recurder}
\|\f'_{\om}\|_\infty=\frac{4}{|q_n(\om)|^2 \left(\left|2+\frac{q_{n-1}(\om)}{q_n(\om)}\right| +\left|\frac{q_{n-1}(\om)}{q_n(\om)} \right|\right)^2}.
\end{equation}
Using the previous recursive formula and employing \texttt{Matlab} we can estimate $Z_2(\tilde{I}_{16}, t)$ for different values of $t$. In particular we obtain that
$$Z_2(\tilde{I}_{16}, 1.544)<0.997.$$
Recalling that $P_{\tilde{I}_{16}}(t)=\inf_{n \in \N} \frac{\log Z_n(\tilde{I}_{16},t)}{n}$ we deduce that 
\begin{equation}
\label{pres23}
P_{\tilde{I}_{16}}(1.544)<0.
\end{equation}
As we remarked earlier $\tilde{I}_{16}$ is an initial segment of $E$ (in particular $\tilde{I}_{16}=I(28)$), therefore \eqref{firstintsum}, \eqref{pres23} and Proposition \ref{keypropoint} imply that
\begin{equation}
\label{firstint}
[1.544, h] \subset DS (\cf).
\end{equation}

Now let $t_1:=1.545$. As before we use \texttt{Matlab} to calculate the quantity $\tilde{f}(1,t_1)$ and using the recursive formula \eqref{recurs} we derive approximations $\tilde{f}(k,t_1) \leq f(k_1,t)$ for all $k \in \N$ such that $e_k \in B$. In particular we obtain that
\begin{equation*}
f(k,t_1) \geq 16^{t_1} \mbox{ for all }e_k \in B \stm \tilde{I}_{10}.
\end{equation*}
Hence by Lemma \ref{lmc1fs9} and Proposition \ref{outofnumbox} we deduce that 
\begin{equation}
\label{firstintsum}
\frac{\sum_{n=k+1}^{\infty} \|\f_{e_n}'\|^{t_1}_\infty}{\|\f_{e_k}'\|^{t_1}_\infty} \geq 16^{t_1} \mbox{ for all }e_k \in E \stm \tilde{I}_{10}.
\end{equation}
We consider the finite conformal iterated function system 
$$\mathcal{F}_{10}=\{\f_e: e \in \tilde{I}_{10}\}.$$
Using \texttt{Matlab} as before
$$Z_2(\tilde{I}_{10}, 1.467)<0.989.$$
Recalling that $P_{\tilde{I}_{10}}(t)=\inf_{n \in \N} \frac{\log Z_n(\tilde{I}_{10},t)}{n}$ we deduce that 
\begin{equation}
\label{pres23}
P_{\tilde{I}_{10}}(1.467)<0.
\end{equation}
As we remarked earlier $\tilde{I}_{10}$ is an initial segment of $E$ (in particular $\tilde{I}_{10}=I(17)$), therefore \eqref{firstintsum}, \eqref{pres23} and Proposition \ref{keypropoint} imply that
\begin{equation}
\label{firstint2}
[1.467, 1.545] \subset DS (\cf).
\end{equation}
By direct computation  we  check that $D_2=\{1, 1+i\}$ therefore $$\tilde{I}_2=\{1, 1+i, 1-i\}.$$
We let $E_2:=E \stm \tilde{I}_2$ and we denote the conformal iterated function system associated to the complex continued fractions with entries in $E_2$ by
$$\cS_2=\{\f_e\}_{e \in E_2}.$$
We endow $E_2$ with the natural order inherited from $E$. Using a modification of the code developed by Priyadarshi \cite{pri},  for $E_2$ we obtain a lower bound for the Hausdorff  dimension of $J_{\cS_2}$:
\begin{equation}
\label{prirem}
\dim_{\cH}(J_{\cS_2}) \geq 1.5.
\end{equation}

Let $t_2=1.5$. Arguing as previously, using \texttt{Matlab} to estimate $\tilde{f}(1,t_2)$, the recursive formula \eqref{recurs} and the approximation sum \eqref{apr38}, we get that
\begin{equation*}
f(k,t_2) \geq 16^{t_2} \mbox{ for all }e_k \in B \stm \tilde{I}_{9},
\end{equation*}
which combined with Proposition \ref{outofnumbox} implies that
\begin{equation}
\label{secintsum}
\frac{\sum_{n=k+1}^{\infty} \|\f_{e_n}'\|^{t_2}_\infty}{\|\f_{e_k}'\|^{t_2}_\infty} \geq 16^{t_2} \mbox{ for all }e_k \in E \stm \tilde{I}_{9}. 
\end{equation}

We consider the finite conformal iterated function system 
$$\mathcal{F}^2_{9}=\{\f_e: e \in \tilde{I}_{9}\stm \tilde{I}_2\}.$$
Using \texttt{Matlab} as before
$$Z_2(\tilde{I}_{9} \stm \tilde{I}_{2} , 1.11)<0.97,$$
hence 
\begin{equation}
\label{pres16}
P_{\tilde{I}_{9} \stm \tilde{I}_{2}}(1.11)<0.
\end{equation}
Notice that \eqref{secintsum} implies that
\begin{equation}
\label{secintsum2}
\frac{\sum_{n=k+1}^{\infty} \|\f_{e_n}'\|^{t_2}_\infty}{\|\f_{e_k}'\|^{t_2}_\infty} \geq 16^{t_2} \mbox{ for all }e_k \in E_2 \stm \tilde{I}_{9}. 
\end{equation}
On the other hand $E_2 \cap \tilde{I}_{9}=\tilde{I}_{9} \stm \tilde{I}_{2}$ is an initial block of $E_2$, therefore \eqref{secintsum2}, \eqref{pres16} and Proposition \ref{keypropoint} imply that
\begin{equation}
\label{secint}
[1.11, 1.5] \subset DS(\cS_2) \subset DS (\cf).
\end{equation}

Finally let $t_3=1.2$. Using \texttt{Matlab} we estimate $\tilde{f}(1,t_3)$ and  the recursive formula \eqref{recurs} allows us to derive approximations $\tilde{f}(k,t_1) \leq f(k_1,t)$ for all $k \in \N$ such that $e_k \in B$. In particular we obtain that 
\begin{equation*}
f(k,t_3) \geq 16^{t_3} \mbox{ for all }e_k \in B \stm \tilde{I}_{2},
\end{equation*}
which combined with Proposition \ref{outofnumbox} implies that
\begin{equation}
\label{thirdintsum}
\frac{\sum_{n=k+1}^{\infty} \|\f_{e_n}'\|^{t_2}_\infty}{\|\f_{e_k}'\|^{t_2}_\infty} \geq 16^{t_2} \mbox{ for all }e_k \in E \stm \tilde{I}_{2}. 
\end{equation}
We consider the finite conformal iterated function system 
$$\mathcal{F}_{2}=\{\f_e: e \in \tilde{I}_{2}\}.$$
Using \texttt{Matlab} as before we obtain that $Z_2(\tilde{I}_{2}, 0.9)<0.93,$ hence we deduce that 
\begin{equation}
\label{pres23n}
P_{\tilde{I}_{2}}(0,9)<0.
\end{equation}
Since $\tilde{I}_2=I_3$, \eqref{thirdintsum}, \eqref{pres23n} and Proposition \ref{keypropoint} imply that
\begin{equation}
\label{thirdint}
[0.9, 1.12] \subset DS (\cf).
\end{equation}

By \cite[Theorem 6.2]{MU2} and the fact that $\theta(\cf)=1$ we deduce that
\begin{equation}
\label{oldint}
[0, 1) \subset DS(\cf).
\end{equation}
Now the theorem follows by combining \eqref{firstint}, \eqref{firstint2}, \eqref{secint}, \eqref{thirdint} and \eqref{oldint}.
\end{proof}

\begin{remark}
\label{intervarith}
In order to obtain the estimate \eqref{prirem} we used a modified version of a code developed by Priyadarshi \cite{pri}. The modified code guarantees the lower bound $1.5$ for the Hausdorff  dimension of $J_{\cS_2}$ as long as the the spectral radius of the restricted Perron-Frobenius operator is greater than one.  Already for the value of $p=7$ in the modified Priyadarshi's code, we obtain the spectral radius of at least $1.14082450939...$. This value is validated by the Matlab Toolbox for Reliable Computing, \texttt{INTLAB} \cite{Rump}, to be accurate to twelve digits (see also \cite{Rump_acta}).  We also verified Priyadarshi's estimate \eqref{prifull}. Actually using the Power Method (Power Algorithm) and  Priyadarshi's approach we were able to obtain $h\geq 1.84$. The \texttt{Matlab} codes used in the proof of Theorem \ref{full7} and for obtaining lower bounds for Hausdorff dimensions of subsystems of $\cf$ using the Power Method are available upon request.  
\end{remark}

\begin{remark} Recalling the definition of $\cf=\{\f_e\}_{e \in E}$ in the beginning of the section, it is easy to see that $\f_e (X)$ is a closed ball for every $e \in E$, where $X=\bar{B}(1/2,1/2)$. A geometric argument very similar to the one used in the proof of \eqref{1fs10} gives that 
$$r_e:=\diam (\f_e(X))=\frac{1}{m^2+n^2+m},$$
for $e=m+ni \in E$. Using this observation we define the {\em linearized continued fractions} IFS $\widehat{\cf}=\{\psi_e\}_{e \in E}$, where
$$\psi_e(z)=r_ez+\tau_e,$$
and $\tau_e$ is the center of $\f_e(X)$. Notice that $\widehat{\cf}$ consists of rotation-free similarities. Moreover
$$\psi_e(X)=\f_e (X) \mbox{ and }\|\psi_e'\|_\infty=r_e$$
for every $e \in E$. 

We now briefly describe how the method developed in the last two sections yields that $\widehat{\cf}$ is of full spectrum as well. We endow $E$ with a natural order as before, i.e. $e_1 \prec e_2$ if and only if $\|\psi_{e_1}'\|_\infty \geq \|\psi_{e_2}'\|_\infty$. Arguing as in Proposition \ref{outofnumbox}, although using much simpler elementary arguments, we deduce that
\begin{equation}
\label{condoutbox2}
\sum_{j \succ e^{\ast}}^\infty \| \psi'_j\|^2_\infty \geq  \| \psi'_e\|^2_\infty
\end{equation}
for all $e=m+ni \in E$ such that $m \geq 8$ or $|n| \geq 9$. Notice that since the maps $\psi_e$ are similarities the distortion constant $K$ is equal to $1$. We then use recursive approximation as in \eqref{apr38} and \texttt{Matlab} to verify that \eqref{condoutbox} also holds for the remaining $e \in E \cap ([1,8] \times [-9,9])$. Finally applying Corollary \ref{c2fs6} and Lemma \ref{lmc1fs9} we deduce that $\widehat{\cf}$ is of full spectrum. We stress that the main reason why the proof is so simpler in the case of linearized complex continued fractions is the fact that the distortion constant is the smallest possible, i.e. equal to $1$.
\end{remark}



\begin{bibdiv}
\begin{biblist}

\bib{atnip}{article}{
 AUTHOR = {Atnip, J.},
       TITLE = {Non-autonomous conformal graph directed Markov systems},
   JOURNAL = {ArXiv 1706.09978}
     }

 \bib{BLU}{book}{
	Author = {Bonfiglioli, A.}, 
	Author={Lanconelli, E.},
	Author={Uguzzoni, F.},
	Publisher = {Springer},
	Series = {Springer Monographs in Mathematics},
	Title = {Stratified {L}ie groups and potential theory for their sub-{L}aplacians},
	Year = {2007}}
	
	\bib{bowen}{article}{
     Author = {Bowen, R.},
	Coden = {PMIHA6},
	Fjournal = {Institut des Hautes \'Etudes Scientifiques. Publications Math\'ematiques},
	Issn = {0073-8301},
	Journal = {Inst.\ Hautes \'Etudes Sci.\ Publ.\ Math.},
	Mrclass = {57S05 (30C20 32G15 58F11)},
	Mrnumber = {556580},
	Mrreviewer = {L. A. Bunimovich},
	Number = {50},
	Pages = {11--25},
	Title = {Hausdorff dimension of quasicircles},
	Url = {http://www.numdam.org/item?id=PMIHES_1979__50__11_0},
	Year = {1979},
	Bdsk-Url-1 = {http://www.numdam.org/item?id=PMIHES_1979__50__11_0}}
	
	\bib{CDPT}{book}{
	Address = {Basel},
	Author={Capogna, L.},
	Author = {Danielli, D.},
	Author= {Pauls, S.},
	Author={Tyson, J.},
	Publisher = {Birkh\"auser Verlag},
	Series = {Progress in Mathematics},
	Title = {An introduction to the {H}eisenberg group and the sub-{R}iemannian isoperimetric problem},
	Volume = {259},
	Year = {2007}}
	

	
	
       \bib{CTU}{article}{
 AUTHOR = {Chousionis, V.},
   AUTHOR = {Tyson, J. T.},
      AUTHOR = {Urba\'nski, M.}
       TITLE = {Conformal graph directed Markov systems on Carnot groups},
   JOURNAL = {Memoirs of the AMS}
     }
     
     	\bib{cu}{article}{
    AUTHOR = {Cusick, T. W.},
     TITLE = {Hausdorff dimension of sets of continued fractions},
   JOURNAL = {Quart. J. Math. Oxford Ser. (2)},
  FJOURNAL = {The Quarterly Journal of Mathematics. Oxford. Second Series},
    VOLUME = {41},
      YEAR = {1990},
    NUMBER = {163},
     PAGES = {277--286},
      ISSN = {0033-5606},
   MRCLASS = {11K55 (11K50)},
  MRNUMBER = {1067484},
MRREVIEWER = {G. Ramharter},
       URL = {https://doi.org/10.1093/qmath/41.3.277},
}
     
     \bib{daj}{article}{
     AUTHOR = {Dajani, K.},
     AUTHOR = {Hensley, D.}, 
     AUTHOR = {Kraaikamp, C.},
     AUTHOR = {Masarotto, V.},
     TITLE = {Arithmetic and ergodic properties of `flipped' continued
              fraction algorithms},
   JOURNAL = {Acta Arith.},
  FJOURNAL = {Acta Arithmetica},
    VOLUME = {153},
      YEAR = {2012},
    NUMBER = {1},
     PAGES = {51--79},
      ISSN = {0065-1036},
   MRCLASS = {11K50 (37A45)},
  MRNUMBER = {2899816},
MRREVIEWER = {Alan Haynes},
       URL = {https://doi.org/10.4064/aa153-1-4},
}
     
      \bib{edm}{article}{
		Author = {Edgar, G. A.},
		Author={Mauldin, D.}	
	Title = {Multifractal decompositions of digraph recursive fractals},
	Journal = {Proc. Lond. Math. Soc.},
	Number = {65},
	Pages = {604--628},
	Volume = {3},
	Year = {1992}
	}
	


   
     \bib{nuss}{article}{
		Author = {Falk, R. S.},
  AUTHOR = {Nussbaum, R. D.},	
	Title = {A New Approach to Numerical Computation of Hausdorff Dimension of Iterated Function Systems: Applications to Complex Continued Fractions},
	Journal = {ArXiv 1612.00869},
	}
	
	\bib{nusscf}{article}{
		Author = {Falk, R. S.},
  AUTHOR = {Nussbaum, R. D.},	
	Title = {$C^m$ eigenfunctions of Perron-Frobenius operators and a new approach to numerical computation of Hausdorff dimension: applications in $\R$},
	Journal = {ArXiv 1612.00870},
	}
	
		\bib{good}{article}{
	AUTHOR = {Good, I. J.},
     TITLE = {The fractional dimensional theory of continued fractions},
   JOURNAL = {Proc. Cambridge Philos. Soc.},
    VOLUME = {37},
      YEAR = {1941},
     PAGES = {199--228},
   MRCLASS = {27.2X},
  MRNUMBER = {0004878},
MRREVIEWER = {P. Erd\"os},
}

\bib{gmr}{article}{
AUTHOR = {Ghenciu, A.},
AUTHOR = {Munday, S.},
AUTHOR = {Roy, M.},
     TITLE = {The {H}ausdorff dimension spectrum of conformal graph directed
              {M}arkov systems and applications to nearest integer continued
              fractions},
   JOURNAL = {J. Number Theory},
  FJOURNAL = {Journal of Number Theory},
    VOLUME = {175},
      YEAR = {2017},
     PAGES = {223--249},
      ISSN = {0022-314X},
   MRCLASS = {11J70 (11K50)},
  MRNUMBER = {3608189},
MRREVIEWER = {Daniel Garrett Glasscock},
       URL = {https://doi.org/10.1016/j.jnt.2016.09.002},
}
     
     \bib{HU}{article}{
 AUTHOR = {Heinemann, S.},
      AUTHOR = {Urba\'nski, M.}
       TITLE = {Hausdorff Dimension Estimates for Infinite Conformal Iterated Function Systems},
   JOURNAL = {Nonlinearity}
   Number = {15},
	Pages = {727--734},
	Year = {2002}
     }
     


\bib{hen1}{article}{
    AUTHOR = {Hensley, D.},
     TITLE = {Continued fraction {C}antor sets, {H}ausdorff dimension, and
              functional analysis},
   JOURNAL = {J. Number Theory},
  FJOURNAL = {Journal of Number Theory},
    VOLUME = {40},
      YEAR = {1992},
    NUMBER = {3},
     PAGES = {336--358},
      ISSN = {0022-314X},
   MRCLASS = {11K50 (28A78 46E99 47B38)},
  MRNUMBER = {1154044},
MRREVIEWER = {F. Schweiger},
       URL = {https://doi.org/10.1016/0022-314X(92)90006-B},
}

 \bib{henbook}{book}{
    AUTHOR = {Hensley, D.},
     TITLE = {Continued fractions},
 PUBLISHER = {World Scientific Publishing Co. Pte. Ltd., Hackensack, NJ},
      YEAR = {2006},
     PAGES = {xiv+245},
      ISBN = {981-256-477-2},
   MRCLASS = {11A55 (11J70 28D05 30B70 37A45 37F10 40A15)},
  MRNUMBER = {2351741},
MRREVIEWER = {Oto Strauch},
       URL = {https://doi.org/10.1142/9789812774682},
}

\bib{henrev}{article}{
    AUTHOR = {Hensley, D.},
     TITLE = {Continued fractions, {C}antor sets, {H}ausdorff dimension, and
              transfer operators and their analytic extension},
   JOURNAL = {Discrete Contin. Dyn. Syst.},
  FJOURNAL = {Discrete and Continuous Dynamical Systems. Series A},
    VOLUME = {32},
      YEAR = {2012},
    NUMBER = {7},
     PAGES = {2417--2436},
      ISSN = {1078-0947},
   MRCLASS = {11K55 (11J70 11K50 11Y60 37A45)},
  MRNUMBER = {2900553},
MRREVIEWER = {Radhakrishnan Nair},
       URL = {https://doi.org/10.3934/dcds.2012.32.2417},
}
     
     \bib{ahur}{article}{
      AUTHOR = {Hurwitz, A.} 
      TITLE = {\"Uber die Entwicklung complexer Gr\"ossen in Kettenbr\"uche.},
       JOURNAL = {Acta Math.}, 
       Number = {XI}
       Pages = {187--200},
     Year = {1888}
     }
     
     \bib{jhur}{article}{
      AUTHOR = {Hurwitz, J.} 
      TITLE = {\"Uber eine besondere Art der Kettenbruch- Entwicklung complexer Gr\"ossen},
       JOURNAL = {Dissertation, University of Halle}, 
     Year = {1895}
     }

     
     \bib{jp1}{article}{
		Author = {Jenkinson, O.},
  AUTHOR = {Pollicott, M.},	
	Title = {Computing the dimension of dynamically defined sets : $E_2$ and bounded
continued fractions},
	Journal = {Ergod. Th. and Dynam. Sys.},
	Number = {21},
	Pages = {1429--1445},
	Year = {2001}
	}

     
     \bib{jp2}{article}{
		Author = {Jenkinson, O.},
  AUTHOR = {Pollicott, M.},	
	Title = {Rigorous effective bounds on the Hausdorff dimension of continued fraction Cantor sets: a hundred decimal digits for the dimension of $E_2$},
	Journal = {Adv. Math, to appear},
	}

     
        \bib{KZ}{article}{
	Author = {Kesseb\"ohmer, M.},
  AUTHOR = {Zhu, S.},
	Journal = {J. Number Theory},
	Number = {116},
	Pages = {230--246},
	Title = {Dimension sets for infinite IFSs: the Texan conjecture.},
	Volume = {1},
	Year = {2006}}
	
	 \bib{kes1}{article}{
	 AUTHOR = {Kesseb\"ohmer, M.}
	 Author = {Kombrink, S.},
     TITLE = {Minkowski content and fractal {E}uler characteristic for
              conformal graph directed systems},
   JOURNAL = {J. Fractal Geom.},
  FJOURNAL = {Journal of Fractal Geometry. Mathematics of Fractals and
              Related Topics},
    VOLUME = {2},
      YEAR = {2015},
    NUMBER = {2},
     PAGES = {171--227},
      ISSN = {2308-1309},
   MRCLASS = {28A80 (28A75 60K05)},
  MRNUMBER = {3353091},
MRREVIEWER = {Steffen Winter},
       URL = {https://doi.org/10.4171/JFG/19},
}

 \bib{kes2}{article}{
	 AUTHOR = {Kesseb\"ohmer, M.}
	 Author = {Kombrink, S.},
    TITLE = {Minkowski measurability of infinite conformal graph directed systems and application to Apollonian packings}
	Journal = {ArXiv  1702.02854},
	}

      
     \bib{MU1}{article}{
	Author = {Mauldin, D.},
  AUTHOR = {Urba\'nski, M.},
	Journal = {Proc.\ London Math.\ Soc.\ (3)},
	Number = {1},
	Pages = {105--154},
	Title = {Dimensions and measures in infinite iterated function systems},
	Volume = {73},
	Year = {1996}}
	
	\bib{MU2}{article}{
	Author = {Mauldin, D.},
  AUTHOR = {Urba\'nski, M.},
	Journal = {Trans. Amer. Math. Soc.},
	Pages = {4995--5025},
	Title = {Conformal iterated function systems with applications
to the geometry of conformal iterated function systems},
	Volume = {351},
	Year = {1999}}
	
	
	\bib{MUbook}{book}{
		Author = {Mauldin, D.},
  AUTHOR = {Urba\'nski, M.},
	Publisher = {Cambridge University Press},
	Series = {Cambridge Tracts in Mathematics},
	Title = {Graph directed {M}arkov systems: Geometry and dynamics of limit sets},
	Volume = {148},
	Year = {2003}}
	
	\bib{MW}{article}{
	Author = {Mauldin, D.},
  AUTHOR = {Williams, S. C.},
	Journal = {Trans. Amer. Math. Soc.},
	Pages = {811--829},
	Title = {Hausdorff dimension in graph directed constructions},
	Volume = {309},
	Year = {1988}}
	

	
	\bib{mcmullen}{article}{
		Author = {McMullen, C. T.},
	Title = {Hausdorff dimension and conformal dynamics. III. Computation of
dimension},
	Journal = {Amer. J. Math.},
	Volume = {120},
	Number = {4},
	Pages = {691--721},
	Year = {1998}}


  \bib{polur}{article}{
	Author = {Pollicott, M.},
  AUTHOR = {Urba\'nski, M.},
	Journal = {ArXiv 1704.06896},
	Title = {Asymptotic Counting in Conformal Dynamical Systems},
	}
	

 \bib{pri}{article}{
		Author = {Priyadarshi, A.},
	Title = {Lower bound on the Hausdorff dimension of a set of complex continued fractions},
	Journal = {J. Math. Anal. Appl.},
	Volume = {449},
	Number = {1},
	Pages = {91--95},
	Year = {2017}}
	
	 \bib{roy}{article}{
	  AUTHOR = {Roy, M.},
     TITLE = {A new variation of {B}owen's formula for graph directed
              {M}arkov systems},
   JOURNAL = {Discrete Contin. Dyn. Syst.},
  FJOURNAL = {Discrete and Continuous Dynamical Systems. Series A},
    VOLUME = {32},
      YEAR = {2012},
    NUMBER = {7},
     PAGES = {2533--2551},
      ISSN = {1078-0947},
   MRCLASS = {37D35},
  MRNUMBER = {2900559},
MRREVIEWER = {Heber Enrich},
       URL = {https://doi.org/10.3934/dcds.2012.32.2533},
}

	
	\bib{Rump}{incollection}{
   Author = {Rump, {S. M.}},
   Title = {{INTLAB - INTerval LABoratory}},
   Editor = {Tibor Csendes},
   Booktitle = {{Developments~in~Reliable Computing}},
   Publisher = {Kluwer Academic Publishers},
   Address = {Dordrecht},
   Pages = {77--104},
   Year = {1999},
   url = {http://www.ti3.tuhh.de/rump/}}
	
     
     \bib{Rump_acta}{article}{
    AUTHOR = {Rump, S. M.},
     TITLE = {Verification methods: rigorous results using floating-point
              arithmetic},
   JOURNAL = {Acta Numer.},
  FJOURNAL = {Acta Numerica},
    VOLUME = {19},
      YEAR = {2010},
     PAGES = {287--449},
      ISSN = {0962-4929},
   MRCLASS = {65G20 (03B35 68T05)},
  MRNUMBER = {2652784},
MRREVIEWER = {Ljiljana Petkovi\'c},
       URL = {https://doi.org/10.1017/S096249291000005X},
}


 \bib{schm}{article}{
		Author = {Schmidt, A.},
	Title = {Diophantine approximation of complex numbers},
	Journal = {Acta Math.},
	Volume = {134},
	Pages = {1-85},
	Year = {1975}}


     
     \end{biblist}
\end{bibdiv}
\end{document}